\documentclass[a4paper,10pt]{amsart}
\usepackage[utf8]{inputenc}
\usepackage{amssymb}
\usepackage{amsmath}
\usepackage{bbm}
\usepackage{mathrsfs}
\usepackage[all]{xy}
\usepackage{bm}
\usepackage{graphicx}
\usepackage{enumitem}
\usepackage[left=2cm,right=2cm,top=2cm,bottom=2cm]{geometry}
\usepackage{tikz}
\usepackage[backref=page]{hyperref}
\title[Analytic functors on $\gr\op$]{On analytic contravariant functors on free groups}
\date{}
\author{Geoffrey Powell}
\address{Univ Angers, CNRS, LAREMA, SFR MATHSTIC, F-49000 Angers, France}
\email{Geoffrey.Powell@math.cnrs.fr}
\urladdr{https://math.univ-angers.fr/~powell/}
\thanks{This work was partially supported by the ANR Project {\em ChroK}, {\tt ANR-16-CE40-0003}.}
\newtheorem{THM}{Theorem}
\newtheorem{COR}[THM]{Corollary}
\newtheorem{thm}{Theorem}[section]
\newtheorem{prop}[thm]{Proposition}
\newtheorem{cor}[thm]{Corollary}
\newtheorem{lem}[thm]{Lemma}
\theoremstyle{definition}
\newtheorem{defn}[thm]{Definition}
\newtheorem{exam}[thm]{Example}
\theoremstyle{remark}
\newtheorem{rem}[thm]{Remark}
\newtheorem{nota}[thm]{Notation}
\newtheorem{hyp}[thm]{Hypothesis}

\renewcommand{\phi}{\varphi}
\renewcommand{\hom}{\mathrm{Hom}}
\newcommand{\sym}{\mathfrak{S}}
\newcommand{\gr}{\mathbf{gr}}
\newcommand{\fs}{{\boldsymbol{\Omega}}}
\newcommand{\kmod}{\mathtt{Mod}_\kring}
\newcommand{\cre}{\mathrm{cr}}
\newcommand{\f}{\mathcal{F}}
\newcommand{\cala}{\mathscr{A}}
\newcommand{\calb}{\mathscr{B}}
\newcommand{\calc}{\mathcal{C}}
\newcommand{\fcatk}[1][\calc]{\f (#1; \kring)}
\newcommand{\fpoly}[2]{\f_{#1}(#2; \kring)}
\newcommand{\nat}{\mathbb{N}}
\newcommand{\ab}{\mathbf{ab}}
\newcommand{\zed}{\mathbb{Z}}
\newcommand{\A}{\mathfrak{a}}
\newcommand{\ext}{\mathrm{Ext}}
\newcommand{\op}{^\mathrm{op}}
\newcommand{\ob}{\mathrm{Ob}\hspace{2pt}}
\newcommand{\kring}{\mathbbm{k}}
\newcommand{\dash}{\hspace{-2pt}-\hspace{-2pt}}
\newcommand{\modules}{\mathrm{mod}}
\newcommand{\fb}{{\bm{\Sigma}}}
\newcommand{\fart}{\f_{< \infty} ^{\mathrm{fin}}(\gr ; \kring)}
\newcommand{\fartop}{\f_{< \infty} ^{\mathrm{fin}}(\gr\op ; \kring)}
\newcommand{\fpi}[1]{\fpoly{<\infty}{#1}}
\newcommand{\opd}{\mathscr{O}}
\newcommand{\fopd}{\f_\opd}
\newcommand{\popd}{P^\opd}
\newcommand{\com}{\mathfrak{Com}}
\newcommand{\ucom}{\com^u}
\newcommand{\lie}{\mathfrak{Lie}}
\newcommand{\ass}{\mathfrak{Ass}}
\newcommand{\uass}{\mathfrak{Ass}^u}
\newcommand{\as}{{\mbox \textexclamdown}}
\newcommand{\pbb}{\mathbb{P}}
\newcommand{\bimodunit}{\mathbbm{1}}
\newcommand{\conv}{\odot}
\newcommand{\tc}{\mathbb{T}_\conv}
\newcommand{\fn}{^\mathrm{fin}}
\newcommand{\kgmod}{\mathtt{gMod}_\kring}
\newcommand{\smod}{\fb\dash\kgmod}
\newcommand{\smodug}{\fb\dash\kmod}
\newcommand{\sop}{\fb\op\dash\kgmod}
\newcommand{\sopug}{\fb\op\dash\kmod}
\newcommand{\sbimod}{\fb\dash\mathtt{BiMod}_\kring}
\newcommand{\cat}{\mathsf{Cat}\ }
\newcommand{\catlie}{\mathsf{Cat}\lie}
\newcommand{\os}{\mathscr{S}}
\newcommand{\fcom}{\f_{\com}}
\newcommand{\flie}{\f_{\lie}}
\newcommand{\fppd}{\f_{\mathscr{P}}}
\newcommand{\tre}{\tilde{\mathrm{tr}}}
\newcommand{\tr}{\overline{\mathrm{tr}}}
\newcommand{\pgrop}{\mathsf{p}}
\newcommand{\alg}{\mathrm{Alg}}
\newcommand{\scom}{\mathfrak{C}}
\newcommand{\bu}{\beta^\sharp}
\newcommand{\g}{\mathfrak{g}}
\newcommand{\ppd}{\mathscr{P}}
\newcommand{\gropcatuass}{{}_{\Delta} \cat \uass}
\newcommand{\id}{\mathrm{Id}}
\newcommand{\ord}{\mathsf{Ord}}
\newcommand{\fg}{\mathsf{F}}
\newcommand{\ao}{\amalg_<}
\newcommand{\finset}{\mathbf{Fin}}
\numberwithin{equation}{section}
\begin{document}

\begin{abstract}
Working over a field $\kring$ of characteristic zero, the category $\f_\omega (\gr\op; \kring)$ of analytic contravariant functors on the category $\gr$ of finitely-generated free groups is shown to be equivalent to the category $\flie$ of representations of the $\kring$-linear category $\cat \lie$ associated to the Lie operad. 

Two proofs are given of this result. The first uses the original Ginzburg-Kapranov approach to Koszul duality of binary quadratic operads and the fact that the category of analytic contravariant functors is Koszul.  

The second proof proceeds by making the equivalence explicit using the $\kring$-linear category $\cat \uass$ associated to the operad $\uass$ encoding unital associative algebras, which provides the `twisting bimodule' between modules over $\cat \lie$ and modules over $\kring \gr\op$. A key ingredient is the Poincaré-Birkhoff-Witt theorem. 

Using the explicit formulation, it is shown how this equivalence reflects the tensor product on the category of analytic contravariant functors, relating this to the convolution product for representations of $\cat \lie$.
 \end{abstract}

\maketitle

\section{Introduction}

Functors from the category $\gr$ of finitely-generated free groups to abelian groups arise in nature. 
For instance, any functor from pointed topological spaces to abelian groups yields  a functor on $\gr$, via precomposition with the classifying space functor $G \mapsto BG$; the example of higher Hochschild homology was studied in \cite{PV} inspired in part by the work of Turchin and Willwacher \cite{MR3982870}. This underlines the interest of having a good understanding of functors on $\gr$ and, in particular, motivated this project.

Throughout,  $\kring$ is taken to be a field of characteristic zero. The focus is upon the structure of the category $\fcatk[\gr]$ of functors from $\gr$ to $\kring$-vector spaces, together with its contravariant counterpart, $\fcatk[\gr\op]$. The main purpose of the paper is to give a more linear description of these categories; this requires placing restrictions on the functors considered. 

In the contravariant case, we work with the full subcategory $\f_\omega (\gr\op; \kring)$ of analytic functors. Analytic functors are defined using the notion of polynomial functor (cf. \cite{MR3340364}), generalizing that introduced by Eilenberg and MacLane for functors on additive categories \cite{MR65162}. For each $d \in \nat$, this gives the subcategory $\fpoly{d}{\gr\op}$ of polynomial functors of degree at most $d$. Every functor $F \in \ob \fcatk[\gr\op]$ admits a canonical filtration 
  \[
\pgrop_0 F 
\subset 
\pgrop_1 F
\subset 
\ldots 
\subset 
\pgrop_d F 
\subset
\pgrop_{d+1} F
\subset
\ldots 
\subset 
F
\]
where $\pgrop_d F$ is the largest degree $\leq d$ polynomial subfunctor of $F$. The functor $F$ is said to be analytic if  it is the colimit of its subfunctors $\pgrop_d F$.

Write $\cat \lie$ for the $\kring$-linear category underlying the PROP  associated to the Lie operad, $\lie$ (see Section \ref{sect:prop}). This  has set of  objects $\nat$ and, by construction, $\cat \lie (n,1) = \lie(n)$;  the endomorphism ring $\cat \lie (n,n)$ is the group ring $\kring [\sym_n]$ of the symmetric group on $n$ letters.  

The category $\flie$ is defined to be the category of representations of $\cat \lie$ or, equivalently, the category of left $\cat \lie$-modules (see Section \ref{sect:rep}). The main result is:

\begin{THM}
\label{THM:Morita}
(Theorem \ref{thm:analytic_grop_flie}.)
The categories  $\f_\omega (\gr\op; \kring)$ and $\flie$ are equivalent.
\end{THM}

There is a dual statement for  covariant functors; for this, we restrict to the category $\fart$ of finite polynomial functors (those admitting a finite composition series).  The result can be  expressed in terms of the category of right modules  over the operad $\lie$, defined with respect to the operadic composition product:

\begin{COR}
(Corollary \ref{cor:finite_poly_gr}.)
 The category  $\fart$ is equivalent to the category of finite right $\lie$-modules. 
 \end{COR}

The category $\flie$ is much simpler to work with than  $\f_\omega (\gr\op; \kring)$. For instance, $\cat \lie$ has the property that $\cat \lie (m,n)=0$ if $m<n$. Hence, if $M \in \ob \flie$ and $d \in \nat$, one has the subfunctor $M^{\leq d}$ defined by 
\[
M^{\leq d} (t) := \left\{
\begin{array}{ll}
M(t) & t \leq d \\
0 & \mbox{otherwise.}  
\end{array}
\right.
\]
This gives an exhaustive, increasing filtration $(M^{\leq d})_{d \in \nat}$ of $M$. 
Under the equivalence of Theorem \ref{THM:Morita}, this identifies with the polynomial filtration  of the  analytic functor  corresponding to $M$.  The simplicity of the definition of the filtration $(M^{\leq d})_{d \in \nat}$ is manifest; moreover, $M(d)$ is a $\sym_d$-module and the composition factors of this module are in bijective  correspondence with the composition factors of polynomial degree exactly $d$ of the corresponding analytic functor. 
 
Basic examples arise as follows. There is a forgetful functor $\flie \rightarrow \smodug$ to the category of  $\fb$-modules (functors from $\fb$ to $\kring$-vector spaces, where $\fb$ is the category of finite sets and bijections). This admits a section which fits into the commutative (up to natural isomorphism) diagram:
\begin{eqnarray}
\label{eqn:abelianization_flie}
\xymatrix{
\smodug 
\ar[d]_\cong 
\ar@{^(->}[r]
&
\flie 
\ar[d]^\cong
\\
\f_\omega(\ab\op; \kring) 
\ar@{^(->}[r]
&\f_\omega(\gr\op; \kring),
}
\end{eqnarray}
in which the right hand vertical equivalence is given by Theorem \ref{THM:Morita}. Here,   $\f_\omega (\ab\op; \kring)$ is the category of analytic functors on $\ab\op$, where $\ab$ is the category of finitely-generated free abelian groups; the faithful embedding $\f_\omega(\ab\op; \kring) \hookrightarrow 
\f_\omega(\gr\op; \kring)$ is induced by abelianization. For example, for $d \in \nat$, the regular representation ring $\kring [\sym_d]$ gives an object of $\flie$ which is supported on $d$. 
 Its image in $\f_\omega (\gr\op; \kring)$ is the tensor product $(\A^\sharp)^{\otimes d}$ where 
  $\A^\sharp$ is defined by $G \mapsto \hom_{\mathrm{Group}} (G, \kring)$.

The equivalence between $\f_\omega (\gr\op; \kring)$ and $\flie$ is described explicitly as follows. The morphism of operads $\lie \rightarrow \uass$ that encodes the commutator Lie algebra of a unital associative algebra  induces a $\kring$-linear functor
$$
\cat \lie \rightarrow \cat \uass,
$$
where $\cat \uass$ is the $\kring$-linear category underlying the PROP associated to the operad $\uass$ encoding unital associative algebras. In particular,   $\cat \uass$ can be considered as a  left $\cat \uass$, right $\cat \lie $ bimodule (see Section \ref{sect:modules} for this terminology). Hence one has the induction functor:
\begin{eqnarray*}
\cat \uass \otimes_{\cat \lie} - 
\end{eqnarray*}
from $\flie$ to left $\cat \uass$-modules. This generalizes the universal enveloping algebra functor from Lie algebras to unital associative algebras, as explained in Section \ref{sect:pbw}.

Moreover, this enriches to a functor with values in $\f_\omega (\gr\op; \kring)$, induced by using the `twisting bimodule' 
$
\gropcatuass,
$ 
which is $\cat \uass$ equipped with a left $\kring \gr\op$, right $\cat \lie$-module structure. 
 This twisting bimodule induces the equivalence of Theorem \ref{THM:Morita}:

\begin{THM}
(Theorem \ref{thm:morita_eq_cat_uass}.)
\label{THM:uass}
The equivalence of categories of Theorem \ref{THM:Morita} is induced by the functors
\begin{eqnarray*}
\hom_{\f_\omega (\gr\op; \kring ) } (\gropcatuass, -  ) 
\ &:& \  
\f_\omega (\gr\op; \kring ) 
\rightarrow 
\f_\lie
\\
\gropcatuass \otimes_{\cat \lie} - 
\ &:& \ 
\flie \rightarrow \f_\omega (\gr\op; \kring).
\end{eqnarray*}
\end{THM}

The functor $\hom_{\f_\omega (\gr\op; \kring ) } (\gropcatuass, -  ) $ is related to standard constructions arising in the study of polynomial functors. Namely, it is a `structured' version of the cross-effect functors, in addition taking into account the natural transformations between these (see Part \ref{part:gr}).

The analogous result for functors on $\gr$ is the following: 

\begin{THM}
(Theorem \ref{thm:covariant_concrete}.)
 The coinduction functor $\hom_{\modules_{\catlie}} (\gropcatuass, -)$ induces an equivalence of categories: 
\[
\hom_{\modules_{\catlie}} (\gropcatuass, -) 
:
\modules_{\catlie}\fn
\rightarrow 
\fart,
\] 
where $\modules_{\catlie}\fn$ is equivalent to the category of finite right $\lie$-modules.
 \end{THM}

These results  may be compared with those of \cite{MR3340364} where, for each $d \in \nat$, the authors provide a model for the category of polynomial functors of degree $d$ on $\gr$,  working over $\zed$. Once again, cross-effects are at the heart of the construction; however, their model requires working with a form of non-linear Mackey functors. The relationship between their approach and the one considered here merits further investigation.

The following result shows that Theorem \ref{THM:uass} can be considered as a far-reaching generalization of the relationship between cocommutative Hopf algebras and exponential functors on $\gr \op$ via the functor $\Phi(-)$ (see Notation \ref{nota:Phi} and the following Remark): 

\begin{THM}
\label{THM:lie_case}
(Theorem \ref{thm:lie_case}.)
For a Lie algebra $\g$, there is a natural isomorphism in $\f_\omega (\gr\op; \kring)$:
\[
\gropcatuass \otimes_{\cat \lie} \underline{\g}
\cong 
\Phi(U\g)
\]
where  $\underline{\g} \in \ob \flie$ is given by $\underline{\g}(n)= \g ^{\otimes n}$, with the action of $\cat \lie$ induced by the Lie algebra structure of $\g$. 
\end{THM} 

Two proofs are given of the equivalence between $\f_\omega (\gr\op; \kring)$ and $\flie$. The first exploits Koszul duality techniques and relies on the theory that is reviewed in Sections \ref{sect:kos_background} and \ref{sect:koszul}.  

The second is much more hands on and, notably, provides the explicit form of the equivalence that is useful for applications. Readers interested in this may prefer to dive straight into Section \ref{sect:pbw}, after having reviewed the background on functors on $\gr$ and on $\gr\op$ given in Section \ref{sect:gr}.

\subsection{The Koszul duality approach}

Theorem \ref{THM:Morita} is proved using the following `Koszul property' deduced from  \cite{V_ext}:
\begin{eqnarray}
\label{eqn:kos_property}
\ext^*_{\f_\omega (\gr\op; \kring)} ((\A^\sharp)^{\otimes m}, (\A^\sharp)^{\otimes n}) =
\left\{
\begin{array}{ll}
0 & * \neq m-n \\
\kring \fs (\mathbf{m}, \mathbf{n} ) & * = m-n ,
\end{array}
\right.
\end{eqnarray}
where $\fs$ is the category of finite sets and surjections (note that this equality only respects composition up to explicit signs).  

The category $\kring \fs$ is equivalent to the category $\cat \com$ associated to the commutative operad $\com$. Theorem \ref{THM:Morita} is deduced from  Ginzburg and Kapranov's treatment of Koszul duality \cite{MR1301191}, and reflects the fact that the operad $\lie$ is the Koszul dual of $\com$. (The relevant Koszul duality theory is reviewed in Section \ref{sect:koszul}.)

A useful biproduct of the Koszul duality  theory is that, for each $d \in \nat$, one obtains an explicit minimal projective resolution of the functor $(\A^\sharp)^{\otimes d}$ in the category $\f_\omega (\gr\op; \kring)$ (see Corollary \ref{cor:proj_Koszul_resolution}).
 Moreover, there is also the following Koszul dual side to this story. The category $\fcom$ (see Example \ref{exam:com_lie}) of representations of $\cat \com$  is equivalent to the category $\fcatk[\fs]$ of functors from $\fs$ to $\kring$-vector spaces; the Koszul duality theory provides explicit resolutions in these categories (see Example \ref{exam:res_com}). These are of interest in relation to the calculation of Pirashvili's higher Hochschild homology \cite{Phh}.

\subsection{The  Poincaré-Birkhoff-Witt theorem approach}

The second approach to proving the equivalence between $\f_\omega (\gr\op; \kring)$ and $\flie$ (as in Theorem \ref{THM:uass})  uses $\gropcatuass$ to give an explicit model for the projective generators of the category $\f_\omega (\gr \op; \kring)$; in particular, the two functors appearing in Theorem \ref{THM:uass} are  entirely explicit. Crucially, this also gives an explicit model for the natural transformations between these projective generators, which relies upon the Poincaré-Birkhoff-Witt theorem.

The proof relies on properties of the polynomial filtration, which depend only upon  (\ref{eqn:kos_property}) in cohomological degrees $* \in \{0, 1\}$; in particular, it is independent of the Koszul property for higher cohomological degree, thus giving an independent proof of the equivalence of Theorem \ref{THM:Morita}. Indeed, one can then {\em deduce} the above Koszul property (\ref{eqn:kos_property}) from this result. Moreover, this method of proof can be generalized to consider related functor categories where the Koszul property is not available.

The equivalence of Theorem \ref{THM:uass} also reflects further structure on $\f_\omega (\gr\op; \kring)$. For instance, the tensor product on $\fcatk[\gr\op]$ restricts to a tensor product on analytic functors (see Proposition \ref{prop:tensor_analytic}). There is a corresponding symmetric monoidal structure  $(\flie, \conv, \kring)$ (see Proposition  \ref{prop:sym_mon_fopd}) which is related to the tensor product:

\begin{THM}
(Theorem \ref{thm:compat_tensor_conv}.)
The functor $\gropcatuass\otimes_{\cat \lie}-$ is symmetric monoidal:
\[
\gropcatuass\otimes_{\cat \lie} - \ : \ 
(\flie, \conv, \kring) 
\rightarrow 
(\f_\omega (\gr\op; \kring), \otimes, \kring).
\]
\end{THM}

Further applications of Theorem \ref{THM:uass} will be given elsewhere; for example, see \cite{2022arXiv220113307P}.

\subsection{Notation}

Denote by
\begin{itemize}[label=]
\item 
$\hom_{\mathscr{C}} (x,y)$ or  $\mathscr{C} (x,y)$ the set of morphisms in a category $\mathscr{C}$ between objects $x, y$;
\item 
$\kring$ a field of characteristic zero;
\item
$\otimes$ denotes $\otimes_\kring$ unless otherwise indicated; 
\item 
$\kmod$ the category of $\kring$-vector spaces and $\kgmod$ the category of $\zed$-graded $\kring$-vector spaces;
\item 
$^\sharp$ the duality functor on $\kgmod$;
\item 
$\nat$ the non-negative integers;
\item 
$\mathbf{Fin}$ the category of finite sets;
\item 
$\fs$ the category of finite sets and surjections; 
\item 
$\fb$ the category of finite sets and bijections;
\item 
$\gr$ the category of finitely-generated free groups, considered as a full subcategory of that of groups; 
\item
$\ab$ the category of finitely-generated free abelian groups, as a full subcategory of the category of groups; 
\item 
$\mathbf{n}:= \{1, \ldots , n \}$, for $n \in \nat$;
\item 
$\sym_n$  the symmetric group given as automorphisms of the set $\mathbf{n}$, for $n\in \nat$;
\item 
$\fg_n$ the free group on the set $\mathbf{n}$, $n \in \nat$,  so that $\fg_1 \cong \zed$ and  $\fg_n \cong \fg_1^{\star n}$, where $\star$ denotes the free product of groups.
\end{itemize}

The category $\kmod$ is equipped with the usual symmetric monoidal structure given by $\otimes_\kring$, likewise  $\kgmod$ with respect to the graded tensor product and with symmetry invoking Koszul signs.  

\begin{rem}
The categories $\mathbf{Fin}$, $\fb$ and $\fs$ have skeleta with objects  $\mathbf{n}$, for $n \in \nat$.
\end{rem}

\subsection{Acknowledgement} This work was inspired by questions addressed in joint work of the author with Christine Vespa \cite{PV}; the author is grateful for her interest in this project and for her comments and suggestions. He also thanks an anonymous referee for their comments on an earlier version.

\tableofcontents

\part{Operads, categories and representations}
\label{part:opd}

\section{Recollections on modules over a category}
\label{sect:modules}

This section recalls some basic results on representations of small ($\kring$-linear) categories and also serves to fix some notation. 

A small $\kring$-linear category should be thought of as a $\kring$-algebra with several objects, in the spirit of Mitchell \cite{MR294454}, and a $\kring$-linear functor as a morphism between $\kring$-algebras with several objects. 

\begin{rem}
A $\kring$-algebra is a $\kring$-linear category with a single object. Conversely, given a small $\kring$-linear category $\cala$, one can form the associated category algebra, given by $\bigoplus_{x, y \in \ob \cala} \cala (x,y)$ and multiplication induced by composition. (A word of warning: if $\ob \cala$ is infinite, then this is not unital.)
\end{rem}

\subsection{The $\kring$-linear case}
\label{subsect:k-linear-modules}

Throughout this section, $\cala$ and $\calb$ denote small $\kring$-linear categories, so that $\cala \op$ is also a small $\kring$-linear category. The tensor product $\cala \otimes \calb$ is the $\kring$-linear category with objects $\ob \cala \times \ob \calb$ and morphisms $\hom_{\cala \otimes \calb} ((x,b), (y, c)) = \cala (x,y) \otimes \calb (b,c)$, with composition defined in the evident way, thus generalizing the tensor product of $\kring$-algebras.

\begin{defn}
\label{defn:left_cala_module}
The category of (left) $\cala$-modules (denoted ${_\cala}\modules$) is the category of $\kring$-linear functors from $\cala$ to $\kmod$; morphisms are natural transformations. The category of right $\cala$-modules (denoted $\modules_\cala$) is the category of $\cala \op$-modules.

The category of left $\cala$-, right $\calb$-bimodules is the category of $\cala \otimes \calb\op$-modules; taking $\calb = \cala$, the category of $\cala$-bimodules is the category of $\cala \otimes \cala\op$-modules.
\end{defn}

The category of $\cala$-modules inherits an abelian structure from $\kmod$.

\begin{rem}
\label{rem:category_left_cala_modules}
An $\cala$-module is equivalent to a set of objects $\{ M(x) \in \ob \kmod \ | \  x \in \ob \cala \}$ equipped with structure morphisms for $x , y \in \ob \cala$:
\[
\cala (x, y ) \otimes M(x) 
\rightarrow 
M(y)
\]
that satisfy the associativity and unital axioms. This formulation generalizes  to the case  where $\cala$ is enriched in graded $\kring$-vector spaces, replacing $\kmod$ by $\kgmod$ everywhere.
\end{rem}

\begin{exam}
For $w,z \in \ob \cala$, the composition of $\cala$ makes: 
\begin{enumerate}
\item 
$\cala (z, -)$ a left $\cala$-module; 
\item 
$\cala (-, w)$ a right $\cala$-module.
\end{enumerate}
These structures commute, hence $\cala (-,-)$ can be considered as an $\cala$-bimodule.
\end{exam}

As usual, left and right modules are related by vector space duality:

\begin{lem}
\label{lem:vs_duality}
Vector space duality induces an exact functor 
\[
(-)^\sharp : 
({_{\cala}}\modules)\op 
\rightarrow 
\modules_{\cala},
\] 
where, for $M \in \ob {_\cala}\modules$, $x\in \ob \cala$, $M^\sharp (x):= (M(x))^\sharp$. 
 This restricts to an equivalence of categories between the full subcategories of functors taking finite-dimensional values.
\end{lem}

Suppose that $f : \cala \rightarrow \calb$ is a  $\kring$-linear functor that is the identity on objects (this condition is only imposed here so as to simplify the presentation). Then $f$ induces a restriction functor from $\calb$-modules to $\cala$-modules. Explicitly, for $N$ a $\calb$-module and $x, y \in \ob \cala$, the structure morphism  is the composite:
\[
\cala (x,y) 
\otimes N(x) 
\stackrel{f \otimes \id}{\rightarrow }
\calb (x,y) \otimes N(x) 
\rightarrow 
N(y)
\]
where the second map is the $\calb$-module structure morphism.

\begin{exam}
Let $f : \cala \rightarrow \calb$ be as above. Then, for any $w \in \ob \calb$, $\calb (-, w)$ has the structure of a right $\cala$-module.
\end{exam}

There is an associated induction functor:

\begin{prop}
\label{prop:left_adjoint_restriction}
For  $f : \cala \rightarrow \calb$  a  $\kring$-linear functor that is the identity on objects, restriction 
${_{\calb}}\modules \rightarrow {_{\cala}}\modules$ has a left adjoint $\calb \otimes _{\cala} - \ : \ {_{\cala}}\modules \rightarrow {_{\calb}}\modules $. 

Explicitly, for an $\cala$-module $M$, the induced module  $\calb \otimes _{\cala} M $ is the coequalizer of
\[
\xymatrix{
\bigoplus _{x, y \in \ob \cala} 
\calb (y, -) 
\otimes 
\cala (x,y) 
\otimes 
M(x) 
\ar@<.5ex>[r]
\ar@<-.5ex>[r]
&
\bigoplus _{z \in \ob \cala} 
\calb (z, -) 
\otimes 
M(z), 
}
\]
where the maps are induced by the right $\cala$-module structure on $\calb$ and the  $\cala$-module structure on $M$.
\end{prop}

Likewise one has the coinduction functor; the following is stated for right modules:

\begin{prop}
\label{prop:coinduction}
For  $f : \cala \rightarrow \calb$  a  $\kring$-linear functor that is the identity on objects, restriction 
$\modules_{\calb} \rightarrow \modules_{\cala}$  has a right adjoint $\hom_{\modules_{\cala}} (\calb, -) $, where $\calb$ is considered as a left $\calb$-, right $\cala$-bimodule. 

Explicitly, for $N$ a right $\cala$-module and $x \in \ob \cala$, 
$$\hom_{\modules_\cala} (\calb, N)(x)= 
\hom_{\modules_\cala} (\calb (-, x) , N),$$
 with the right $\calb$-module structure on $\hom_{\modules_{\cala}} (\calb, N)$ induced by the left $\calb$-module structure on $\calb$.
\end{prop}

\begin{rem}
Analogously to the presentation of $\calb \otimes_\cala - $ via a coequalizer in Proposition \ref{prop:left_adjoint_restriction}, the right adjoint $\hom_{\modules_{\cala}} (\calb, -) $ can be described as an equalizer. Namely, for $x \in \ob \cala$, $\hom_{\modules_\cala} (\calb, N)(x)$ is the equalizer of 
\[
\xymatrix{
\prod _{z \in \ob \cala}
\hom_\kring (\calb (z, x), N(z))   
\ar@<.5ex>[r]
\ar@<-.5ex>[r]
&
\prod_{f \in \hom_\cala (u,v)} 
\hom_\kring (\calb (u, x), N(v)), 
}
\]
where the maps are induced by the right $\cala$-module structures of $\calb(-,x)$ and $N$, as usual in considering natural transformations.
\end{rem}

Induction and coinduction are related by the vector space duality of Lemma \ref{lem:vs_duality}:

\begin{prop}
\label{prop:induct_coinduct_dual}
For  $f : \cala \rightarrow \calb$  a $\kring$-linear functor that is the identity on objects and $M$ a left $\cala$-module, there is a natural isomorphism of right $\calb$-modules 
\[
\hom_{\modules_\cala}(\calb , M^\sharp) 
\cong
(\calb\otimes_\cala M)^\sharp.
\]

Hence, if $N$ is  a right $\cala$-module that takes finite-dimensional values, there is a natural isomorphism of right $\calb$-modules:
\[
\hom_{\modules_\cala}(\calb , N) 
\cong
(\calb\otimes_\cala N^\sharp)^\sharp.
\]
\end{prop}

\begin{proof}
To illustrate the ideas involved, we sketch the proof. For $x \in \ob \cala$, the coequalizer expression for
 $(\calb \otimes _{\cala} M) (x) $ given in Proposition \ref{prop:left_adjoint_restriction} is isomorphic to the coequalizer of 
 \[
 \xymatrix{
 \bigoplus _{f \in  \cala(u,v)} 
\calb (v, x)  
\otimes 
M(u) 
\ar@<.5ex>[r]
\ar@<-.5ex>[r]
&
\bigoplus _{z \in \ob \cala} 
\calb (z, x) \otimes M(z),
}
 \]
with morphisms induced by the right $\cala$-module structure of $\calb (-,x)$ and the left $\cala$-module structure of $M$.

Dualizing, this gives the equalizer of
\[
 \xymatrix{
\prod _{z \in \ob \cala} 
(\calb (z, x) \otimes M(z))^\sharp
\ar@<.5ex>[r]
\ar@<-.5ex>[r]
&
\prod_{f \in  \cala(u,v)} 
(\calb (v, x)  
\otimes 
M(u))^\sharp. 
}
 \]
Using the isomorphisms of $\kring$-vector spaces $(\calb (z, x) \otimes M(z))^\sharp \cong \hom_\kring (\calb (z, x), M(z)^\sharp)$ and $(\calb (v, x)  
\otimes 
M(u))^\sharp \cong \hom_\kring (\calb (v,x) , M(u)^\sharp)$, one checks that this identifies with the equalizer defining $\hom_{\modules_\cala}(\calb , M^\sharp) $.
\end{proof}

\subsection{Small categories}

When working more generally, without imposing $\kring$-linearity, one uses the following:

\begin{defn}
\label{defn:calc_modules}
For $\calc$ a small category, the category of (graded) $\calc$-modules is the category of functors from $\calc$ to $\kgmod$, denoted $\calc \dash \kgmod$.  When restricted to values in $\kmod$, this category is denoted $\calc\dash \kmod$ or  $\fcatk[\calc]$. 
\end{defn}

\begin{rem}
For $\calc$ a small category, the $\kring$-linearization $\kring \calc$ is the small $\kring$-linear category with the same objects as $\calc$ such that $\hom_{\kring \calc} (x,y )= \kring \hom_\calc (x,y)$, for $x , y \in \ob \calc$. By the associated universal property, a $\kring$-linear functor from $\kring \calc$ to $\kmod$ is equivalent to a functor from $\calc$ to $\kmod$. It follows that  the category $\calc \dash \kmod$ is equivalent to the category of left $\kring \calc$-modules, as defined above. 
\end{rem}

When working with $\calc \dash \kgmod$, one has:

\begin{prop}
\label{prop:tensor_modules}
The symmetric monoidal structure of $\kgmod$  induces a symmetric monoidal tensor product on $\calc \dash \kgmod$, where the tensor product is formed pointwise. This restricts to a tensor product on $\calc \dash \kmod$.
\end{prop}

 \section{The category associated to an operad}
\label{sect:prop}

This section serves to recall the construction of the category associated to an operad. This requires introducing the underlying structures of operad theory, in particular the category of $\fb\op$-modules, where $\fb$ is the category of finite sets and bijections. The category of $\fb$-bimodules arises when passing to the associated category.

Most of this material is standard; the presentation also serves to introduce  notation. Section \ref{subsect:schur_opd} recalls the Schur functor associated to a $\fb\op$-module.

For the reader's convenience, references are mostly given to \cite{LV} rather than the original sources.
 
\begin{rem}
\label{rem:finite-dimension}
We work with $\kgmod$, the category of $\zed$-graded $\kring$-vector spaces. An object of $\kgmod$ has finite dimension if it is finitely-generated as a graded $\kring$-vector space. Equivalently, it is non-zero for only finitely many $n$ and each graded component has finite dimension.
\end{rem}
 
\subsection{Background} 
\label{subsect:background}
 
The category $\sop$ is the category of functors from $\fb\op$ to $\kgmod$, where $\fb$ is the category of finite sets and bijections.  Similarly, one has the category $\smod$ of functors from $\fb$ to $\kgmod$. Since $\fb$ is a groupoid, it is isomorphic to $\fb\op$, so that there is an isomorphism of categories $\smod \cong \sop$.

 By restricting to the skeleton of $\fb$ given by $\{ \mathbf{n}\  |\  n \in \nat \}$, a $\fb \op$-module $M$ is equivalent to a sequence $\{M(n) \ | \ n \in \nat \}$, where the term $M(n)$ in {\em arity} $n$  is a  $\zed$-graded $\kring$-vector space equipped with a right $\sym_n$-action. 
 
 The following notation is used without further comment:
 
 \begin{nota}
 For $M$ a $\fb\op$-module (or a $\fb$-module) and $n \in \nat$, $M(n)$ denotes $M(\mathbf{n})$, equipped with the appropriate $\sym_n$-action.
 \end{nota}

\begin{defn}
\label{defn:conv}
Let $(\sop, \conv, \kring)$ be the symmetric monoidal structure on $\sop$ given by the convolution product $\conv$, which is induced by the disjoint union of  finite sets. Namely, for $S$ a finite set, and $\fb\op$-modules $M$ and $N$, 
$$
M \conv N (S) := 
\bigoplus_{S_1 \amalg S_2 =S } 
 M (S_1) \otimes N(S_2)
 ,$$
where the sum is indexed by ordered decompositions of $S$ into two subsets.  The unit is the $\fb\op$-module sending a finite set $S$ to $\kring$ (in degree zero) if $S = \emptyset$ and $0$ if $S \neq \emptyset$. 
 \end{defn}

This allows the construction of the `tensor algebra' associated to $\conv$:

\begin{nota}
\label{nota:tc}
Denote by $\tc : \sop  \rightarrow \sop$ the functor given on $M \in \ob \sop$ by 
\[
\tc (M) := \bigoplus_{n \in \nat} M^{\conv n}.
\]
\end{nota}

 \begin{nota}
 \label{nota:circ}
  (Cf. \cite[Section 5.2]{LV}.)
  Denote by $(\sop, \circ, I)$ the monoidal structure on $\sop$ given by the composition product, where $I$ is the $\fb\op$-module with $I(n)=\kring$ (in degree zero) for $n=1$ and $0$ otherwise. 
 \end{nota}

Recall (see \cite[Section 5.2]{LV}, for example) that an operad $\opd$ is a unital monoid in $(\sop, \circ, I)$, in particular is equipped with a unit map $\eta : I \rightarrow  \opd$ and a composition $\mu : \opd \circ \opd \rightarrow \opd$.  An operad $\opd$ is reduced if $\opd(0)=0$.

\begin{exam}
\label{exam:opds}
The key examples of operads that arise here are:
\begin{enumerate}
\item 
the commutative operad $\com$ (encoding non-unital associative commutative algebras) and $\ucom$ (encoding unital associative  commutative algebras); 
\item 
the Lie operad $\lie$ (encoding Lie algebras);
\item 
the operad $\ass$ (encoding associative algebras) and $\uass$ (encoding unital associative algebras).
\end{enumerate}
The operads $\ass$, $\com$ and $\lie$ are reduced, whereas $\ucom$ and $\uass$ are not.

There is a commutative diagram of morphisms of operads
\[
\xymatrix{
\ass 
\ar[r]
\ar@{^(->}[d]
&
\com 
\ar@{^(->}[d]
\\
\uass 
\ar[r]
&
\ucom
}
\]
in which the horizontal morphisms represent forgetting the commutativity of the multiplication (any commutative, associative algebra is an associative algebra) and the vertical morphisms represent forgetting the unit. 

Moreover, there is a morphism of operads $\lie \hookrightarrow \ass$ that encodes the commutator Lie algebra of an associative algebra; by composition, this gives $\lie \hookrightarrow \uass$. 
\end{exam}

\begin{nota}
\label{nota:tensor_H}
Denote by $\otimes_H$ the tensor product on $\sop$ induced by the symmetric monoidal structure $(\kgmod, \otimes, \kring)$ (cf. Proposition \ref{prop:tensor_modules}), with unit the constant module $\kring$. 
\end{nota}

For operads $\opd$, $\ppd$,  $\opd \otimes_H \ppd$ has a natural operad  structure (see \cite[Section 5.3.2]{LV}). This defines a symmetric monoidal structure on the category of operads with unit the operad $\ucom$, and $\otimes_H$ is referred to as the Hadamard product.

This allows the introduction of the operadic suspension:

\begin{defn}
\cite[Section 7.2.2]{LV}
\label{defn:os}
\ 
\begin{enumerate}
\item 
Let $\os$ be the endomorphism operad $\mathrm{End} (s \kring)$, where $s \kring$ is in (homological) degree one. 
Explicitly, the underlying $\fb\op$-module is given for $n \in \nat$ by   $\os (n) = \hom ((s\kring)^{\otimes n}, s \kring)$, which identifies as the signature representation of $\sym_n\op$ placed in  degree $1-n$. 
\item 
The operadic suspension of an operad $\opd$ is  $\os \otimes_H \opd$. 
\end{enumerate}
\end{defn}

\subsection{The category associated to an operad}
\label{subsect:cat_opd}

Recall that a  PROP is a strict symmetric monoidal category with set of objects $\nat$, such that the monoidal structure on objects is given by addition \cite{MR171826}.

\begin{rem}
\label{rem:prop_bimodule}
For a  PROP $\mathbb{P}$, the symmetric monoidal structure ensures that, for all $m,n \in \nat$, the morphisms from $m$ to $n$, $\mathbb{P} (m,n)$ has a natural $\kring [\sym_m\op \times \sym_n]$-module structure. This means that the PROP $\mathbb{P}$ has an underlying $\fb$-bimodule (see Section \ref{subsect:fb_bimodule} for these).
\end{rem}

\begin{nota}
(Cf. \cite{MR1301191}, \cite[Section 5.4.1]{LV}.)
For $\opd$ an operad, denote by $\cat \opd$ the associated PROP, so that $\cat \opd (m,1) = \opd (m)$ for $m \in \nat$. 
\end{nota}

Explicitly, writing $\mathbf{Fin}$ for the category of finite sets, 
\begin{eqnarray}
\label{eqn:cat_opd}
\cat \opd (m,n) = \bigoplus_{f \in \hom_{\mathbf{Fin}}(\mathbf{m}, \mathbf{n}) } \bigotimes_{i=1}^n \opd (f^{-1} (i) ). 
\end{eqnarray}
If $\opd$ is reduced, then the sum can be taken to be indexed by $f \in \hom _{\fs} (\mathbf{m}, \mathbf{n})$, where $\fs$ is the category of finite sets and surjections. 

\begin{rem}
\label{rem:cat_opd}
\ 
\begin{enumerate}
\item 
The formation of $\cat \opd$ is functorial: in particular, a morphism of operads $\opd \rightarrow \ppd$ induces a morphism $\cat \opd \rightarrow \cat \ppd$ of PROPs. 
\item 
\label{item:tc_opd}
 The structure underlying $\cat \opd$ is  $\tc \opd := \bigoplus _{n \in \nat} \opd^{\conv n}$. Namely,  
for $m, n \in \nat$,  
$
\cat \opd (m, n) 
= 
\opd^{\conv n} (m).
$
\item 
For an operad $\opd$ in $\kgmod$, $\cat \opd$ has underlying category that is enriched in $\kgmod$.
\end{enumerate}
\end{rem}

The following is clear:

\begin{lem}
\label{lem:PROP_opd}
Suppose that  $\opd$ is a reduced operad such that  $ \opd (n)$ has finite dimension for all $n \in \nat$. Then, for $m,n \in \nat$:
\begin{enumerate}
\item 
$\cat \opd  (m, n)$ has finite dimension; 
\item 
$\cat\opd  (m,n) =0$ if $m< n$;
\item 
if $\opd (1) = \kring$ in degree zero,  the unit induces an isomorphism $\cat \opd (m,m) \cong \kring [\sym_m]$ of $\kring$-algebras.
\end{enumerate}
\end{lem}

\begin{exam}
\label{exam:unit_cat}
For the unit operad $I$, $\cat I (m, n)=0$ unless $m=n$, when $\cat I (m,m) = \kring [\sym_m]$. 
\end{exam}

\begin{exam}
\label{exam:com}
(Cf. \cite[Section 5.4.1, Example 1]{LV}.)
\begin{enumerate}
\item 
The commutative  operad $\com$  satisfies the hypotheses of Lemma \ref{lem:PROP_opd}, with $\com (1) = \kring$. 
 The category   $\cat \com$  is equivalent to the $\kring$-linearization $\kring \fs$ of the category $\fs$ of finite sets and surjections.
\item 
The unital commutative operad $\ucom$ is not reduced. The category $\cat \ucom$ is equivalent to $\kring \mathbf{Fin}$, the $\kring$-linearization of the category of finite sets.
\item 
The morphism of operads $\com \rightarrow  \ucom$ induces $\kring \com \rightarrow \kring \ucom$, which identifies with the embedding $\kring \fs \hookrightarrow \kring \mathbf{Fin}$ induced by the inclusion  $\fs \subset \mathbf{Fin}$ of the subcategory of surjective maps. 
\end{enumerate}
\end{exam}

\begin{exam}
\label{exam:assu}
(Cf. \cite[Section 5.4.1, Example 2]{LV}.) The unital associative operad $\uass$ is the $\kring$-linearization of a set operad, so $\cat \uass$ is the $\kring$-linearization of a category.

There is a  morphism of operads $\uass \rightarrow \ucom$ that corresponds to forgetting commutativity. This induces  a $\kring$-linear functor
$ 
\cat \uass \rightarrow \cat \ucom \cong \kring \mathbf{Fin}
$ 
which is essentially surjective.

For $m,n \in \nat$, $\mathbf{Fin}(\mathbf{m},\mathbf{n})$ forms a basis of $\cat \ucom (m,n)$. Correspondingly,  $\cat \uass (m,n)$ has  basis given by pairs 
 $(f : \mathbf{m} \rightarrow \mathbf{n} , \ord(f) )$, where 
 $f \in \mathbf{Fin}(\mathbf{m},\mathbf{n})$ and   $\ord (f)$ is an order on the fibres of $f$:
 \[
 \ord (f) : \mathbf{m_i} \stackrel{\cong} {\rightarrow} f^{-1}(i) \subset \mathbf{m}
 \]
 for $1 \leq i \leq n$, where $m_i := |f^{-1} (i)|$ so that $\mathbf{m_i} = \{ 1, \ldots , m_i \}$. 
 
For instance, taking $m=2$, $n=1$, $\mathbf{Fin}(\mathbf{2},\mathbf{1})$ contains a unique element.  Then $\cat \uass (2,1)$  
 has basis $f_{1<2}$ and $f_{2<1}$, corresponding to the orders $1<2$ and $2<1$ of the fibre $\mathbf{2}= \{1,2 \}$.
 
Composition in $\cat \uass$ extends that of $\kring \mathbf{Fin}$ in the obvious way, so as to be compatible with the orderings. In particular, for $t \in \nat$, the group $\sym_t$ acts on $\cat \uass (m,t)$ via post-composing by automorphisms of $\mathbf{t}$.

A graphical representation of the basis elements of $\cat \uass (m,t)$ is given by configurations of $m$ beads (labelled bijectively by $\mathbf{m}$) arranged on $t$ oriented line segments (labelled bijectively by $\mathbf{t}$). For instance, for $m=6$ and $t=4$, one such basis element is represented by: 

\ 

\begin{center}
 \begin{tikzpicture}[scale = 1.2]
\draw (0,0) -- (2,0); 
\draw (2.5,0) -- (4.5,0);
\draw (5,0) -- (7,0); 
\draw (7.5,0) -- (9.5,0);
  \draw [fill=white,thick] (.5,0) circle[radius = .15];
 \node at (.5,0) {$1$};
  \draw [fill=white,thick] (1,0) circle [radius = .15];
 \node at (1,0) {$5$};
  \draw [fill=white,thick] (1.5,0) circle [radius = .15];
 \node at (1.5,0) {$2$};
 \draw [fill=white,thick] (3.25,0) circle [radius = .15];
 \node at (3.25,0) {$6$}; 
 \draw [fill=white,thick] (8.5,0) circle [radius = .15];
 \node at (8.5,0) {$4$};
 \draw [fill=white,thick] (9,0) circle [radius = .15];
 \node at (9,0) {$3$}; 
 \node [below] at (0,0) {$\scriptstyle{1}$};
 \node [below] at (2.5,0) {$\scriptstyle{2}$};
 \node [below] at (5,0) {$\scriptstyle{3}$};
 \node [below] at (7.5,0) {$\scriptstyle{4}$};
\end{tikzpicture}
\end{center}
where the start of the $i$th line segment $i \in \mathbf{t}$ is indicated by the subscript. Note that it is the order of  beads on a segment which is important, not their position. 

Then the action of $\mathrm{Aut}(\mathbf{t})$ corresponds to reindexing the segments (which can be viewed as reordering). More generally, composition can be understood as a (discrete) generalization of the composition of the little $1$-discs operad, by viewing a bead as corresponding to an open interval. 
\end{exam}

\begin{exam}
\label{exam:lie}
The Lie operad $\lie$  satisfies the hypotheses of Lemma \ref{lem:PROP_opd}, with $\lie (1) = \kring$. However, the category $\cat \lie$ is not the $\kring$-linearization of a category, since $\lie$ does not arise from an operad in sets. The category $\cat \lie$ is described explicitly as follows. For $m, n \in \nat$, $\cat \lie (m,n)$ is the quotient of the $\kring$-vector space generated by forests of rooted binary planar trees with the set of roots labelled (bijectively) by $\mathbf{n}$ and the set of leaves labelled (bijectively) by $\mathbf{m}$, modulo the antisymmetry (AS) and Jacobi (IHX) relations. Composition is induced by the operation of grafting the root of a rooted binary planar tree onto the leaf of another.

The morphism of operads $\lie \rightarrow \uass$ that encodes the commutator Lie algebra of a unital associative algebra  induces a $\kring$-linear functor
$$
\cat \lie \rightarrow \cat \uass
$$
and this is injective. If one takes into account the PROP structure, this is determined by $\cat \lie (2, 1) 
\rightarrow \cat \uass (2,1)$. This can be identified explicitly as follows. One has $\cat \lie (2,1) \cong \kring$, with generator represented by the unique rooted binary planar tree with two leaves, with the left leaf labelled by $1\in \mathbf{2}$. This is sent to the difference $f_{1<2} - f_{2<1} \in \cat \uass (2,1)$, using the notation introduced in Example \ref{exam:assu}.

This is closely related to the STU relation that occurs in considering Jacobi diagrams (see, for example, \cite[Chapter 5]{MR2962302}). The latter can be represented as

\ 
\begin{center}
\begin{tikzpicture}[scale = .8]
\draw (-.25,0) -- (1.25,0); 
\draw (2.5,0) -- (4.5,0);
\draw (5,0) -- (7,0); 
\draw (3,0) -- (3,1); 
\draw (4,0) -- (4,1);
\draw (5.5,0) -- (6.5,1);
\draw [fill=white,draw =none] (6,.5) circle[radius = .1];
\draw (5.5,1) -- (6.5,0);
\node at (1.75,.5) {$=$};
\node at (4.75,.5) {$-$};
\draw (0,1) -- (.5, .5) -- (1,1);
\draw (.5,.5) -- (.5,0); 
\draw [fill=black,draw =none] (.5,.5) circle[radius = .05];
 \draw [fill=white,thick] (.5,0) circle[radius = .15];
 \draw [fill=white,thick] (3,0) circle [radius = .15];
  \draw [fill=white,thick] (4,0) circle [radius = .15];
 \draw [fill=white,thick] (5.5,0) circle [radius = .15]; 
 \draw [fill=white,thick] (6.5,0) circle [radius = .15];
\node [below] at (-.25,0) {$\scriptstyle{i}$}; 
\node [below] at (2.5,0) {$\scriptstyle{i}$}; 
\node [below] at (5,0) {$\scriptstyle{i}$}; 
\end{tikzpicture}
\end{center}
where the label $i$ of each oriented line segment stresses  that these represent the {\em same} segment. 
\end{exam}

\subsection{The $\fb$-bimodule viewpoint}
\label{subsect:fb_bimodule}
$\fb$-bimodules (as introduced below) arise naturally when considering PROPs and their underlying $\kring$-linear categories, as observed in Remark \ref{rem:prop_bimodule}.

\begin{defn}
\label{defn:sbimod}
Let $\sbimod$ be the category of $\fb$-bimodules, i.e., functors from $\fb\op \times \fb$ to $\kgmod$. 
\end{defn}

\begin{rem}
Using the skeleton  of $\fb$,  
a $\fb$-bimodule $B$ is equivalent to  a sequence of $\kring[\sym_m \op \times \sym_n]$-modules $B(m,n)$ in $\kgmod$, indexed by $m,n \in \nat$.  
\end{rem}

The coend 
$
\otimes_\fb : \sop
\times \smod 
\rightarrow 
\kgmod
$ 
is 
given explicitly by 
$
N \otimes_\fb M := \bigoplus _{t \in \nat}
N(t) \otimes_{\sym_t} M(t).
$
This induces
$$
\otimes_\fb : \sbimod \times
\sbimod
\rightarrow
\sbimod
$$
so that   
$
(B_1 \otimes_\fb B_2) (m,n)  = 
\bigoplus_{t \in \nat} 
B_1 (t, n ) \otimes _{\sym_t} B_2 (m, t).
$

\begin{nota}
\label{nota:bimodunit}
Denote by $\bimodunit \in \ob \sbimod$ the $\fb$-bimodule given by 
$$
\bimodunit (m,n) :=
\left\{ 
\begin{array}{ll}
0 & m \neq n 
\\
\kring [\sym_n]& m=n,
\end{array}
\right. 
$$
where $\kring [\sym_n]$ is concentrated in degree zero, equipped  with the regular left and right actions. 
\end{nota}

One has the following standard result for $\fb$-bimodules:

\begin{prop}
\label{prop:fb_bimodules_otimes_fb}
The functor $\otimes_\fb$ induces a monoidal structure $(\sbimod, \otimes_\fb, \bimodunit)$. 
\end{prop}

This allows the consideration of monoids in  $(\sbimod, \otimes_\fb, \bimodunit)$. The key examples here are provided by the following:

\begin{prop}
\label{prop:bimod_monoid}
For an operad $\opd$, $\cat \opd$ has the structure of a unital monoid in 
$(\sbimod, \otimes_\fb, \bimodunit)$, with unit $\bimodunit \rightarrow \cat \opd$ induced by the operad unit $I \rightarrow \opd$ and composition $\cat \opd \otimes_\fb \cat \opd \rightarrow \cat \opd$ induced by the operad multiplication $\opd \circ \opd \rightarrow \opd$.
\end{prop}

\begin{exam}
\label{exam:I_bimodunit}
For $\opd =I$, the underlying $\fb$-bimodule of $\cat I$ is $\bimodunit$, equipped with its canonical monoid structure (cf. Example \ref{exam:unit_cat}).
\end{exam}

The monoidal structure $(\sbimod, \otimes_\fb, \bimodunit)$ allows consideration of left (respectively right) modules over a monoid. The categorical definition gives a notion internal to $\sbimod$. Here we use the following:

\begin{defn}
\label{defn:left_B-module}
For $B$ a unital monoid in $\sbimod$, a left $B$-module is $M\in \ob \smod$ equipped with a structure morphism $ \psi_M : B \otimes_\fb M \rightarrow M$ in $\smod$  that satisfies the unital and associativity axioms. 

A morphism of left $B$-modules from $(M, \psi_M)$ to $(N, \psi_N)$ is a morphism $M \rightarrow N$ in $\smod$ that is compatible with the structure morphisms $\psi_M$, $\psi_N$. 
\end{defn}

\begin{rem}
One can also consider {\em comonoids} in $(\sbimod, \otimes_\fb, \bimodunit)$.  Given a comonoid, one has the notion of a left (respectively right) comodule over the comonoid.
\end{rem}

\subsection{Exploiting Schur functors}
\label{subsect:schur_opd}

Schur functors provide an important tool for studying $\fb\op$-modules and operads. The standard reference for the underlying relation between representations of the symmetric groups and polynomial functors over a field of characteristic zero is \cite[Appendix A to Chapter I]{MR3443860}. 

In this section, objects are  ungraded, i.e.,  working with $\sopug$, $\smodug$ and ungraded objects of $\sbimod$.

\begin{nota}
\label{nota:underline_V}
For $V \in \ob \kmod$, let $\underline{V} \in \ob \smodug$ be given by  $\underline{V} (n):= V^{\otimes n}$, where $\sym_n$ acts on $V^{\otimes n}$ by place permutations of the tensor factors.
\end{nota}

\begin{defn}
\label{defn:schur_functor}
(Cf. \cite[Section 5.1.2]{LV}.)
For $F \in \ob \sopug$,  the Schur functor $V \mapsto F(V)$ (a functor from $\kmod$ to $\kmod$) is given for $V \in \ob \kmod$ by 
$
F(V):= F \otimes_\fb \underline{V}, 
$
so that $$F(V) = \bigoplus_{n \in \nat} F(n) \otimes_{\sym_n} V^{\otimes n}.$$
\end{defn}

\begin{rem}
For $B \in \ob \sbimod$ a $\fb$-bimodule,  $V \mapsto B(V) := B \otimes_\fb \underline{V}$, for $V \in \ob \kmod$,  is considered as a functor from $\kmod$ to $\smodug$ (or, equivalently, a $\kmod \times \fb$-module, with values in $\kmod$).
\end{rem}

\begin{rem}
\label{rem:Schur_identifications}
Structure of $\sopug$ is reflected in that of the category of functors on $\kmod$. In particular, for $F , G \in \ob \sopug$, there are natural isomorphisms of Schur functors (taking $V \in \ob \kmod$):
\begin{enumerate}
\item 
$(F \conv  G) (V) \cong F(V) \otimes G(V)$ (see \cite[Proposition 5.1.2]{LV}); 
\item 
$(F \circ G) (V) \cong F (G (V) ) $ (see \cite[Proposition 5.1.3]{LV}). 
\end{enumerate}

Moreover, working over a field $\kring$ of characteristic zero, the Schur functor $V \mapsto F(V)$ encodes the $\fb\op$-module $F$ (see \cite[Appendix A to Chapter I]{MR3443860} or \cite[Lemma 5.1.1]{LV}). Hence Schur functors provide a powerful tool in this context.
\end{rem}

The free left $\cat\opd$-module generated by $V \in \ob \kmod$ is, by definition, $\cat \opd \otimes_\fb \underline{V}$. This is related to the free $\opd$-algebra on $V$ by the following:

\begin{lem}
\label{lem:identify_free_cat_opd_module}
Let $\opd$ be an operad in $\kmod$. For $V \in \ob \kmod$, there is a natural isomorphism  $\cat \opd \otimes_\fb \underline{V}  \cong \underline{\opd (V)}$ in $\smodug$.
\end{lem}

\begin{proof}
As observed in Remark \ref{rem:cat_opd}, the underlying $\fb\op$-module of $\cat \opd (-, n)$ is $\opd ^{\conv n}$ and this is $\sym_n$-equivariant, using the symmetry for $\conv$ to define the action on the right hand side. By the first identification given in Remark \ref{rem:Schur_identifications}, this gives $\cat \opd (-, n) (V) \cong \opd (V)^{\otimes n}$, naturally with respect to $V \in \ob \kmod$. Moreover, the $\sym_n$-action induced by that on $\cat \opd (-,n)$ corresponds to the place permutation of the tensor factors. 
\end{proof}

\section{Representations of $\cat \opd$} 
 \label{sect:rep}
 
The category of representations of an operad $\opd$ is introduced in this section.  The main interest is in the case where $\opd$ is a reduced operad in $\kmod$, for which the presentation of Section \ref{subsect:rep} is preferred; a more general approach (allowing for gradings) is outlined in Section \ref{subsect:rep_module}. Here we focus upon the covariant setting, which corresponds to working with  left modules; one can also work contravariantly (using right modules) and these situations are related via duality. 
 
These structures are not new. For instance, the contravariant framework is equivalent to the category of right $\opd$-modules with respect to the operadic composition product, as considered by Fresse in \cite{MR2494775}, for example.
 
\subsection{Representations as functors}
\label{subsect:rep}

Throughout this subsection, the following hypothesis is imposed, although some of the results hold in greater generality. 

\begin{hyp}
\label{hyp:opd}
The operad $\opd$ is reduced and   concentrated in degree zero (i.e., is an operad in $\kmod$),  with $\opd (1) = \kring$ and $\dim \opd (n) < \infty$ for all $n\in \nat$. 
\end{hyp}

\begin{defn}
\label{defn:fopd}
Let 
\begin{enumerate}
\item
$\fopd$ be the category of $\kring$-linear functors from $\cat \opd$ 
 to $\kmod$;
\item 
$\popd_n$ denote the functor $\cat \opd  (n, -)$, for $n\in \nat$.
\end{enumerate}
\end{defn}

The category $\fopd$ is termed the category of representations of $\cat \opd$.

\begin{exam}
\label{exam:f_I}
For the unit operad $I$, the category of representations $\f_I$ is equivalent to the category $\smodug$, by the identification of $\cat I$ given in Example \ref{exam:unit_cat}. 
\end{exam}

 There is a natural source of objects of $\fopd$, namely the category $\alg_{\opd}$ of  $\opd$-algebras:

\begin{prop}
\label{prop:opd_alg}
\cite[Proposition 5.4.2]{LV} 
There is  a faithful embedding $
\alg_{\opd} 
\hookrightarrow 
\fopd
$ 
that sends an $\opd$-algebra $A$ to the functor $\underline{A}$ with $\underline{A}(n) := A^{\otimes n}$ and with morphisms acting via the $\opd$-algebra structure of $A$. 
\end{prop}

Yoneda's lemma provides projective generators for $\fopd$:

\begin{prop}
\label{prop:fopd_ab}
The category $\fopd$ is abelian and has enough projectives. 
In particular,
\begin{enumerate}
\item 
$\{ \popd_n | n \in \nat \}$ is a set of projective generators of $\fopd$;
\item
for $ t \in \nat$,  $\popd_n (t) $ has finite dimension and $\popd_n (t) = 0$ for $t>n$; 
\item 
$\popd_n$ is finite (i.e., it has a finite composition series).
\end{enumerate} 

The Yoneda embedding gives a fully-faithful $\kring$-linear functor $\popd_\bullet : (\cat \opd)  \op \rightarrow \fopd$,  $n \mapsto  \popd_n$.
 \end{prop}

\begin{proof}
That $\fopd$ is abelian follows from the general case, as presented in Section \ref{subsect:k-linear-modules}. 

The existence of enough projectives and the first point follow from the Yoneda lemma. The second point follows directly from Lemma \ref{lem:PROP_opd}. 
 This implies readily that $\popd_n$ is finite, by considering the dimension of the values of objects of $\fopd$. 

For the final statement, by Yoneda, for $m, n \in \nat$, one has $\hom_{\fopd} (\popd_n, \popd_m ) = \cat \opd (m,n)$. Thus the full subcategory of $\fopd$ with objects $\{ \popd_n | n \in \nat\}$ is equivalent to  $(\cat \opd)\op$. 
\end{proof}

The category $\fopd$ has a natural filtration defined as follows:

\begin{defn}
\label{defn:filt_fopd}
Let 
\begin{enumerate}
\item 
$\fopd^{\leq n}$, for $n \in \nat$, be the full subcategory of $\fopd$ with objects $F$ such that $F(t) =0$ for $t >n$; 
\item 
$\fopd^{< \infty}:= \bigcup_t \fopd ^{\leq t} \subset \fopd$;
\item 
$\fopd\fn$ be the full subcategory of $\fopd$ of objects that have a finite composition series (equivalently, the full subcategory of $\fopd^{< \infty}$ of objects that take finite-dimensional values).
\end{enumerate}
\end{defn}

This gives the increasing filtration of full subcategories: 
\[
0 \subset \fopd^{\leq 0} \subset \fopd^{\leq 1} \subset \ldots \subset \fopd^{\leq n} \subset \fopd^{\leq n+1} \subset \ldots \subset \fopd.
\]

\begin{prop}
\label{prop:ext_fopd_fn}
\ 
\begin{enumerate}
\item 
The subcategory $\fopd\fn$ is an abelian subcategory of $\fopd$ with enough projectives and  the inclusion $\fopd \fn \hookrightarrow \fopd $ preserves projectives. 
\item 
For $F, G \in \ob \fopd\fn$, the inclusion $\fopd\fn \subset \fopd$ induces an isomorphism
\[
\ext^*_{\fopd\fn} (F, G) 
\stackrel{\cong}{\rightarrow} 
\ext^*_{\fopd} (F, G). 
\]
\item 
The category $\fopd$ is locally finite (i.e., every object is the colimit of its finite subobjects).
\end{enumerate}
\end{prop}

\begin{proof}
For $n \in \nat$, $\popd_n$ belongs to $\fopd^{\leq n}$ and to $\fopd\fn$, by Proposition \ref{prop:fopd_ab}.  
Thus $\fopd\fn$ is an abelian subcategory of $\fopd$ with enough projectives and the inclusion $\fopd \fn \hookrightarrow \fopd $ preserves projectives. This implies the  statement concerning $\ext$.

That $\fopd$ is locally finite follows from the fact that the projectives of $\fopd$ are finite.
\end{proof}

\begin{prop}
\label{prop:fopd_colimit}
For $n\in \nat$ and $F \in \ob \fopd$,
\begin{enumerate}
\item
 the inclusion $\fopd^{\leq n} \subset \fopd$ admits an exact  right adjoint $(-)^{\leq n}$ given by $F^{\leq n} (t)= F(t)$ for $t \leq n$, and $0$ otherwise; 
 \item 
there are  canonical inclusions $F^{\leq n} \subseteq F^{\leq n+1} \subseteq F$;
\item 
$F \cong \lim_{\substack{\rightarrow \\ n \mapsto \infty} } F^{\leq n}.$ 
\end{enumerate}
\end{prop}

The following result establishes a precise relationship between representations of the symmetric group $\sym_n$ and the category $\fopd$:

\begin{prop}
\label{prop:fopd_eval_n}
For $n \in \nat$, evaluation on $\mathbf{n}$ induces an exact functor $(-)^n : \fopd^{\leq n} \rightarrow \kring [\sym_n] \dash \modules$. Moreover, 
\begin{enumerate}
\item
$(-)^n$ has kernel $\fopd^{\leq n-1}$; 
\item 
$(-)^n$ has an exact right adjoint given by extension by zero (i.e., a $\kring [\sym_n]$-module is considered as an object of $\fopd$ supported on $\mathbf{n}$).
\end{enumerate}
\end{prop}

In particular, this leads to the classification of the simple objects of $\fopd$ via:

\begin{cor}
\label{cor:fopd_simples}
For $n\in \nat$, the set of isomorphism classes of simple objects of $\fopd^{\leq n}$ is finite and identifies with $\bigcup_{0\leq j \leq n} X_j$, where $X_j$ is the set of isomorphism classes of simple $\sym_j$-modules.
\end{cor}

\begin{proof}
A proof is outlined so as to show how Proposition \ref{prop:fopd_eval_n} applies. The result is proved by an obvious increasing induction upon $n$; the inductive step is proved below. 

Consider an object $F$ of $\fopd ^{\leq n}$; for the adjunction of Proposition \ref{prop:fopd_eval_n}, the adjunction unit gives   a natural transformation $F \rightarrow F^n$ in $\fopd^{\leq n}$, where $F^n$ is considered as an object of $\fopd ^{\leq n}$ by extension by zero. By construction, this is an isomorphism when evaluated upon $\mathbf{n}$, in particular it is surjective, with kernel in $\fopd^{\leq n-1}$. 

Suppose that $S$ is a simple object of $\fopd^{\leq n}$. By the above, there are two possibilities: either $S$ lies in the subcategory $\fopd^{\leq n-1}$ or $S \cong S^n$. The first case is treated by the inductive hypothesis. In the second, one shows that $S$ is simple if and only if the $\sym_n$-module $S^n$ is simple; this follows since the above argument shows that the full subcategory of $\fopd^{\leq n}$ of objects supported on $\mathbf{n}$ is equivalent to the category of $\sym_n$-modules. 

The inductive step follows and hence the result.
\end{proof}

The above constructions are natural with respect to the operad. For instance, one has:

\begin{prop}
\label{prop:naturality}
Suppose  $\ppd$ is an operad that also satisfies Hypothesis \ref{hyp:opd}. For a morphism of operads $\opd \rightarrow \ppd$, restriction along the induced functor $ 
\cat \opd \rightarrow \cat \ppd, 
$   gives an exact functor 
$
\fppd \rightarrow \fopd
$.
 This is compatible with the respective filtrations: i.e., for $n \in \nat$, this restricts to $\fppd^{\leq n}
\rightarrow \fopd^{\leq n}$.
\end{prop}

\begin{exam}
\label{exam:F_I_opd}
The unit $I \rightarrow \opd$  of the operad $\opd$ induces 
\[
 \fopd   \rightarrow \f_I  \cong \smodug
\] 
using the identification of Example \ref{exam:f_I}. This  is the forgetful functor to the underlying $\fb$-module.

Hypothesis \ref{hyp:opd} implies that the operad $\opd$ has an augmentation $\opd \rightarrow I$. This induces the section
\[
 \smodug  \cong \f_I \rightarrow \fopd ;
\] 
this encodes all the extension by zero functors of Proposition \ref{prop:fopd_eval_n}.
\end{exam}

\subsection{The left $\cat \opd$-module approach}
\label{subsect:rep_module}
In this subsection, $\opd$ is a reduced operad, but Hypothesis \ref{hyp:opd} is not imposed.

If $\opd$ is not concentrated in degree zero, then one can modify the above approach by using categories enriched in $\kgmod$. There is an  alternative approach that is adopted here: namely, we consider $\cat \opd$ as a unital monoid in  $(\sbimod, \otimes_\fb, \bimodunit)$  by Proposition \ref{prop:bimod_monoid} and work with the category of left $\cat \opd$-modules (see Definition \ref{defn:left_B-module}). 

The following  establishes compatibility with the previous construction in the ungraded case: 

\begin{prop}
\label{prop:left_modules_fopd}
Suppose that $\opd$ satisfies Hypothesis \ref{hyp:opd}. 
The category $\fopd$ of $\kring$-linear functors from $\cat \opd$ to $\kmod$ is equivalent to the category of left $\cat \opd$-modules.
\end{prop}

\begin{proof}
This generalizes the identification of the category of left $\cala$-modules, for a $\kring$-linear category, that is given in Remark \ref{rem:category_left_cala_modules}. The definition of a left $\cat \opd$-module given in Definition \ref{defn:left_B-module} includes the underlying $\fb$-module structure as part of the data; this is  canonically derived from the unit map $\cat I \rightarrow \cat \opd$, leading to the stated equivalence.

For clarity, the argument is explained for the case of $\kring$-algebras; the extension to the case at hand is straightforward. Thus we fix an inclusion $B_0  \hookrightarrow B$ of unital $\kring$-algebras. 

By restriction any $B$-module $M$ is canonically a $B_0$-module and the structure morphism $B \otimes_\kring M \rightarrow M $ factors canonically across the surjection of left $B$-modules $B \otimes_\kring M \twoheadrightarrow B \otimes_{B_0} M$, where the tensor product $\otimes_{B_0}$ is defined with respect to the obvious right $B_0$-module structure on $B$. 

In particular, this applies to the case of $B$, considered as a left $B$-module. The multiplication $B \otimes_\kring B \rightarrow B$ factors canonically to give $B \otimes_{B_0} B \rightarrow B$ and this is a morphism of $B_0$-bimodules with respect to the obvious structures. 

This makes $B$ a unital monoid in the category of $B_0$-bimodules, so that we may consider the appropriate form of left module in this setting, namely a left $B_0$-module $N$ equipped with a structure morphism $B \otimes_{B_0} N \rightarrow N$ of left $B_0$-modules that satisfies the associativity and unital axioms. Then the composite 
$B \otimes_\kring N \twoheadrightarrow B \otimes_{B_0} N \rightarrow N$ makes $N$ into a $B$-module in the usual sense. 

These constructions are natural and provide the equivalences between the two different notions of left $B$-module considered above.
\end{proof}

Without imposing further hypotheses upon $\opd$, one has the following result, which  in particular provides a generalization of the projectives arising in Proposition \ref{prop:fopd_ab}:

\begin{prop}
\label{prop:proj_cat_opd-modules}
For a reduced operad $\opd$: 
\begin{enumerate}
\item
the category of left $\cat \opd$-modules is abelian; 
\item 
for $m \in \nat$, the monoid structure makes $\cat \opd (m, -)$ a left $\cat \opd$-module and this is projective; 
\item 
if $\opd (1) = \kring$, then $\kring [\sym_m]$ (considered as a $\fb$-module concentrated in arity $m$ and in degree zero) is 
naturally a left $\cat \opd$-module quotient of $
\cat \opd (m, -)$ 
and this exhibits $\cat\opd (m, -)$ as the projective cover of $\kring [\sym_m]$.  
\end{enumerate}
In particular, the category of left $\cat\opd$-modules has enough projectives.
\end{prop}
 
\begin{proof}
The abelian category structure of the category of left $\cat \opd$-modules is derived from that of the category of $\fb$-bimodules. 

The left $\cat \opd $-module structure of $\cat \opd (m, -)$ is given by restricting the monoid structure of $\cat \opd$ in $\fb$-bimodules. That $\cat \opd (m, -)$ is  projective is a generalization of Yoneda's lemma to this enriched context; more precisely, one shows that $\cat \opd (m,-)$ corepresents the evaluation functor $M \mapsto M(m)$. The proof is essentially the same as that of Yoneda's lemma, {\em mutatis mutandis}. 

For the final numbered statement, the hypothesis that $\opd$ is reduced and that $\opd (1) = \kring$ (necessarily concentrated in degree zero) implies that $\cat \opd (m,n)=0$ for $n >m$ and that $\cat \opd (m,m) \cong \kring [\sym_m]$ as $\kring$-algebras (concentrated in degree zero), as in Lemma \ref{lem:PROP_opd}. 

A straightforward generalization of Proposition \ref{prop:fopd_eval_n} then provides the surjection of $\cat\opd$-modules:
\[
\cat \opd (m, -) \twoheadrightarrow \kring [\sym_m],
\]
where $\kring [\sym_m]$ is considered as a $\cat \opd$-module supported on $\mathbf{m}$. This is an isomorphism when evaluated upon $\mathbf{m}$.

Since we have already established that $\cat \opd (m,-)$ is projective, it remains to show that $\cat \opd (m,-)$ is the projective cover of $\kring [\sym_m]$. This follows from the fact that the endomorphism algebras of $\cat \opd (m, -)$ and of $\kring [\sym_m]$ are isomorphic, a consequence of the above identification.

To see that the category of left $\cat \opd$-modules has enough projectives, it suffices to observe that the family $\cat \opd (m, -) [t]$, where $m \in \nat$, $t \in \zed$ and $[t]$ denotes shift of grading, is a set of projective generators. 
\end{proof}

 \subsection{The case of the Lie operad}
 \label{subsect:flie}
This subsection aims to make the structure of $\flie$ more explicit. Here we are working in the ungraded setting.
 
 The Lie operad $\lie$ is generated by $\lie (2) \cong \kring$, with the generator corresponding to the Lie bracket, $[ -, -]$. It follows that $\cat \lie$ is generated as a PROP by the corresponding generator in $\cat \lie (2,1) = \lie (2)$. If one forgets the symmetric monoidal structure, only retaining the $\kring$-linear category structure, one gives a quadratic presentation of the morphisms as follows. 
 
 As in Section \ref{sect:kos_background}, the morphisms of $\cat \lie$ are $\nat$-graded by setting $\cat \lie (s,t)$ to have degree $s-t$. The degree zero part corresponds to the subcategory of endomorphisms, identified by $\cat \lie (n,n) \cong \kring [\sym_n]$ for all $n \in \nat$. 
 
 Next we consider the

 \begin{nota}
 \label{nota:alpha}
 For $1 \leq n \in \nat$, let $\alpha_n \in \cat \lie (n+1, n)$ be  the morphism associated via the identification of equation  (\ref{eqn:cat_opd}) to the order preserving surjection $f : \mathbf{n+1} \twoheadrightarrow \mathbf{n}$ determined by $f^{-1} (1)= \{ 1, 2 \}$ using the canonical generators of $\lie (1)$ and $\lie (2)$.
 \end{nota}
 
\begin{lem}
\label{lem:degree_1_catlie}
For $1 \leq n \in \nat$, $\cat \lie (n+1, n)$ is generated as a $\sym_{n+1}\op \times \sym_n$-module by $\alpha_n$. This induces the isomorphism of $\sym_{n+1}\op \times \sym_n$-modules:
\[
\cat \lie (n+1, n) \cong (\mathrm{sgn}_2 \boxtimes \kring [\sym_{n-1}] ) \uparrow _{(\sym_2 \times \sym_{n-1})\op \times \sym_{n-1}} ^{\sym_{n+1}\op \times \sym_n},
\]
where the signature representation  $\mathrm{sgn}_2$ is considered as a $\sym_2\op$-module and $\kring [\sym_{n-1}]$ as a $\kring [\sym_{n-1}]$-bimodule. 
\end{lem} 
 
 \begin{proof}
 This follows from the description of morphisms of $\cat \lie$ given in equation (\ref{eqn:cat_opd}).
 \end{proof}

For $d \in \nat$, write $\cat \lie^d$ for the subspace of morphisms of degree $d$.  Thus composition in $\cat \lie$ restricts to:
\[
\cat \lie^d \otimes \cat \lie ^e \rightarrow \cat \lie^e. 
\]

 \begin{prop}
 \label{prop:catlie_generation}
 The morphisms of the $\kring$-linear category $\cat \lie$ are generated over $\cat \lie^0$ by $\cat \lie^1$. 
 \end{prop}
 
 \begin{proof}
 This follows from the construction of $\cat \lie$ from the operad $\lie$ and the fact that the Lie operad is generated by $\lie (2)$. 
 \end{proof}
 
 \begin{rem}
 \label{rem:cat_lie_presentation}
More is true: $\cat \lie$ has a quadratic presentation (cf. Proposition \ref{prop:quadratic_cat} below). There are two types of relation for the composition of morphisms of $\cat \lie^1$
\begin{enumerate}
\item 
commutation relations, when the composition does not invoke the composition in the operad $\lie$; 
\item 
the Jacobi relation, coming from the operad $\lie$. 
\end{enumerate}
 \end{rem}
 
Proposition \ref{prop:catlie_generation} leads to the identification of the structure that encodes a $\cat \lie$-module, showing that these are very accessible.

\begin{cor}
\label{cor:cat_lie_modules}
An object $M$ of $\flie$ is uniquely determined by 
\begin{enumerate}
\item 
the underlying $\fb$-module, namely the family of representations $M (n) \in \ob \kring [\sym_n]\dash\modules$, for $n \in \nat$; 
\item 
the family of $\sym_{n+1}\op \times \sym_n$-equivariant structure maps 
\begin{eqnarray}
\label{eqn:str_map}
\cat \lie (n+1, n) \rightarrow \hom_\kring (M(n+1), M(n)),
\end{eqnarray}
for $1 \leq n \in \nat$. 
\end{enumerate}
The structure map (\ref{eqn:str_map}) is uniquely determined by the image of $\alpha_n$ in $\hom_\kring (M(n+1), M(n))$. 
\end{cor}

\begin{proof}
The first statement follows immediately from Proposition \ref{prop:catlie_generation}. The final statement then follows from 
Lemma \ref{lem:degree_1_catlie}. 
\end{proof}

 \begin{rem}
 \ 
 \begin{enumerate}
 \item 
 The statement of Corollary \ref{cor:cat_lie_modules} extends to treat morphisms in $\flie$.
 \item 
 Not every element of $\hom_\kring (M(n+1), M(n))$ corresponds to an  equivariant map $\cat \lie (n+1, n) \rightarrow \hom_\kring (M(n+1), M(n))$; indeed, using the description of $\cat \lie (n+1, n)$ given in Lemma \ref{lem:degree_1_catlie}, there is an evident necessary and sufficient criterion. 
 \item 
 The structure maps (\ref{eqn:str_map}) must be compatible with the quadratic relations outlined in Remark \ref{rem:cat_lie_presentation}. Using this, one can refine Corollary \ref{cor:cat_lie_modules} to give an explicit characterization of $\flie$ in terms of the given structure morphisms.
 \end{enumerate}
 \end{rem}

That $\flie$ is not semisimple is shown by the following example.

\begin{exam}
\label{exam:flie}
Consider the  projective $P^\lie_2$, so that $P^\lie_2 (2) \cong \kring[\sym_2]$,   $P^\lie_2 (1) \cong \lie (2) \cong \kring$ 
 and $P^\lie_2 (n)=0$ for $n\in \nat \backslash \{1,2 \}$. 
 
One has the splitting as $\kring [\sym_2]$-modules $P^\lie_2 (2) \cong \kring_{\mathrm{triv}} \oplus \kring_{\mathrm{sgn}}$, as the sum of the trivial and sign representations. The structure morphism $\cat \lie (2,1) \cong \lie (2)$ is antisymmetric hence acts via:
\[
P^\lie_2 (2) \cong \kring_{\mathrm{triv}} \oplus \kring_{\mathrm{sgn}}
\rightarrow 
P^\lie_2 (1)  \cong \kring
\]
sending $\kring_{\mathrm{triv}}$ to zero and acting as the identity on the underlying $\kring$-vector spaces 
$\kring_{\mathrm{sgn}} \rightarrow \kring$.

This exhibits a non-split short exact sequence in $\flie$
\[
0
\rightarrow 
\kring (1) 
\rightarrow 
\mathscr{E}
\rightarrow 
\kring_{\mathrm{sgn}} (2)
\rightarrow 
0
\]
where $\kring (1) $ and $\kring_{\mathrm{sgn}} (2)$ denote the $\kring [\sym_1]$ and $\kring [\sym_2]$-modules considered 
 as $\cat \lie$-modules.
 
 The functor $P^\lie_2$ splits as $\mathscr{E} \oplus \kring_{\mathrm{triv}}(2)$, so that both $\mathscr{E}$ and $\kring_{\mathrm{triv}}(2)$ are projective in $\flie$.
\end{exam}

\part{Koszul duality}
 \section{Background on Koszul duality}
\label{sect:kos_background}

This section serves to present the notion of a Koszul abelian category in a way that is sufficient for current purposes, namely for the application to the study of the category $\fopd$ of representations of a binary quadratic operad. The Koszul duality theory for operads is then covered in Section \ref{sect:koszul}.

\subsection{Quadratic and Koszul categories}
\label{subsect:quad_kos}

First recall the  notion of a quadratic ring: 

\begin{defn}
\cite[Definition 1.2.2]{MR1322847}
A quadratic ring is a non-negatively graded ring $ A = \bigoplus_{j \in \nat} A_j$ such that $A_0$ is semisimple and $A$ is generated over $A_0$ by $A_1$ with relations in degree $2$.
\end{defn}

This carries over to suitable categories. This is outlined below in sufficient generality for current purposes.

\begin{hyp}
\label{hyp:calc}
Suppose that $\calc$ is a $\kring$-linear category such that 
\begin{enumerate}
\item 
 $\ob \calc = \nat$;
\item 
$\dim \calc (m,n) < \infty$  $\forall m,n \in \nat$; 
\item 
$\calc (m,n)=0$ if $m< n$.
\end{enumerate}
\end{hyp}

One has the following family of category algebras:

\begin{defn}
\label{defn:cat_alg}
For $\calc$ satisfying Hypothesis \ref{hyp:calc} and $N \in \nat$, let $A (\calc, N)$ be the $\nat$-graded $\kring$-algebra
\[
A (\calc ,N) := 
\bigoplus_{m,n \leq N} \calc (m,n),
\]
where $\calc (m,n)$ has degree $m-n$, with product induced by the composition of $\calc$.
\end{defn}

It follows that, as vector spaces, 
\begin{eqnarray}
\label{eqn:splitting}
A (\calc ,N +1) = 
A (\calc, N) \oplus 
\bigoplus_{n \leq N+1} \calc (N+1,n)
\end{eqnarray}
and $\bigoplus_{n \leq N+1} \calc (N+1,n)$ is a two-sided  ideal in $A (\calc ,N +1) $.

The following is clear:

\begin{lem}
\label{lem:cat_alg}
For $\calc$ satisfying Hypothesis \ref{hyp:calc} and $N \in \nat$, $A (\calc, N)$ is a finite-dimensional, $\nat$-graded $\kring$-algebra. Moreover, the splitting of equation (\ref{eqn:splitting}) induces a natural surjection $
A (\calc ,N+1) \twoheadrightarrow A (\calc , N)$ of $\nat$-graded $\kring$-algebras.
\end{lem}

\begin{rem}
The inclusion $A (\calc, N) \subseteq A (\calc, N+1)$ is a morphism of {\em nonunital} algebras that splits the surjection given in Lemma \ref{lem:cat_alg}. 
\end{rem}

\begin{defn}
\label{defn:calc_quad}
A category $\calc$ satisfying Hypothesis \ref{hyp:calc} is quadratic if, for all $N \in \nat$, the ring $A (\calc ,N) $ is quadratic.
\end{defn}

The  key examples here are given by:

\begin{exam}
It $\opd$ is a binary quadratic operad (see Section \ref{subsect:opd_kos_duality} below), then $\cat \opd$ is a quadratic category (cf. Proposition \ref{prop:quadratic_cat}).
\end{exam}

Likewise, the notion of a Koszul category generalizes that of a Koszul ring:

\begin{defn}
\cite[Definition 1.1.2]{MR1322847}
\label{defn:koszul_ring}
A Koszul ring is a non-negatively graded ring $ A = \bigoplus_{j \in \nat} A_j$ such that $A_0$ is semisimple and, considered as a graded left $A$-module,  admits a graded projective resolution 
\[
\ldots \rightarrow P^n \rightarrow \ldots \rightarrow P^2 \rightarrow P^1 \rightarrow P^0
\]
such that $P^i$ is generated by its degree $i$ component. 
\end{defn}

\begin{rem}
\label{rem:koszul_ring}
\ 
\begin{enumerate}
\item 
Any Koszul ring is quadratic \cite[Proposition 1.2.3]{MR1322847}.
\item 
By \cite[Proposition 2.1.3]{MR1322847},
the projective resolution condition is equivalent to $\ext^i_A (A_0, A_0\langle n \rangle)=0$ unless $i=n$, where $A_0\langle n \rangle $ denotes $A_0$ placed in degree $n$, considered as a graded left $A$-module.
\end{enumerate}
\end{rem}

For later use, we introduce notation for the Yoneda algebra associated to the Koszul algebra $A$, following \cite{MR1322847}:

\begin{nota}
\label{nota:yoneda_algebra}
For $A$ a Koszul algebra as above, denote by $E(A)$ the associated Yoneda algebra, which is $\nat$-graded, with degree $n$ component:
\[
E(A) _n := 
\ext^n _A (A_0, A_0\langle n \rangle) 
\]
and product given by Yoneda composition.
\end{nota}

\begin{defn}
\label{defn:cat_koszul}
A $\kring$-linear category $\calc$ satisfying Hypothesis \ref{hyp:calc} is Koszul if the ring $A (\calc, N)$ is Koszul for each $N \in \nat$. In particular, $\calc (m,m)$ is a finite-dimensional semisimple $\kring$-algebra for each $m \in \nat$.
\end{defn}

Hypothesis \ref{hyp:calc} together with the semisimplicity hypothesis gives the following, which generalizes properties of $\fopd$ considered in Section \ref{sect:rep}:

\begin{lem}
Suppose that $\calc$ is a $\kring$-linear category satisfying Hypothesis \ref{hyp:calc} and such that $\calc (m,m)$ is a finite-dimensional semisimple $\kring$-algebra for each $m \in \nat$. Then, for each $m \in \nat$, $\calc (m,m)$ is a semisimple $\calc$-module (in the sense of Definition \ref{defn:calc_modules}): it decomposes as a finite direct sum of simple $\calc$-modules. 

Moreover, if $M$ is a simple $\calc$-module, then there exists $m$ such that $M$ is a direct summand of $\calc (m,m)$ as a $\calc$-module. 
\end{lem}

The $\ext$-criterion for a Koszul ring (see Remark \ref{rem:koszul_ring}) adapts to the case of categories by considering the category of $\calc$-modules (cf. the approach of \cite{BGS} to Koszul abelian categories).

\begin{prop}
\label{prop:ext_criterion_calc}
Suppose that $\calc$ is a $\kring$-linear category satisfying Hypothesis \ref{hyp:calc} and such that $\calc (m,m)$ is a finite-dimensional semisimple $\kring$-algebra for each $m \in \nat$. 

Then $\calc$ is Koszul if and only if, for each $m,t  \in \nat$,  $\ext^i_{{_\calc}\modules} (\calc(m,m), \calc (t,t))=0$ unless $i=m-t$.
\end{prop}

\begin{proof}
This is a reformulation of the $\ext$-criterion for a graded ring $A$ with semisimple $A_0$, based upon the definition of the $\kring$-algebras $A (\calc, N)$ given in Definition \ref{defn:cat_alg}, which, in particular,  specifies the $\nat$-grading.  

Heuristically one considers that $A_0$ corresponds to the $\kring$-algebra $\prod_{m \in \nat} \calc (m,m)$. The $\ext$-criterion then splits to give the stated criterion, noting that the homological shift $\langle n \rangle$ corresponds to $m-t$.

The reader should provide details for themselves, in particular so as to understand the behaviour of the homological grading.  
\end{proof}

\subsection{The Koszul property for $\fopd$}

Suppose that $\opd$ is an operad that satisfies Hypothesis \ref{hyp:opd} and consider the category $\fopd$ of representations.
Note that, for $m \in \nat$, the endomorphism ring $\cat \opd (m,m)$ is isomorphic to $\kring [\sym_m]$, which is semisimple since $\kring$ is a field of characteristic zero. Moroever, Lemma \ref{lem:PROP_opd} implies that 
 $\cat \opd$ satisfies Hypothesis \ref{hyp:calc}.

Hence, motivated by Proposition \ref{prop:ext_criterion_calc}, we give the following definition, in which $\kring [\sym_m]$ (respectively $\kring [\sym_t]$)  is considered as an object of $\fopd$ with support $m$  (resp. $t$):

\begin{defn}
\label{defn:fopd_koszul}
The category $\fopd^{<\infty}$ is Koszul if,  $ \forall i \neq m -  t$, $
  \ext^i_{\fopd} (\kring [\sym_m], \kring[\sym_t])=0$.
\end{defn}

Proposition \ref{prop:ext_criterion_calc} implies that 
this is compatible with that for $\cat \opd$ given in the previous section:

\begin{prop}
\label{prop:kos_ab_vs_kos_cat}
The category $\fopd^{<\infty}$ is Koszul if and only if $\cat \opd$ is a Koszul  category. 
\end{prop}

\section{Koszul duality for operads}
\label{sect:koszul}

In this section, we review the Koszul duality theory for operads that is relevant to the applications in Part \ref{part:gr}. This is a return to the source, since the introduction of Koszul duality for binary quadratic operads by Ginzburg and Kapranov \cite{MR1301191} was based upon Koszul duality for quadratic categories. We note that the latter approach has been developed and generalized in recent work of Batanin and Markl (see \cite{2018arXiv181202935B} for example), although this is not used here.

\subsection{Koszul duality and resolutions}
\label{subsect:opd_kos_duality}

We review the seminal work of Ginzburg and Kapranov \cite{MR1301191}, presenting the details that are required for the applications. 

Recall from \cite[Section 7.1]{LV} the notion of a quadratic operadic  datum $(E, R)$. A datum is  {\em reduced} if  $E(0)=0$; it is  {\em binary} if $E(n)=0$ for $n \neq 2$.   

\begin{nota}(Cf. \cite[Section 7.2]{LV}.)
For $(E,R)$ a reduced quadratic datum,   denote by $s$ the homological suspension and 
\begin{enumerate}
\item 
$ \mathscr{P} (E, R)$ the associated quadratic operad; 
\item 
$\mathscr{P}^{\as} := \mathscr{P}(E, R)^{\as}$ the Koszul dual cooperad, i.e., the quadratic cooperad on the datum $(sE, s^2 R)$; 
\item 
$\mathscr{P}^!:= \mathscr{P} (E,R)^!$, the Koszul dual operad.
\end{enumerate}
\end{nota}

\begin{rem}
The Koszul dual operad $\mathscr{P}^!$ is the dual of the cooperad $\os^c \otimes_H \mathscr{P}^{\as}$, where $\otimes_H$ denotes the Hadamard tensor product (applied to cooperads) and $\os^c$ the cooperad dual to $\os$.

The structure of $\mathscr{P}^!$ is made explicit in the case of interest in Remark \ref{rem:koszul_dual_operad}.
\end{rem}

Clearly one has the following finiteness property, so that Lemma \ref{lem:PROP_opd} can be applied.

\begin{lem}
\label{lem:finiteness_quadratic operad}
For $(E,R)$ a reduced quadratic datum such that $E(n)$ has finite  dimension for all $n\in \nat$,  the operad $\mathscr{P}:=\mathscr{P}(E,R)$  is reduced and $\ppd (n)$ has finite dimension for all $n \in \nat$. 

Moreover, if $E(1)=0$, then $\mathscr{P} (1) = \kring$, concentrated in degree zero. 
\end{lem}

The construction of the relevant Koszul complex is given here using $\fb$-bimodules.

\begin{lem}
\label{lem:comonoid}
Suppose that $(E, R)$ is a reduced quadratic datum with $E(n)$ of finite dimension for all $n \in \nat$, with associated  Koszul dual cooperad $\ppd ^\as := \ppd^\as (E, R)$. The bimodule $\tc \ppd^\as$ has the structure of a counital comonoid in $(\sbimod, \otimes_\fb, \bimodunit)$, induced by the cooperad structure of $\ppd ^\as$.
\end{lem}

\begin{proof}
Since we are working over a field of characteristic zero, under the finiteness hypothesis, this can be deduced by using vector space duality from Proposition \ref{prop:bimod_monoid} applied to the dual operad $\opd:= (\ppd ^\as)^\sharp$.

By  Proposition \ref{prop:bimod_monoid}, the operad structure of $\opd$ makes $\cat \opd$ into a unital monoid in $(\sbimod, \otimes_\fb, \bimodunit)$ and, moreover,  the underlying bimodule of $\cat \opd$ is given by $\tc \opd$ (see Remark \ref{rem:cat_opd} (\ref{item:tc_opd}). On dualizing, one obtains the required structure.
\end{proof}

\begin{nota}
For $(E, R)$ a binary quadratic datum as above, write 
\begin{enumerate}
\item 
$\pbb$ for the unital monoid $ \cat \ppd (E, R) $ in $\fb$-bimodules; 
\item 
$\pbb ^\as$ for the counital comonoid structure given by Lemma \ref{lem:comonoid}. 
\end{enumerate}
\end{nota}

\begin{defn}
\label{defn:I_linear_N}
For $X \in \ob \sop$, let $(I; X)$ be the $\fb$-bimodule given as the 
direct summand of $\tc (I \oplus X) $ consisting of terms that are linear in $X$.
\end{defn}

\begin{lem}
\label{lem:generators_cogenerators}
Let  $(E, R) $ be a reduced quadratic datum. 
\begin{enumerate}
\item 
There is a canonical inclusion $(I; E) \rightarrow \pbb $ of $\fb$-bimodules. 
Hence the monoid structure of $\pbb$ induces a morphism of $\fb$-bimodules
$$ 
\pbb \otimes_\fb (I; E) \rightarrow \pbb 
$$
that is a morphism of left $\pbb$-modules in $\sbimod$. 
\item 
If  $E(n)$ has finite dimension for all $n \in \nat$, there is a canonical  
surjection $\pbb^\as 
\twoheadrightarrow (I; sE) \cong s (I; E)$ of $\fb$-bimodules. 
Hence the comonoid structure of $\pbb^\as$ induces a  morphism of $\fb$-bimodules:
$$
\pbb^\as  \rightarrow  (I; sE) \otimes_\fb \pbb^\as
$$
that is a morphism of right $\pbb^\as$-comodules in $\sbimod$.  
\end{enumerate}
\end{lem}

\begin{proof}
The second statement is dual to the first (up to the degree shift), so we concentrate on the case of $\pbb$. 

By construction of $\pbb$, there are inclusions $I \hookrightarrow \pbb (-, 1)$ and $E \hookrightarrow \pbb (-, 1)$. These induce a surjection $\tc (I \oplus E) \twoheadrightarrow \pbb$ in $\sbimod$. The morphism $(I;  E) \rightarrow \pbb$ is given by restricting to terms linear in $E$; this is a monomorphism. 

The structure morphism is the composite
\[
\pbb \otimes_\fb (I; E) \rightarrow \pbb \otimes_\fb \pbb \rightarrow \pbb,
\]
constructed using the monoid structure of $\pbb$. The associativity and unital properties of the monoid $\pbb$ imply that this is  a morphism of $\pbb$-modules in $\sbimod$.
\end{proof}

\begin{defn}
\label{defn:left_koszul_cx}
Let  $(E, R) $ be a reduced quadratic datum such that  $E(n)$ has finite  dimension for all $n \in \nat$. The left Koszul complex associated to $(E,R)$ is the complex in $\sbimod$
 given by $\pbb \otimes _\fb \pbb^\as$, equipped with the  differential given by the composite:
 \begin{eqnarray*}
 \pbb \otimes _\fb \pbb^\as  
 \rightarrow
 \pbb \otimes _\fb (I;sE)\otimes _\fb  \pbb^\as 
\cong 
s (\pbb \otimes _\fb (I;E)\otimes _\fb  \pbb^\as )
\rightarrow 
s ( \pbb  \otimes _\fb \pbb^\as ),
 \end{eqnarray*}
where the outer maps are induced by the structure morphisms of Lemma \ref{lem:generators_cogenerators}.
\end{defn}

\begin{rem}
\ 
\begin{enumerate}
\item 
Restricted to arities of the form $(-,1)$, this coincides with the operadic left Koszul complex $\mathscr{P} \circ \mathscr{P}^\as$  of \cite[Chapter 7]{LV}. 
\item 
The $\fb$-bimodule Koszul complex can be {\em constructed} from the operadic left Koszul complex by applying the functor $\tc$, imposing that $d$ be a derivation.
\end{enumerate}
\end{rem}

The  construction of the left Koszul complex implies:

\begin{prop}
\label{prop:left_module_right_comodule}
\ 
\begin{enumerate}
\item 
The left Koszul complex $( \pbb \otimes _\fb \pbb^\as, d)$ is a complex of left $\pbb$-modules and right $\pbb^\as$-comodules and these structures commute.
\item 
For $m \in \nat$, the $\pbb$-module $\pbb \otimes _\fb \pbb^\as (m, -)$ is projective of finite type.
\end{enumerate}
\end{prop}

\begin{proof}
The first statement follows from the module and comodule structures given in Lemma \ref{lem:generators_cogenerators}.

By definition,  $\pbb \otimes _\fb \pbb^\as (m, -) = \bigoplus_{0 \leq t \leq m} \pbb (t, -) \otimes_{\sym_t} \pbb^\as (m, t)$, where the indexing follows from the fact that $ \pbb^\as (m, t) = 0$ if $t >m$.  Now, for any $t\in \nat$, $\pbb^\as (m, t)$ has finite dimension. Since $\kring$ has characteristic zero,  the category of $\kring [\sym_t]$-modules is semisimple; the second statement thus follows from the fact that $\pbb (t, -)$ is a projective of finite type, by Proposition \ref{prop:proj_cat_opd-modules}.
\end{proof}

The form of the complex can be made more explicit when $(E, R)$ is binary, using the following observation:

\begin{lem}
\label{lem:binary_IE}
If $(E, R)$ is a binary quadratic datum, then $(I; E)$ is concentrated in arities of the form $(t+1, t)$,  $t \in \nat\backslash\{0\}$.
\end{lem}

This implies the following:

\begin{lem}
\label{lem:restrict_Koszul_diff}
Suppose that $(E, R)$ is binary   with $E(2)$ finite-dimensional. 
For $m,n,t \in \nat$ the Koszul differential of $(\pbb \otimes_\fb \pbb^\as, d)$ restricts to 
\[
\pbb (t,n) \otimes_{\sym_t} \pbb^{\as} (m,t) 
\rightarrow 
\pbb (t+1,n) \otimes_{\sym_{t+1}} \pbb^\as (m,t+1) 
\]
and is zero for $t=0$.

Moreover, if $E$ is concentrated in degree zero, the term $\pbb (t,n) \otimes_{\sym_t} \pbb^{\as} (m,t)$ is concentrated in homological degree $m-t$. 
\end{lem} 

Putting these results together gives:

\begin{prop}
\label{prop:koszul_complex}
Suppose that $(E, R)$ is binary,   with $E(2)$ finite-dimensional concentrated in degree zero. 
 Then, for $m \in \nat$, the Koszul complex $
(\pbb \otimes_\fb \pbb^\as )(m, -)$ gives a finite length complex 
of projective left $\pbb$-modules of finite type:
\begin{eqnarray*}
\pbb (1, -) \otimes_{\sym_1} \pbb^{\as}  (m, 1)
\rightarrow 
\pbb (2, -) \otimes_{\sym_2} \pbb^{\as}  (m, 2)
\rightarrow 
\ldots 
\rightarrow 
\pbb (m, -) \otimes_{\sym_m} \pbb^\as  (m, m)
\end{eqnarray*}
augmented by $\pbb (m,-) 
\twoheadrightarrow 
\kring [\sym_m]$ given by Proposition \ref{prop:proj_cat_opd-modules}.
\end{prop}

\begin{proof}
$\pbb^\as (m, t) =0$ if $t>m$, hence the Koszul complex restricts to a finite length complex  in left $\pbb$-modules, as indicated; the terms are finite type projectives by Proposition \ref{prop:left_module_right_comodule}.  The isomorphism $\pbb (m, -) \otimes_{\sym_m} \pbb^\as  (m, m)
\cong \pbb (m,-)$ follows from the identification $\pbb^\as  (m, m) \cong \kring [\sym_m]$.

The morphism $\pbb (m, -) \twoheadrightarrow \kring [\sym_m]$  given by Proposition \ref{prop:proj_cat_opd-modules} is a morphism of $\pbb$-modules. By  Lemma \ref{lem:restrict_Koszul_diff},  this defines an augmentation of the complex. 
\end{proof}

The above relates to the Koszul complex for quadratic categories, by using:

\begin{prop}
\label{prop:quadratic_cat}
\cite{MR1301191}
For $(E,R)$ a reduced, binary  quadratic datum  such that $E(2)$ is concentrated in degree zero, the  category  $\cat \mathscr{P} (E,R)$ is quadratic. 
\end{prop}

\begin{proof}
(Indications.) The category $\cat \ppd$ has generators $(I; E)$, as in the proof of Lemma \ref{lem:generators_cogenerators}. The quadratic relations are of two types:  commutation relations (analogous to those for partial compositions for an operad, cf. \cite[Section 5.3.4, equation (II)]{LV}) and the quadratic relations arising from $R$. 
\end{proof}

\begin{rem}
For $(E,R)$ a reduced, binary  quadratic datum  such that $E(2)$ is finite-dimensional concentrated in degree zero, the Koszul complex for $\pbb$  given by Definition \ref{defn:left_koszul_cx} is the Koszul complex of the quadratic category $\cat \ppd$. 
\end{rem}

\subsection{Relating to operadic Koszul duality}
Theorem \ref{thm:GK} below reformulates results of \cite{MR1301191} using the Koszul dual of a binary operad.

\begin{rem}
\label{rem:koszul_dual_operad}
Let  $\mathscr{P}:= \mathscr{P}(E,R)$ be a binary quadratic operad, where $E(2)$ is finite-dimensional, concentrated in degree zero. 
\begin{enumerate}
\item 
The Koszul dual cooperad $\mathscr{P}^\as (E, R)$ is the dual of the operad $\os \otimes_H \mathscr{P}^!$ (see \cite[Section 7.2.3]{LV}, in particular the proof of \cite[Proposition 7.2.1]{LV}).
\item 
The operad $\ppd^!$ is binary quadratic; explicitly $\ppd^! \cong \ppd (E^\vee, R^\perp)$, where $E^\vee = E^* \otimes \mathrm{sgn}_2$ and the orthogonal $R^\perp$ to $R$ is defined accordingly  (see \cite[Theorem 7.6.2]{LV}).
\end{enumerate}
\end{rem}

\begin{thm}
\label{thm:GK}
(Cf. \cite[Lemma 4.1.15]{MR1301191} and \cite[Section 2.9]{MR1322847}.) 
Let  $\ppd:= \ppd(E,R)$ be a binary quadratic operad, where $E(2)$ is finite-dimensional, concentrated in degree zero. Then the following conditions are equivalent: 
\begin{enumerate}
\item 
the operad $\ppd$ is Koszul; 
\item 
$\cat \ppd$ is Koszul;
\item 
for each $m \in \nat$, the complex of Proposition \ref{prop:koszul_complex} is exact, in particular yields a minimal projective resolution of $\kring [\sym_m]$ in left $\pbb$-modules. 
\end{enumerate}

If these conditions are satisfied, for $m,n \in \nat$, there is an isomorphism 
\[
\ext^{m-n} _{\fppd} (\kring [\sym_m], \kring [\sym_n] ) 
\cong 
\cat(\os \otimes_H \ppd^!) (m,n)
\] 
that is compatible with the composition. 
\end{thm}

\begin{proof}
The first statement follows from \cite{MR1301191}.

For the second part, one considers the category with set of objects $\nat$ and morphisms $\ext^{m-n} _{\fppd} (\kring [\sym_m], \kring [\sym_n] )$, with composition given by Yoneda composition. The statement asserts that this category is equivalent to $ \cat(\os \otimes_H \ppd^!)$. This follows from the argument of \cite[Theorem 2.10.1]{MR1322847}. 
\end{proof}

\begin{rem}
\ 
\begin{enumerate}
\item 
The category with morphisms $\ext^{m-n} _{\fppd} (\kring [\sym_m], \kring [\sym_n] )$ as above is sometimes referred to as the {\em Yoneda category} of $\cat \ppd$. It is the analogue of the Yoneda algebra $E(A)$ of a Koszul algebra $A$ introduced in Notation \ref{nota:yoneda_algebra}.
\item 
The statement of Theorem \ref{thm:GK} is compatible with the grading in the following sense: $\cat(\os \otimes_H \ppd^!) (m,n)$ is  in {\em homological} degree $n-m$, whereas $\ext^{m-n} _{\fppd} (\kring [\sym_m], \kring [\sym_n] ) $ is in {\em cohomological} degree $m-n = -(n-m)$.
\item 
 Duality in the operadic theory exploits the fact that $\fb$ is a groupoid,  so that $\fb \cong \fb\op$; this allows the usage of the opposite category as in the statement of \cite[Theorem 2.10.1]{MR1322847} to be avoided. This yields the variance in the statement of Theorem \ref{thm:GK}.
\end{enumerate}
\end{rem}

\begin{exam}
\label{exam:com_lie}
The operads $\com$ and $\lie$ are binary quadratic and Koszul dual, so that 
$
\com^! = \lie$, $\lie^! = \com$ (cf. \cite[Proposition 13.1.2]{LV}, for example). 

As in Section \ref{subsect:rep}, one has the associated categories $\fcom$ and $\flie$, together with their respective full subcategories $\fcom^{< \infty}$ and $\flie^{< \infty}$.  Theorem \ref{thm:GK} gives:
\begin{eqnarray*}
\ext^*_{\fcom} (\kring [\sym_m], \kring [\sym_n] ) & \cong &
\left\{
\begin{array}{ll}
\cat (\os \otimes_H \lie) (m,n) & *= m-n \\
0 & \mbox{otherwise;}
\end{array}
\right.
\\
\ext^*_{\flie} (\kring [\sym_m], \kring [\sym_n] ) & \cong &
\left\{
\begin{array}{ll}
\cat (\os \otimes_H \com) (m,n) & *= m-n \\
0 & \mbox{otherwise;}
\end{array}
\right.
\end{eqnarray*}
such that the associated Yoneda categories of $\fcom$ and $\flie$  identify respectively with $\cat (\os \otimes_H \lie)$ and $\cat (\os \otimes_H \com)$.
 In particular, $\fcom^{< \infty}$ and $\flie^{< \infty}$ are Koszul abelian categories.
 \end{exam}

\begin{exam}
\label{exam:res_com}
The Koszul resolution provided by Proposition \ref{prop:koszul_complex} and Theorem \ref{thm:GK} gives explicit projective resolutions in $\fcatk[\fs]$, since this is equivalent to $\fcom$. Namely, for $0< m \in \nat$, there is an explicit  projective resolution of $\kring [\sym_m]$ in $\fcatk[\fs]$:
\begin{eqnarray*}
\kring \fs (\mathbf{1}, -)  \otimes_{\sym_1} \mathbb{L}_m(1)
\rightarrow 
\kring \fs (\mathbf{2}, -)  \otimes_{\sym_2} \mathbb{L}_m (2)
\rightarrow 
\ldots 
\kring \fs (\mathbf{m-1}, -)  \otimes_{\sym_2} \mathbb{L}_m (m-1)
\rightarrow 
\kring \fs (\mathbf{m}, -),
\end{eqnarray*}
where the $\sym_i$-module $\mathbb{L}_m (i)$ is dual to $\cat (\os \otimes_H \lie) (m,i)$, in particular is determined by the Lie operad $\lie$. 
\end{exam}

\subsection{A recognition principle}
\label{subsect:recog}

In Section \ref{subsect:rep}, the category $\fopd\fn$ associated to the ungraded operad $\opd$ was introduced; it is equipped with the filtration given by the subcategories $\fopd^{\leq n} \subset \fopd$, $n \in \nat$. The purpose of this section is to spell out the  converse, in the Koszul setting, based upon quadratic duality.

Quadratic duality in the classical setting of an $\nat$-graded algebra $A$ (as in Section \ref{sect:kos_background}) implies that one can recover the algebra $A$ from its Yoneda algebra $E(A) = \ext^*_A (A_0, A_0 \langle * \rangle)$ (as introduced in Notation \ref{nota:yoneda_algebra}).  Explicitly, under the appropriate finiteness hypotheses, \cite[Theorem 2.10.2]{MR1322847}
asserts that, if $A$ is Koszul, then so is its Yoneda algebra $E(A)$ and there is a canonical isomorphism 
\[
E (E(A)) \cong A.
\]
Hence, if $A$, $A'$ are two Koszul algebras, with $A_0 = A_0'$ and satisfying the requisite finiteness hypotheses, if $E(A)\cong E(A')$ as $A_0$-algebras, then $A \cong A'$.

This argument  adapts to the setting of $\kring$-linear categories, as outlined in Section \ref{sect:kos_background}.

\begin{nota}
Let $\f$ be a $\kring$-linear abelian category and $\{ P_n \ |\ n \in \nat\}$ be a family of projectives of $\f$, such that $\hom_\f (P_m, P_n)=0$ if $m >n$. For $n \in \nat$, denote by $\widetilde{P_n}$  the cokernel of the morphism 
induced by evaluation:
\[
\Big(\bigoplus_{\{  f \in \hom_{\f} (P_t, P_n) \ | \ t <n  \}} P_t \Big)
\rightarrow 
P_n
 \]
(the map on the summand indexed by $f$ is $P_t \stackrel{f}{\rightarrow} P_n$).
\end{nota}

\begin{thm}
\label{thm:recog}
Let  $\f$ be a $\kring$-linear category such that:
\begin{enumerate}
\item 
all objects of $\f$ have a finite composition series; 
\item
there is a  family of projective generators $\{ P_n \ | \ n \in \nat\}$ of $\f$, such that $\mathrm{End}_\f (P_n) \cong \kring [\sym_n]$ and $\hom_\f (P_m, P_n)=0$ if $m >n$;
\item 
there exists a binary Koszul operad $\ppd$ such that 
\[
\ext^*_\f (\widetilde{P_m}, \widetilde{P_n}) = 
\left\{ 
\begin{array}{ll}
0 & * \neq m-n \\
\cat (\os \otimes_H \ppd^!) (m,n) & * = m-n,
\end{array}
\right.
\]
with Yoneda composition corresponding to composition in $\cat (\os \otimes_H \ppd^!)$. 
\end{enumerate} 
Then  $\ppd(2)$ has finite dimension and there is an equivalence of categories $\f \cong \fppd\fn$. 
\end{thm}

\begin{proof}
Let $\calc$ be the opposite of the full subcategory of $\f$ with objects $\{ P_n \ | \ n \in \nat \}$, so that $\calc (m,n) = \hom_\f (P_n, P_m)$ (indexing the objects of $\calc$ by $\nat$). Then $\f$ is equivalent to the full subcategory of finite objects in ${_\calc}\modules$, by Freyd's theorem (see \cite[Theorem 3.1]{MR294454} for a version over $\zed$). 

By the hypotheses, the $\kring$-linear category has the following properties:  for $m \in \nat$, $\calc (m,m) \cong \kring [\sym_m]$ (using that the inverse induces an isomorphism of algebras $\kring [\sym_m]\op \cong \kring [\sym_m]$), $\calc (m,n) =0$ for $m<n$ and $\calc (m,n)$ is always finite-dimensional.

Moreover, under the equivalence between $\f$ and the full subcategory of finite objects of ${_\calc}\modules$, $\widetilde{P_m}$ corresponds to the $\calc$-module $\kring [\sym_m]$ supported on $m$.

The $\ext$ hypothesis implies, in particular, that the $\kring$-linear category $\calc$ is Koszul. Moreover,  the finiteness hypotheses imply that the binary quadratic operad $\ppd$ has finite-dimensional space of generators $\ppd (2)$. 

Now Theorem \ref{thm:GK} shows that the quadratic category $\cat \ppd$ has isomorphic Yoneda algebra to that of $\calc$. By quadratic duality  (cf. \cite[Theorem 2.10.2]{MR1322847}), as outlined at the beginning of this subsection, this implies that 
$\calc$ is isomorphic to $\cat \ppd$ and hence $\f$ is equivalent to $\fppd\fn$.
\end{proof}

The proof of Theorem \ref{thm:recog} identifies $\cat \ppd$ with the opposite of the full subcategory of $\f$ with objects the $P_n$, $n \in \nat$, with the associated $\kring$-linear functor 
\begin{eqnarray}
\label{eqn:ff_embedding}
P_\bullet : (\cat \ppd)\op \hookrightarrow \f.
\end{eqnarray}

\begin{cor}
\label{cor:Morita_functor}
Suppose that the hypotheses of Theorem \ref{thm:recog} are satisfied. Then the  equivalence of categories $\fppd \fn \stackrel{\cong}{\rightarrow } \f$ is induced by the functor 
$
F \mapsto P_\bullet \otimes _{\cat \ppd} F$.
In particular, this sends $P^{\ppd}_n$ to $P_n$, for $n \in \nat$.  
\end{cor}

\part{Applications to functors on free groups}
\label{part:gr}

\section{Functors on $\gr$}
\label{sect:gr}

The purpose of this section is to review the structure of covariant and contravariant functors from the category $\gr$ of finitely-generated free groups to $\kring$-vector spaces. Throughout  $\kring$ is taken to be a field of characteristic zero; part  of the theory holds over more general commutative rings. The principal interest is in the category of  polynomial functors on $\gr\op$ and the associated category $\f_\omega (\gr\op; \kring)$ of  analytic functors. The notion of polynomial functor generalizes that introduced by Eilenberg and Mac Lane for functors on an additive category \cite{MR65162}; the case of polynomial functors on $\gr\op$ (as opposed to $\gr$) is not as readily available in the literature, hence is treated in somewhat greater detail.  

Section \ref{subsect:further_grop} establishes the properties that are required in Section \ref{sect:gr_kos} to show that the category $\f_\omega (\gr\op; \kring)$ is Koszul.

\begin{nota}
Denote by $\gr$ the category of finitely-generated free groups, which is essentially small, with skeleton having objects $\fg_n$, the free group on $n$ generators, for $n \in \nat$.
\end{nota}

The following standard result relates contravariant and covariant functors on $\gr$:

\begin{prop}
\label{prop:co_contra}
\ 
\begin{enumerate}
\item
There is an equivalence of categories
$ 
\fcatk[\gr\op] \op 
\cong 
\mathrm{Func} (\gr, \kmod\op), 
$
the category of functors with values in the opposite of $\kmod$.
\item 
Vector space duality $(-)^\sharp$ induces an adjunction 
$
\fcatk[\gr]\op 
\rightleftarrows
\fcatk[\gr\op]
$ 
that restricts to an equivalence between the full subcategories of functors taking finite-dimensional values. 
\end{enumerate}
\end{prop}

\begin{rem}
Explicitly, for $F \in \ob \fcatk[\gr]$ and $G \in \ob \fcatk[\gr\op]$, the duality adjunction provides the natural isomorphism 
$
\hom_{\fcatk[\gr]} (F, G^\sharp) 
\cong 
\hom_{\fcatk[\gr\op]} (G, F^\sharp).
$
\end{rem}

The following introduces the fundamental example of a non-constant polynomial functor on $\gr$:

\begin{defn}
\label{defn:abelianization}
Let $\A \in \ob \fcatk[\gr]$ be the abelianization functor that sends a free group $G$ to $(G/[G,G])\otimes_\zed \kring$. 
\end{defn}

\begin{rem}
\ 
\begin{enumerate}
\item 
The abelianization functor is additive, since $\A (G_1 \star G_2) \cong \A (G_1) \oplus \A (G_2)$. Likewise for $\A^\sharp$.
\item 
The dual functor $\A^\sharp \in \ob \fcatk[\gr\op]$ identifies as $G \mapsto \hom_{\mathrm{Group}} (G, \kring) \cong \hom_{\mathrm{Group}} (G/[G,G], \kring) $. 
\item 
The functors $\A$ and $\A^\sharp$ take finite-dimensional values; in particular $\A \cong (\A^\sharp)^\sharp$.
\end{enumerate}
\end{rem}

\subsection{Polynomial functors on $\gr$}
\label{subsect:poly_gr}

The Eilenberg-MacLane notion of polynomial degree, defined in terms of cross-effects, applies to functors on $\gr$, using the monoidal structure $(\gr , \star, e)$ provided by the free product of groups, $\star$.  (See \cite[Section 3]{MR3340364} for generalities and  \cite[Section 1]{DPV} for the specific case of $\gr$.)

\begin{nota}
 For $d \in \nat$, the full subcategory of functors of polynomial degree at most $d$ is denoted $\fpoly{d}{\gr}$, so that there are inclusions $\fpoly{d}{\gr} \subset \fpoly{d+1}{\gr} \subset \fcatk[\gr]$. The full subcategory of polynomial functors is $\fpi{\gr} := \bigcup _{d \in \nat} \fpoly{d}{\gr}$.
\end{nota}

\begin{exam}
\label{exam:A_otimes_d}
For $d \in \nat$, $\A^{\otimes d}$ is polynomial of degree $d$ and is not polynomial of degree $d-1$.
\end{exam}

\begin{prop}
\label{prop:A_otimes_d_properties}
\cite{DPV}
For $d \in \nat$,
\begin{enumerate}
\item
 the place permutation action on $\A^{\otimes d}$ induces an isomorphism 
$
\mathrm{End}_{\fpoly{d}{\gr}}(\A^{\otimes d}) \cong \kring [\sym_d];
$
\item 
$\A^{\otimes d}$ is projective in $\fpoly{d}{\gr}$ and the functor $\hom_{\fpoly{d}{\gr}} (\A^{\otimes d}, -)$ is an exact functor to $\kring [\sym_d]$-modules;  this  identifies with $\cre_d$, the $d$th cross-effect functor.
\end{enumerate}
\end{prop}

For current purposes, the reader can take the $d$th cross-effect $\cre_d : \fpoly{d}{\gr}
 \rightarrow 
 \kring[\sym_d]\dash \modules$ to be defined to be the functor $\hom_{\fpoly{d}{\gr}} (\A^{\otimes d}, -)$.

 \begin{defn}
  \label{defn:beta}
 \cite[Section 4]{DPV}
 For $d \in \nat$, let $\beta_d :  \kring[\sym_d]\dash \modules \rightarrow \fpoly{d}{\gr}$ be the right adjoint to the $d$th cross-effect $
 \cre_d : \fpoly{d}{\gr}
 \rightarrow 
 \kring[\sym_d]\dash \modules.
$
 \end{defn}

\begin{prop}
\label{prop:beta_DPV}
\cite[Section 4]{DPV}
For $d \in \nat$,
\begin{enumerate}
\item
the functor $\beta_d$ is exact and preserves injectives;
\item 
if $M$ is a finite-dimensional $\kring [\sym_d]$-module, the functor $\beta_d M$ takes finite-dimensional values and has a finite composition series. 
\end{enumerate}
\end{prop}

Moreover, this yields a set of injective cogenerators of the category  $\fpi{\gr}$ of polynomial functors (without bounded polynomial degree):

\begin{prop}
\label{prop:beta_injective}
\cite{PV}
The set $\{ \beta_d \kring [\sym_d] \  | \ d \in \nat \}$ forms a set of injective cogenerators of $\fpi{\gr}$.
\end{prop}

\begin{rem}
Although the proofs of these properties do not require the full force of the Koszul property, they do use the cohomological degrees $0$ and $1$ cases of Theorem \ref{thm:main_Vext} below, corresponding to (\ref{eqn:kos_property}) of the Introduction. 
\end{rem}

\subsection{Polynomial functors on $\gr\op$}
\label{subsect:poly_grop}

Polynomiality for functors of $\gr\op$ is defined using the monoidal structure $(\gr\op, \star, e )$,  as in \cite[Section 3]{MR3340364}.

\begin{rem}
There is an alternative approach, by considering  polynomial functors from $\gr$ to $\kmod \op$. That these approaches are equivalent is established by  \cite[Proposition 3.13]{MR3340364}.
\end{rem}

For current purposes, the following definition is preferred to that given in \cite[Definition 3.9]{MR3340364} (for the equivalence between these approaches, use  \cite[Proposition 3.11]{MR3340364}).

\begin{defn}
\label{defn:tre}
For $d \in \nat$; 
\begin{enumerate}
\item 
let $\tre_d : \fcatk[\gr\op] \rightarrow \fcatk[(\gr\op)^{\times d}]$ be the functor given by 
\[
\tre_d F (G_1, \ldots , G_d) := 
\mathrm{Coker}
\Big [
\bigoplus _{k=1}^d
F(G_1 \star \ldots \star \ldots \star G_{k-1} \star \{e\} \star G_{k+1} \star \ldots \star G_d) 
\rightarrow 
F (  G_1 \star \ldots  \star G_d) 
\Big ],
\]
where the morphism is induced by the projections $G_k \twoheadrightarrow \{e\}$ for $1 \leq k \leq d$;
\item 
a functor $F \in \ob \fcatk[\gr\op]$ is polynomial of degree at most $d$ if $\tre_{d+1}F =0$; 
\end{enumerate}
\end{defn}

The full subcategory of functors of polynomial degree at most $d$ is denoted by $\f_d (\gr \op; \kring)$. By construction there is an increasing filtration:
\[
\f_0 (\gr\op; \kring) 
\subset
\f_1 (\gr\op; \kring)
\subset 
\ldots 
\subset 
\f_d (\gr \op; \kring)
\subset 
\f_{d+1}(\gr\op; \kring)
\subset 
\ldots 
\subset 
\fcatk[\gr\op].
\]
Dually to Example \ref{exam:A_otimes_d}, for $d \in \nat$,  $(\A^\sharp)^{\otimes d}$ is polynomial of degree exactly $d$, hence the above inclusions above are strict. Moreover, Proposition \ref{prop:A_otimes_d_properties} implies:

\begin{lem}
\label{lem:end_Ad_dual}
For $d\in \nat$, there is an isomorphism of rings $\mathrm{End}_{\fcatk[\gr\op]} ((\A^\sharp)^{\otimes d}) \cong \kring [\sym_d]$. 
\end{lem}

One has the dual to the $d$th Taylorization construction of \cite[Proposition 3.17]{MR3340364}:

\begin{prop}
\label{prop:pgrop}
For $d \in \nat$, the inclusion $\f_d (\gr \op; \kring) \rightarrow \fcatk[\gr\op]$ admits a right adjoint $\pgrop_d : \fcatk[\gr\op] \rightarrow \f_d (\gr \op; \kring)$ that is  given  by 
\[
\pgrop _d F (G)
:= 
\ker 
\Big (
F (G) 
\rightarrow 
F (G^{\star d+1}) 
\rightarrow 
\tre_{d+1} F (G, \ldots, G)
\Big ),
\]
where the first morphism is induced by the fold map $G^{\star d+1} \rightarrow G$ and the second is the canonical projection.
\end{prop}

It follows that a functor $F \in \ob \fcatk[\gr\op]$ admits a natural filtration 
\[
\pgrop_0 F 
\subset 
\pgrop_1 F
\subset 
\ldots 
\subset 
\pgrop_d F 
\subset
\pgrop_{d+1} F
\subset
\ldots 
\subset 
F.
\]
This leads to the notion of an analytic functor on $\gr\op$:

\begin{defn}
\label{defn:analytic}
A functor $F \in \ob \fcatk[\gr\op]$ is analytic if $F \cong \lim_{\rightarrow} \pgrop_d F$. 
  The full subcategory of analytic functors is denoted  by $\f_\omega (\gr\op; \kring)$.
\end{defn}

Duality extends to  polynomial functors, using the results of \cite[Section 3.2]{MR3340364}: 

\begin{prop}
\label{prop:duality_gr_poly}
For $d \in \nat$, vector space duality induces an adjunction:
\[
\f_d (\gr ; \kring)\op 
\rightleftarrows 
\f_d (\gr\op; \kring) 
\]
that restricts to an equivalence of categories between the full subcategories of  functors taking finite-dimensional values.
\end{prop}

One has the following analogue for $\gr\op$ of the cross-effect functor $\cre_d $:

\begin{defn}
\label{defn:tr}
For $d \in \nat$, let $\tr_d : \fpoly{d}{\gr\op} \rightarrow \kring [\sym_d]\dash \modules$ be the composite of the restriction of $\tre_d$ to $\fpoly{d}{\gr\op}$ with the evaluation on $(\fg_1, \ldots , \fg_1) \in (\gr\op)^{\times d}$, where the $\sym_d$-action is given by permuting the factors.
\end{defn}

\begin{prop}
\label{prop:properties_tr}
For $d \in \nat$,  
\begin{enumerate}
\item 
the functor $\tr_d : \fpoly{d}{\gr\op} \rightarrow \kring [\sym_d]\dash \modules$ is exact and has kernel $\fpoly{d-1}{\gr\op}$;
\item 
for $F \in \ob \fpoly{d}{\gr\op}$, there is a natural isomorphism:
\[
\hom_{\fpoly{d}{\gr\op}} (F, (\A^\sharp)^{\otimes d}) 
\cong 
(\tr_d F)^\sharp
\]
of $\sym_d$-modules. In particular, $(\A^\sharp)^{\otimes d}$ is injective in $ \fpoly{d}{\gr\op}$.
\end{enumerate}
\end{prop}

\begin{proof}
The first statement is clear from the definition of $\tr_d$. The second follows by  analysing the relationship between $\cre_d$ and $\tr_d$, using  the duality adjunction of Proposition \ref{prop:duality_gr_poly}.
\end{proof}

\begin{prop}
\label{prop:right_adjoint_tr}
For $d \in \nat$, the functor $\tr_d : \fpoly{d}{\gr\op} \rightarrow \kring [\sym_d]\dash \modules$ has right adjoint given by 
\[
M 
\mapsto 
((\A^\sharp)^{\otimes d} \otimes M)^{\sym_d}, 
\]
where $\sym_d$ acts diagonally. This functor is exact and $((\A^\sharp)^{\otimes d} \otimes M)^{\sym_d}$ is semisimple of polynomial degree $d$.

For $F \in \ob \fpoly{d}{\gr\op}$ there is a natural short exact sequence 
\[
0
\rightarrow 
\pgrop_{d-1} F 
\rightarrow F
\rightarrow 
((\A^\sharp)^{\otimes d} \otimes \tr_d F)^{\sym_d}
\rightarrow 
0.
\]
\end{prop}

\begin{proof}
The first statement follows from Proposition \ref{prop:properties_tr}. 

Since $\kring$ has characteristic zero,  $\A^{\otimes d}$ is a finite direct sum of simple functors of $\fpoly{d}{\gr}$ of degree exactly $d$ (cf. \cite{PV}). This is proved by using the fact that 
$\hom _{\fcatk[\gr]} (\A^{\otimes m}, \A^{\otimes n})$ is non-zero if and only if $m=n$, when it is isomorphic to $\kring [\sym_n]$ (see Theorem \ref{thm:main_Vext} below). 

 By duality, one has the analogous statement for $(\A^\sharp)^{\otimes d}$ and hence for $((\A^\sharp)^{\otimes d} \otimes M)^{\sym_d}$, again using that $\kring$ is of characteristic zero. The latter  is exact as a functor of $M$.

The  kernel of $F
\rightarrow 
((\A^\sharp)^{\otimes d} \otimes \tr_d F)^{\sym_d}$ is $\pgrop_{d-1}F$,  by definition. It remains to establish the surjectivity of the morphism. This can be checked after applying $\tr_d$, where it is clear, using that  $\tr_d ( (\A^\sharp)^{\otimes d}) \cong \kring [\sym_d]$, as follows from Lemma \ref{lem:end_Ad_dual}.
\end{proof}

\subsection{Further properties of polynomial and analytic functors on $\gr\op$}
\label{subsect:further_grop}

\begin{prop}
\label{prop:pgrop_exact}
For $d \in \nat$, the restriction $\pgrop_d : \f_\omega (\gr\op ; \kring ) \rightarrow \fpoly{d}{\gr\op}$ of the functor $\pgrop_d$ to the subcategory of analytic functors is exact. 
\end{prop}

\begin{proof}
The functor $\pgrop_d$ is defined as the right adjoint to the inclusion $\fpoly{d}{\gr\op} \subset \fcatk[\gr\op]$, thus is left exact. It remains to show that it preserves surjections between analytic functors.

Let $F_1 \twoheadrightarrow F_2$ be a surjection in $\f_\omega(\gr\op; \kring)$. Then, the cokernel of $\pgrop_d (F_1) \rightarrow \pgrop_d (F_2)$ has polynomial degree $d$. By the snake lemma,  it is a subobject of $F_1/ \pgrop_d (F_1)$.

Now, by hypothesis, $F_1 := \lim_\rightarrow \pgrop_n (F_1)$. In particular, $F_1/ \pgrop_d (F_1)$ has an increasing, exhaustive filtration with subquotients $\pgrop_n (F_1)/\pgrop_{n-1} (F_1)$, $n >d$.  Proposition \ref{prop:right_adjoint_tr} implies that the subquotient indexed by $n$ is semisimple with composition factors of polynomial degree exactly $n$. 
 It follows that $F_1/ \pgrop_d (F_1)$ contains no non-trivial subquotient of polynomial degree $d$. Thus,  $\pgrop_d (F_1) \rightarrow \pgrop_d (F_2)$ is surjective, as required.
\end{proof}

Propositions \ref{prop:properties_tr} and \ref{prop:pgrop_exact} imply the following Corollary, which introduces the functor $\gamma_d$ that identifies the $d$th layer of the polynomial filtration.

\begin{cor}
\label{cor:gamma}
For $d \in \nat$, the functor $\gamma_d := \tr_d \circ \pgrop_d : \f_\omega (\gr\op; \kring) \rightarrow \kring [\sym_d]\dash\modules$ is exact. 
\end{cor}

The following is key to understanding the finiteness properties of $\f_\omega (\gr\op; \kring)$:

\begin{prop}
\label{prop:fpoly_lf}
For $d \in \nat$, the category $\fpoly{d}{\gr\op}$ is locally finite. 
\end{prop}

\begin{proof}
Since the category of $\kring[\sym_n]$-modules is locally finite for all $n \in \nat$, using induction upon $d$, one reduces to establishing the following statement: for $F \in \ob \fpoly{d} {\gr \op}$ such that $Q:= F/p_{d-1} F$ is finite, there exists a finite subfunctor $G\subset F$ such that $G/p_{d-1} G \cong Q$. 

Dualizing the defining short exact sequence gives $0\rightarrow Q^\sharp \rightarrow F^\sharp \rightarrow (p_{d-1}F)^\sharp \rightarrow 0$ in $\fpoly{d}{\gr}$. Take $Q^\sharp \hookrightarrow I$ to be  the injective envelope of $Q^\sharp$; explicitly, $I = \beta_d \cre_d (Q^\sharp)$ and the inclusion is given by the adjunction unit. The functor $I$ is finite in $\fpoly{d}{\gr}$ and $I/(Q^\sharp)$ lies in $\fpoly{d-1}{\gr}$. 

By injectivity of $I$ in $\fpoly{d}{\gr}$, one obtains the extension: 
\[
\xymatrix{
Q^\sharp 
\ar@{^(->}[r]
\ar@{^(->}[d]
&
F^\sharp 
\ar@{.>}[ld]^f 
\\
I.
}
\]
Consider $K:= \ker f$; since $Q^\sharp \cap \ker f= 0$, the composite $K \rightarrow F^\sharp \rightarrow (p_{d-1}F)^\sharp$ is injective; it has finite cokernel, since the image of $f$ is finite. 

The adjoint morphism $p_{d-1}F \rightarrow K^\sharp$    fits into the commutative diagram in $\fpoly{d}{\gr\op}$:
\[
\xymatrix{
p_{d-1}F 
\ar@{^(->}[d]
\ar[rd]
\\
((p_{d-1}F)^\sharp) ^\sharp
\ar[r]
&
K^\sharp
}
\]
in which the horizontal arrow is the dual of $K \rightarrow (p_{d-1}F)^\sharp$ and the vertical arrow is the adjunction unit, which is a monomorphism. Since the horizontal morphism has finite kernel, so does  $p_{d-1}F \rightarrow K^\sharp$.

The adjoint to the composite $K \hookrightarrow F^\sharp \rightarrow (p_{d-1} F)^\sharp$, gives the commutative diagram 
\[
\xymatrix{
p_{d-1} F 
\ar[r]
\ar[rd]
&
F 
\ar[d]^g
\\
&
K^\sharp,
}
\]
where $g$ is the adjoint to $K \hookrightarrow F^\sharp$. Set $G:= \ker g \subset F$. By the above, $G \cap p_{d-1} F$ is finite, hence $G$ is finite, since $G/p_{d-1}G$ is a subfunctor of $Q$, which is finite, by hypothesis.

To complete the proof, it remains to show that the composite $G \rightarrow F \rightarrow Q$ is surjective. By construction $K$ embeds in $(p_{d-1}F)^\sharp$, hence has polynomial degree $d-1$; thus $K^\sharp$ has polynomial degree $d-1$. The cokernel $ C:= \mathrm{coker} [G \rightarrow Q]$ is a quotient of the image of $g$, which has polynomial degree $d-1$ as a subfunctor of $K^\sharp$.
 Thus $C = p_{d-1}C$; however, $p_{d-1}C=0$,  since $p_{d-1} Q=0$ by construction and $p_{d-1}$ is exact by Proposition \ref{prop:pgrop_exact}. The result follows.
\end{proof}

\begin{cor}
\label{cor:fpoly_omega_lf}
The category $\f_\omega (\gr\op ; \kring)$ is locally finite. 
\end{cor}

\begin{proof}
By definition, every object $F \in \ob \f_\omega (\gr\op ; \kring)$ is the colimit of its subfunctors $(p_{d} F)_{d \in \nat}$. Since each category $\fpoly{d}{\gr\op}$ is locally finite, by Proposition \ref{prop:fpoly_lf}, the result follows. 
\end{proof}

The following notation is introduced for typographical clarity:

\begin{nota}
\label{nota:bu}
For $d \in \nat$, write $\bu_d \ :\ \big(\kring [\sym_d]\dash\modules\big) \op \rightarrow \fcatk[\gr\op]$ for the functor $M \mapsto \bu _d M:= (\beta_d M)^\sharp$. 
\end{nota}

The following Corollary uses the functor $\gamma_d$ introduced in Corollary \ref{cor:gamma}:

\begin{cor}
\label{cor:proj_gen_anal_grop}
For $d \in \nat$, there is a natural right action of $\sym_d$ on $\bu_d\kring[\sym_d]$ and, 
 for $F \in \ob \f_\omega (\gr\op ; \kring)$, there is a natural $\sym_d$-equivariant isomorphism
\begin{eqnarray}
\label{eqn:corep_gamma}
\hom_{\f_\omega (\gr\op; \kring) } (\bu_d\kring[\sym_d], F) 
\cong 
\gamma_d F.
\end{eqnarray}

In particular, the category $\f_\omega (\gr\op ; \kring)$ has set of small projective generators $\{ \bu_d \kring [\sym_d] | d \in \nat \}$.
\end{cor}

\begin{proof}
The right $\sym_d$-action on $\bu_d\kring[\sym_d]$ is induced by the regular representation of $\sym_d$. 

The functor  $\bu_d\kring[\sym_d]$ is polynomial of degree $d$. Since $p_d$ is exact by Proposition \ref{prop:pgrop_exact}, to establish projectivity, one reduces to showing that it is projective in $\fpoly{d}{\gr\op}$. By Proposition \ref{prop:beta_DPV}, $\bu_d\kring[\sym_d]$ is a finite functor, hence takes finite-dimensional values. Thus, by duality, Proposition \ref{prop:beta_DPV} implies that it is projective in the full subcategory of $\fpoly{d}{\gr\op}$ of functors taking finite-dimensional values and, more precisely, the isomorphism (\ref{eqn:corep_gamma}) holds for $F$ in this full subcategory. 

Since the category $\fpoly{d}{\gr\op}$ is locally finite by Proposition \ref{prop:fpoly_lf}, this implies projectivity in $\fpoly{d}{\gr\op}$ and the isomorphism (\ref{eqn:corep_gamma})  is given by passage to the colimit of the diagram of finite subfunctors of $F$.

It remains to show that $\{ \bu_d \kring [\sym_d] | d \in \nat \}$ forms a set of projective generators of $\f_\omega (\gr\op ; \kring)$. Consider a functor $F \in \ob \f_\omega (\gr\op; \kring)$; one has the associated sequence of $\sym_d$-modules $(\gamma_d F) _{d\in \nat}$. Hence one can form 
 \[
 \bigoplus_{d\in \nat} \bu_d\kring[\sym_d] \otimes_{\sym_d} \gamma _d F
 \rightarrow 
 F,
 \]
 where the morphism is constructed by using the isomorphism (\ref{eqn:corep_gamma}). One checks that this is surjective. Since $\kring$ has characteristic zero, the domain is projective, which implies the result.
\end{proof}

The above establishes the counterpart of the defining property of $\beta_d$ (see Definition \ref{defn:beta}):

\begin{cor}
\label{cor:left_adj_gamma}
For $d \in \nat$, the functor $\gamma_d : \f_\omega (\gr\op ; \kring) \rightarrow \kring [\sym_d]\dash\modules $ has exact left adjoint:
\[
M 
\mapsto 
\bu_d \kring [\sym_d] \otimes_{\sym_d} M.
\]  
\end{cor}

\section{Koszulity of analytic functors on $\gr\op$}
\label{sect:gr_kos}

The purpose of this section is to show that the category $\f_\omega (\gr\op ; \kring)$ is Koszul. More precisely, Theorem \ref{thm:analytic_grop_flie} shows that it is equivalent to $\flie$, the category of representations of the Lie operad.
 From this, using vector space duality, the corresponding covariant result (restricted to {\em finite} polynomial functors,  $\fart$) is deduced  in Corollary \ref{cor:finite_poly_gr}. These results are made more concrete in Section \ref{sect:pbw}, where an alternative proof is provided which does not depend explicitly upon operadic Koszul duality.

Using the properties of $\f_\omega (\gr\op ; \kring)$ established in Section \ref{subsect:further_grop}, Theorem \ref{thm:analytic_grop_flie} is a consequence of two cohomological results that are recalled in Section \ref{subsect:cohom_input}; the first relates the calculation of $\ext$ in a category of polynomial functors on $\gr$ with that in $\fcatk[\gr]$ and the second is the Koszul property for $\ext$. 

\subsection{Cohomological input}
\label{subsect:cohom_input}

The main result of \cite{DPV} gives the following important relationship between the homological properties of $\fcatk[\gr]$ and of its subcategories of polynomial functors (the result given in \cite{DPV} is more general:  it holds over any commutative ring).

\begin{thm}
\label{thm:DPV_poly}
\cite[Théorème 1]{DPV}
For  $F, G \in \ob \f_d (\gr; \kring)$, $d \in \nat$, the inclusion $\f_d (\gr; \kring) \subset \fcatk[\gr]$ induces an isomorphism
\[
\ext^* _{\fpoly{d}{\gr}} (F, G ) 
\rightarrow 
\ext^* _{\fcatk[\gr]} (F, G). 
\] 
\end{thm}

\cite[Theorems 1 and  2]{V_ext} (where the ground ring is taken to be $\zed$) imply the following: 

\begin{thm}
\label{thm:main_Vext}
For $m,n \in \nat$, there are isomorphisms of $\kring$-modules:
\[
\ext^* _{\fcatk[\gr]} (\A^{\otimes n}, \A ^{\otimes m} ) 
\cong 
\left\{
\begin{array}{ll}
0 & * \neq m-n \\
\kring \fs (\mathbf{m}, \mathbf{n} ) & * = m-n .
\end{array}
\right.
\]

The tensor product of $\fcatk[\gr]$ and Yoneda composition induce a graded PROP structure on  $\ext^{m-n} _{\fcatk[\gr]} (\A^{\otimes n}, \A ^{\otimes m} )$ with respect to the cohomological grading. This PROP is freely generated by the graded operad $\scom$ with $ \scom (m):= \ext^{m-1} _{\fcatk[\gr]} (\A, \A ^{\otimes m} )$. 
\end{thm}

The following identification of the  graded operad $\scom$ appearing in Theorem \ref{thm:main_Vext} is implicit in \cite{V_ext} and is established in \cite{2021arXiv210514497K} (working over a more general ring). For the reader's convenience, a quick proof is indicated, working over $\kring$. 

\begin{prop}
\label{prop:identify_opd_caly}
 The graded operad $\scom$ (with homological grading) is isomorphic to $\os \otimes_H \com$. 
\end{prop}

\begin{proof}
Theorem \ref{thm:main_Vext} establishes a Koszul-type property for the category $\fcatk[\gr]$ (to apply the results of Section \ref{sect:kos_background} directly, and thus use the terminology {\em Koszul}, one should restrict to working with finite, polynomial functors). In particular, this implies that the Yoneda algebra is quadratic. It follows that the operad $\scom$ is binary quadratic.   

In \cite{V_ext} it is shown that $\scom(m)$ is the signature representation of $\sym_m$, for all $m \in \nat$. Hence the operad $\os^{-1} \otimes _H \scom$ is an {\em ungraded} operad, which is the trivial representation in each arity. 
 The only such binary quadratic operad is the commutative operad, $\com$ (cf. the proof of \cite[Proposition 13.1.1]{LV}).
\end{proof}

 \subsection{Consequences of Koszul duality for analytic functors on $\gr\op$}
Combining Theorem \ref{thm:main_Vext} with Theorem \ref{thm:recog} gives:

\begin{thm}
\label{thm:analytic_grop_flie}
There is an equivalence of categories 
$
\f_\omega (\gr\op; \kring)
\cong \flie.
$ 
\end{thm}

\begin{proof}
The category $\f_\omega (\gr\op; \kring)$ has coproducts and the set of small projective generators  $\{ \bu_d \kring [\sym_d] \ | \ d \in \nat\}$ by Corollary  \ref{cor:proj_gen_anal_grop}. Hence by Freyd's theorem (see \cite[Theorem 3.1]{MR294454} for a version over $\zed$), it is equivalent to the category of contravariant $\kring$-linear functors to $\kmod$ from the full subcategory of $\f_\omega (\gr\op;\kring)$ with objects the projective generators $\bu_d \kring [\sym_d]$, for $d \in \nat$.
 
These projective generators are finite,  hence one can restrict to considering the full subcategory $\f$ of finite objects of $\f_\omega (\gr\op; \kring)$  and the respective full subcategory of $\flie$. This allows Theorem \ref{thm:recog} to be applied.

By Theorem \ref{thm:recog}, it suffices to show that the Yoneda algebra
\[
\ext^* _{\f} ((\A^\sharp)^{\otimes m}, (\A^\sharp)^{\otimes n}), \mbox{ \ $m, n \in \nat$ }
\] 
is the underlying structure of the  PROP generated by the graded operad $\os \otimes_H \com$. More precisely, one shows that duality induces an isomorphism:
\[
\ext^* _{\f} ((\A^\sharp)^{\otimes m}, (\A^\sharp)^{\otimes n})
\stackrel{\cong}{\rightarrow}
\ext^* _{\fcatk[\gr]} (\A^{\otimes n}, \A ^{\otimes m} ) 
\]
that is compatible with the additional structure. The identification then follows from Theorem \ref{thm:main_Vext}.

The above duality isomorphism is established as follows. Theorem \ref{thm:DPV_poly} allows the statement of Theorem \ref{thm:main_Vext} to be translated into one about polynomial functors. The $\ext$ groups are calculated by using injective resolutions by {\em finite} injective functors. Dualizing provides projective resolutions in $\f_\omega (\gr\op, \kring)$ and these calculate the $\ext$ groups as required.
\end{proof}

\begin{rem}
The equivalence of Theorem \ref{thm:analytic_grop_flie} is made explicit in Section \ref{sect:pbw}, using a model for the functors $\bu_d \kring [\sym_d]$ and morphisms between them. 
\end{rem}

\begin{rem}
\label{rem:gr_ab}
One can consider  $\ab$, the category of finitely-generated free abelian groups, in place of $\gr$, using $\oplus$ to define polynomiality. This leads to the category $\f_\omega (\ab\op ; \kring)$ of analytic functors on $\ab\op$. This category is semisimple; more precisely, the appropriate cross-effects functor induces an equivalence of categories: 
\[
\f_\omega (\ab\op ; \kring) \cong \smodug.
\] 
The category $\smodug$ is equivalent to $\f_I$, for $I$ the unit operad.

The abelianization functor $\gr \rightarrow \ab$, $G \mapsto G/[G,G]$, induces the restriction functor 
$\f_\omega (\ab\op ; \kring) \rightarrow \f_\omega (\gr\op ; \kring)$. Under the equivalence of 
 Theorem \ref{thm:analytic_grop_flie},  this corresponds to 
 $
\f_I \cong  \smodug \rightarrow \flie,
 $ 
induced by restriction along the augmentation $\lie \rightarrow I$ (cf. Example \ref{exam:F_I_opd}). This leads to the diagram (\ref{eqn:abelianization_flie}) of the Introduction:
\begin{eqnarray*}
\xymatrix{
\smodug 
\ar[d]_\cong 
\ar@{^(->}[r]
&
\flie 
\ar[d]^\cong
\\
\f_\omega(\ab\op; \kring) 
\ar@{^(->}[r]
&\f_\omega(\gr\op; \kring),
}
\end{eqnarray*}
which explains how $\f_\omega(\ab\op; \kring) $ fits into the theory.
\end{rem}

\begin{exam}
\label{exam:group_ring}
Under the equivalence of Theorem \ref{thm:analytic_grop_flie}, for $d \in \nat$, the polynomial functor $(\A^\sharp)^{\otimes d}$ corresponds to $\kring [\sym_d]$, considered as a $\cat \lie$-module concentrated in arity $d$. 
\end{exam}

\subsection{The covariant case}
The covariant case (i.e.,  functors on $\gr$) is also of  interest. Whereas in the contravariant case, one is lead to consider {\em analytic}  functors, in the covariant case the appropriate categorically dual notion should be used. (The general case involves some technical issues that are treated in \cite{malcev}.) Instead,  we restrict to {\em finite} polynomial covariant functors, which allows Corollary \ref{cor:finite_poly_gr} to be deduced directly by duality from  
Theorem \ref{thm:analytic_grop_flie}.
  
 \begin{nota}
 \label{nota:fart}
Denote by
\begin{enumerate}
\item 
$\fart$  the full subcategory of finite polynomial functors in $\fcatk[\gr]$;
\item 
$\fartop$  the full subcategory of  finite polynomial functors in $\fcatk[\gr\op]$.
\end{enumerate}
 \end{nota}

Proposition \ref{prop:duality_gr_poly} implies:
 
 \begin{cor}
 Vector space duality induces an equivalence of categories $\fart \op \stackrel{\cong}{\rightarrow} \fartop$.
 \end{cor}
 
Analogously, Lemma \ref{lem:vs_duality}  implies: 
 
 \begin{lem}
 \label{lem:duality_cat_lie_modules} 
 Vector space duality induces an equivalence of categories between $(\flie \fn)\op$ and the full subcategory $\modules_{\catlie}\fn$ of finite objects of the category of right $\cat \lie$-modules.
 \end{lem}

\begin{rem} 
The category $\modules_{\catlie}$ of right $\cat \lie$-modules is equivalent to the category of right $\lie$-modules (defined with respect to the operadic composition product, $\circ$) by \cite[Proposition 1.2.6]{KM}.   (The category of right modules over an operad is  considered in \cite{MR2494775}, for example.)

Hence  $\modules_{\catlie}\fn$  is equivalent to the full subcategory of finite right $\lie$-modules (again, those admitting a finite composition series). 
 \end{rem}
 
Putting these results together, Theorem \ref{thm:analytic_grop_flie} implies:
 
 \begin{cor}
 \label{cor:finite_poly_gr}
 The category  $\fart$ is equivalent to $\modules_{\catlie}\fn$ and hence to the category of finite right $\lie$-modules. 
 \end{cor}
 
\begin{proof}
Theorem \ref{thm:analytic_grop_flie} implies that $\fartop$ is equivalent to $\flie\fn$ and hence that $\fart$ is equivalent to $\modules_{\catlie}\fn$ by using the above duality equivalences. 
The second statement is then a consequence of Proposition \ref{prop:induct_coinduct_dual}.
\end{proof}

\section{The explicit equivalence via $\cat \uass$}
\label{sect:pbw}
 
This section serves to make the equivalence of categories of Theorem \ref{thm:analytic_grop_flie} more concrete. The explicit form of the  equivalence is given in Theorem \ref{thm:morita_eq_cat_uass} (corresponding to Theorem \ref{THM:uass} of the introduction), using the relationship between the Lie operad $\lie$ and the unital associative operad $\uass$. This result gives an alternative proof of the equivalence of Theorem \ref{THM:Morita}.

The methods are  independent of the Koszul duality arguments used in the proof of Theorem \ref{thm:analytic_grop_flie}. 
 The proof of Theorem \ref{thm:morita_eq_cat_uass} uses information on the projective polynomial functors in $\fcatk[\gr\op]$ with fundamental input provided by the relationship between cocommutative Hopf algebras and functors on $\gr\op$ through exponential functors. A further  key ingredient is  the Poincaré-Birkhoff-Witt theorem.

The explicit also construction also leads to an explicit form of the model for finite polynomial functors on $\gr$; this is given in Theorem \ref{thm:covariant_concrete}.

\subsection{Revisiting the Poincaré-Birkhoff-Witt theorem}

Recall  that $\uass$ denotes the operad governing  unital associative $\kring$-algebras (cf. Example \ref{exam:opds}). There is a  morphism of operads $\lie \rightarrow \uass$ that induces the functor $\alg_{\uass} \rightarrow \alg_{\lie}$ that restricts a unital, associative algebra $A$ to the Lie algebra $(A, [-,-])$ for the commutator bracket given by $[a,b]= ab-ba$ for $a, b \in A$.

This implies:

\begin{lem}
\label{lem:uass_right_lie}
The $\fb\op$-module underlying $\uass$ has the structure of a right $\lie$-module with respect to the monoidal structure 
$(\sopug, \circ, I)$. 
\end{lem}

\begin{proof}
The structure morphism   $\uass \circ \lie \stackrel{\rho}{\rightarrow} \uass$ is given by the composite:
\[
\uass \circ \lie 
\rightarrow 
\uass \circ \uass 
\rightarrow 
\uass,
\]
where the first map is induced by $\lie \rightarrow \uass$ and the second by the operad structure of $\uass$.
\end{proof}

Recall that $\ucom$ is the operad governing unital commutative algebras (cf. Example \ref{exam:opds}).
The following is clear working over a field of characteristic zero (note that it only considers the underlying $\fb\op$-modules, not the operad structures):

\begin{lem}
\label{lem:ucom_transfer}
There is a monomorphism of $\fb\op$-modules
 $
\iota : \ucom \hookrightarrow \uass 
$ 
that identifies in arity $ n \in \nat$ with the transfer $\kring \hookrightarrow \kring [\sym_n]$, 
$1 \mapsto \frac{1}{n!} \sum_{\sigma \in \sym_n} [\sigma]$.
\end{lem}

This gives the following  operadic form of the Poincaré-Birkhoff-Witt theorem:

\begin{prop}
\label{prop:pbw}
The composite 
\[
\ucom \circ \lie 
\stackrel{\iota \circ \mathrm{Id}}{\longrightarrow}
\uass \circ \lie 
\stackrel{\rho}{\longrightarrow}
\uass
\]
is an isomorphism of right $\lie$-modules, where $\ucom \circ \lie$ is considered as the free  right $\lie$-module on $\ucom$.
\end{prop}

\begin{proof}
That this is an isomorphism of $\fb\op$-modules follows from the classical Poincaré-Birkhoff-Witt theorem (as in \cite[Section 2A]{MR3546467}, for example). It is a morphism of right $\lie$-modules, since $\rho$ is, when  $\uass \circ \lie$ is considered as the free right $\lie$-module on $\uass$. 
\end{proof}

The morphism of operads $\lie \rightarrow \uass$ induces a $\kring$-linear functor $\cat \lie \rightarrow \cat \uass$ 
(see Remark \ref{rem:cat_opd}). As in Lemma \ref{lem:uass_right_lie}, the underlying $\fb$-bimodule of $\cat \uass$ inherits a right $\cat \lie$-module structure, also denoted by $\rho$. 

Using the identification of $\cat \opd$ as $\tc \opd$ (cf. Remark \ref{rem:cat_opd}), one has:

\begin{lem}
\label{lem:iota_cat}
The monomorphism of $\fb\op$-modules $\iota : \ucom \rightarrow \uass$ induces a monomorphism of 
$\fb$-bimodules:
\[
\tc \iota : 
\tc \ucom 
\hookrightarrow 
\tc \uass,
\]
and hence a monomorphism between the underlying $\fb$-bimodules: $\iota : \cat \ucom \hookrightarrow \cat \uass$.
\end{lem}

Thus Proposition \ref{prop:pbw} has the following counterpart:

\begin{cor}
\label{cor:pbw_cat}
The $\fb$-bimodule $\cat \uass$ is free as a right $\cat \lie$-module.
 Explicitly, the composite 
\[
\cat \ucom \otimes_\fb \cat \lie 
\stackrel{\iota \otimes \mathrm{Id}}{\longrightarrow}
\cat \uass \otimes_\fb \cat \lie 
\stackrel{\rho}{\longrightarrow} 
\cat \uass
\]
is an isomorphism of right $\cat \lie$-modules. 
\end{cor}

\begin{proof}
By construction, the map given in the statement is a morphism of right $\cat \lie$-modules, hence it suffices to show that it is an isomorphism at the level of the underlying $\fb$-bimodules. 

This is established by using the relationship between $\fb\op$-modules and Schur functors (cf. Section \ref{subsect:schur_opd}). Namely, over a field $\kring$ of characteristic zero, it suffices to show that the map is an isomorphism after applying the functor $-\otimes _\fb \underline{V}$, naturally with respect to $V \in \ob \kmod$.

Using Lemma \ref{lem:identify_free_cat_opd_module}, one shows that the map 
$\cat \ucom \otimes_\fb \cat \lie \otimes _\fb \underline{V} 
\rightarrow 
\cat \uass \otimes _\fb \underline{V} $
identifies as
\[
\underline{\ucom (\lie (V))} 
\rightarrow 
\underline{\uass (V)}
\]
induced by the isomorphism $\ucom \circ \lie \stackrel{\cong}{\rightarrow} \uass $ of  Proposition \ref{prop:pbw}, hence is an isomorphism, as required.
\end{proof}

\subsection{The universal enveloping algebra}
\label{subsect:env_alg}

Restriction along $\lie \rightarrow \uass$ induces the forgetful functor $\alg_{\uass}\rightarrow \alg_{\lie}$ that gives the underlying Lie algebra of a unital, associative algebra. Classically, the universal enveloping algebra functor $U : \alg_{\lie} \rightarrow \alg_{\uass}$ is defined as the left adjoint to the above restriction.  

The morphism of operads $\lie \rightarrow \uass$ induces the relative circle product $\uass \circ_\lie - : \alg_{\lie} \rightarrow \alg_{\uass}$ (see \cite[Section 5.2.12]{LV} for the relative circle product). Essentially by construction of the latter, one has:

\begin{prop}
\label{prop:lie_tensor}
The universal enveloping algebra functor $U : \alg_{\lie} \rightarrow \alg_{\uass}$  is naturally isomorphic to 
 $\uass \circ_\lie - : \alg_{\lie} \rightarrow \alg_{\uass}$. 
\end{prop}

\begin{rem}
\label{rem:classical_pbw}
Using Proposition  \ref{prop:lie_tensor}, 
Proposition \ref{prop:pbw} yields the following  form of the classical Poincaré-Birkhoff-Witt theorem: for $\g$ a Lie algebra, there are natural isomorphisms of $\kring$-vector spaces
\[
U \g 
\cong 
\uass \circ _\lie \underline{\g}
\cong 
\ucom (\g),
\]
where $\underline{\g}$ is as in Section \ref{subsect:schur_opd}.
\end{rem}

\begin{rem}
\label{rem:Hopf_Ug}
The universal enveloping algebra functor enriches to a functor with values in cocommutative Hopf algebras. For generalization later, an explanation of this structure is outlined.

The free unital associative algebra $\uass (V)$ is the tensor algebra $T(V)$; this is a cocommutative Hopf algebra with respect to the shuffle coproduct, which is determined by the condition that  $V$ is contained in the primitives. 

For $\g $ a Lie algebra, the structure morphism $\lie (\g ) \rightarrow \g $ induces a morphism of Hopf algebras $T (\lie (\g ) ) \rightarrow T(\g )$ on applying the tensor algebra functor $T$. One also has the composite 
\[ 
T(\lie (\g  ) ) \rightarrow T (T (\g ))
\rightarrow T (\g ) 
\]
induced by $\lie \rightarrow \uass$ and the operadic composition $\uass \circ \uass \rightarrow \uass$. Since $\lie (\g )$ coincides with the primitives of $T (\g )$, this composite is also a morphism of Hopf algebras. 

By construction, $U \g$ is the coequalizer in associative algebras of these two morphisms of Hopf algebras
\[
 T(\lie (\g ) ) 
 \rightrightarrows
 T (\g ). 
\]
The coequalizer inherits a canonical Hopf algebra structure from $T(\g)$.
\end{rem}

\subsection{Enriching the structure on $\cat \uass$}
\label{subsect:enrich_cat_uass}

Corollary \ref{cor:pbw_cat} exhibits $\cat \uass$ as a  right $\cat \lie$-module. The purpose of this section is to show that this structure can be enriched.

\begin{nota}
\label{nota:Phi}
(Cf. \cite[Section 5]{PV}.)
Let $\Phi :\mathbf{Hopf}_\kring^{\mathrm{cocom}}\rightarrow \fcatk[\gr\op]$, $H \mapsto \Phi H$, denote the functor from  cocommutative Hopf algebra over $\kring$ to functors on $\gr\op$, where $\Phi H$ is the associated exponential functor (see the following Remark).
\end{nota}

\begin{rem}
\label{rem:structure_Phi}
The structure of $\Phi H$, for $H$ a cocommutative Hopf algebra over $\kring$, is as follows. For $n \in \nat$,  
$$
\Phi H ( \fg_n) := H^{\otimes n}.
$$ 

The action of morphisms of $\gr\op$ is induced by the Hopf algebra structure of $H$ and place permutations of the tensor product. By Lemma \ref{lem:generating_gr}, the action is determined by  the following identifications:

\begin{enumerate}
\item 
$\Phi H (\fg_0)  \rightarrow \Phi (H) (\fg_1)\rightarrow \Phi H (\fg_0) $ induced by the canonical $\fg_0 = \{e\} \rightarrow \fg_1 \rightarrow \{e \}= \fg_0$ are respectively the unit and counit $\kring \rightarrow H \rightarrow \kring$;
\item 
$\Phi H (\fg_1) \rightarrow \Phi H (\fg_1)$ induced by $\fg_1 \stackrel{(-)^{-1}}{\rightarrow } \fg_1$ is the Hopf algebra conjugation $\chi : H\op \rightarrow H$;
\item 
$\Phi H (\fg_1) \rightarrow \Phi H (\fg_2)$ induced by the fold map $\fg_2 \cong \fg_1 \star \fg_1 \rightarrow \fg_1$ is the coproduct $H \rightarrow H^{\otimes 2}$; 
\item 
$\Phi H (\fg_2) \rightarrow \Phi H (\fg_1)$ induced by the morphism $\fg_1 \rightarrow \fg_2$ sending 
the generator of $\fg_1$ to the group product $x_1 x_2$, where $x_i$ denotes the generator of the $i$th copy of $\fg_1$ in $\fg_2 \cong \fg_1 \star \fg_1$, is the product $H \otimes H \rightarrow H$.
\end{enumerate}
\end{rem}

The following is the key example for our purposes:

\begin{exam}
\label{exam:Phi_TV}
For $V$ a $\kring$-vector space, the tensor Hopf algebra $T(V)$ is a cocommutative Hopf algebra with respect to the shuffle coproduct. One has the associated exponential functor $\Phi T(V) \in \ob \fcatk[\gr\op]$ and this is natural with respect to $V$.
\end{exam}

The $\kring$-linear category  $\cat \uass$ has an underlying $\fb$-bimodule (cf. section \ref{subsect:cat_opd}). The key idea is that the exponential functor construction means that this structure extends to a left $\gr\op$, right  $\fb$-bimodule structure (equivalently a left $\gr\op\times \fb\op$-module structure).  
Here,  extension is understood with respect to the following embedding:

\begin{lem}
\label{lem:free_grp}
The free group functor induces a faithful embedding $\fb \hookrightarrow \gr\op$ via the composite $\fb \cong \fb\op \hookrightarrow \gr\op$.
\end{lem}

\begin{prop}
\label{prop:structure_cat_uass}
The right $\cat  \lie$-module structure on $\cat \uass$ in $\fb$-bimodules enriches to a right $\cat  \lie$-module structure  in left $ \gr\op \times \fb\op$-modules.

Namely, there is a  left $ \gr\op\times \fb\op$-module structure on $\cat \uass$ that extends the $\fb$-bimodule structure and such that the right $\cat  \lie$-module structure morphism 
\[
\cat \uass \otimes_\fb \cat \lie 
\rightarrow 
\cat \uass
\]
is a morphism of left $ \gr\op$-modules.
\end{prop}

\begin{rem}
Part of the structure of $\cat \uass$ is made explicit in Section \ref{sect:catassu}; this might make the following proof more accessible.
\end{rem}

\begin{proof}[Proof of Proposition \ref{prop:structure_cat_uass}]
To construct the $\gr\op \times \fb\op$-module structure on $\cat \uass$, it suffices to show that the natural $\fb$-module structure on the Schur functor $\cat \uass (V)$ extends to a $\gr\op$-module structure, naturally with respect to $V \in \ob \kmod$. 

By Lemma \ref{lem:identify_free_cat_opd_module}, $\cat \uass (V) \cong 
\underline{\uass (V)} = \underline{T(V)}$ as a $\fb\op$-module. The latter is the $\fb\op$-module underlying the exponential functor $\Phi (T(V)) \in \ob \fcatk[\gr \op]$. These structures are natural in $V$, so this yields the required  $\gr\op \times \fb\op$-module structure on $\cat \uass$. 

It remains to show that this is compatible with the right $\cat \lie$-module structure. It is sufficient to show that the  morphism induced by the right $\cat \lie$-module structure on $\cat \uass$ 
\[
\cat \uass \otimes_\fb \cat \lie \otimes_\fb \underline{V} 
\rightarrow 
\cat \uass \otimes _\fb \underline{V}
\]
is a morphism in $\fcatk[\gr\op]$, naturally with respect to $V \in \ob \kmod$. 

 By Lemma \ref{lem:identify_free_cat_opd_module}, the underlying map of $\fb\op$-modules 
 identifies with 
\begin{eqnarray}
\label{eqn:cat_lie_action}
\underline{T (\lie (V)) }
\rightarrow 
\underline {T(V)}
\end{eqnarray}
that is induced by the Hopf algebra morphism given by the composite 
\[
T (\lie (V)) 
\rightarrow 
T(T(V))
\rightarrow 
T(V)
\]
induced by $\lie(V) \subset T(V)$ and the composition $T(T(V)) \rightarrow T(V)$. This morphism of Hopf algebras is natural with respect to $V \in \ob \kmod$.  

The map (\ref{eqn:cat_lie_action}) underlies the map of $\gr\op$-modules 
\[
\Phi (T (\lie(V))) 
\rightarrow 
\Phi (T(V))
\]
obtained by applying the exponential functor construction $\Phi (-)$ to the Hopf algebra map and this is natural with respect to $V$. The result follows.
\end{proof}

\begin{nota}
\label{nota:gropcatuass}
Denote by $\gropcatuass$ the left $\gr\op$, right $\cat \lie$-bimodule  provided by Proposition \ref{prop:structure_cat_uass}.
\end{nota}

\subsection{The explicit equivalence}
The set of projective generators $\{ \bu_d\kring [\sym_d] \  | \ d \in \nat \}$ assembles to a $\gr\op \times \fb\op$-module for which the following notation is introduced.

\begin{nota}
\label{nota:bu_sym_bullet}
Let $\bu \kring [\sym_\bullet]$ denote the $\gr\op \times \fb\op$-module corresponding to the functor $\fb\op \rightarrow \fcatk[\gr\op]$ given by $\mathbf{d} \mapsto  \bu_d\kring [\sym_d] $.
\end{nota}

The following result describes this in terms of $\gropcatuass$.

\begin{prop}
\label{prop:bd_vs_cat_uass}
There is an isomorphism $\bu \kring [\sym_\bullet] \cong \gropcatuass$  in $\gr\op \times \fb\op$-modules.

Equivalently, for $d \in \nat$, there is an isomorphism of functors 
in $\f_\omega (\gr\op;\kring)$:
\[
\bu_d \kring [\sym_d]  \cong \gropcatuass (d,- )
\]
that is $\sym_d$-equivariant with respect to the canonical actions.
\end{prop}

\begin{proof}
This is a consequence of the proof of \cite[Theorem 9.6]{PV}; for completeness, an argument is outlined here.
 We require to show that, for each $d \in \nat$, there is a $\sym_d$-equivariant isomorphism of functors 
$
\bu_d \kring [\sym_d] \cong \gropcatuass (d,- ).
$

Proposition \ref{prop:catuass_poly} shows that $ \gropcatuass (d,- )$ has polynomial degree $d$ with $\gamma_d \big( \gropcatuass (d,- )\big) \cong \kring [\sym_d]$. Hence, by Corollary \ref{cor:left_adj_gamma}, one has the morphism 
\[
\bu_d \kring [\sym_d] \rightarrow \gropcatuass (d,- )
\]
of $\fcatk[\gr\op]$ corresponding to the unit $[e]  \in  \kring [\sym_d]$.

Since the $\gr\op$-structure on $\gropcatuass$ extends composition of morphisms in $\cat \uass$, this morphism is surjective. (This can also be checked directly by using the structure of $\gropcatuass$ given in Section \ref{sect:catassu}.)

 Both functors take finite-dimensional values, hence it suffices to show that these values have the same dimension; this follows directly from the analysis of the functor  $\beta_d$ given in the proof of \cite[Proposition 4.4]{DPV}. 
\end{proof}

The advantage of working with $\gropcatuass$ rather than $\bu \kring [\sym_\bullet]$ is that it has an explicit right $\cat \lie$-module structure. Moreover, one has the following:

\begin{prop}
\label{prop:fully_faithful_embedding}
The right $\cat \lie$-module structure on $\cat \uass$ induces a fully-faithful 
$\kring$-linear functor 
\begin{eqnarray*}
(\cat \lie)\op & \rightarrow & \f_\omega (\gr\op; \kring) \\
d & \mapsto & \gropcatuass (d, -) .
\end{eqnarray*}
\end{prop}

\begin{proof}
 Corollary \ref{cor:pbw_cat} shows that $\cat \uass$ is free as a right $\cat \lie$-module and Proposition \ref{prop:structure_cat_uass} that this action is compatible with the $\gr\op$-action. It follows that the right $\cat \lie$-action induces embeddings:
\begin{eqnarray}
\label{eqn:cat_lie_embed}
\cat \lie (d, e) 
\hookrightarrow 
\hom_{\gr\op} (\gropcatuass (e,- ) , \gropcatuass (d, -) ) 
\end{eqnarray}
for $d, e \in \nat$. These are compatible with composition. To complete the proof that the associated $\kring$-linear functor $(\cat \lie)\op  \rightarrow \f_\omega (\gr\op; \kring) $ is fully-faithful, we require to show that it is also surjective. 

By Proposition \ref{prop:bd_vs_cat_uass}, there are isomorphisms
$\bu_d \kring [\sym_d] \cong \cat \uass (d,- )$  and $\bu_e \kring [\sym_e]   \cong \cat \uass (e,- )$ in $\f_\omega (\gr\op; \kring)$. Hence the above embedding can be rewritten as: 
\begin{eqnarray}
\label{eqn:cat_lie_embed_bu}
\cat \lie (d, e) 
\hookrightarrow 
\hom_{\gr\op} ( \bu_e \kring [\sym_e] , \bu_d \kring [\sym_d] ). 
\end{eqnarray}
Both sides vanish if $d<e$ and the map  is clearly an isomorphism for $d=e$. We require to prove that it is an isomorphism for all $d, e$.

Now, by Corollary \ref{cor:proj_gen_anal_grop}, 
\[
\hom_{\gr\op} ( \bu_e \kring [\sym_e] , \bu_d \kring [\sym_d] )
\cong \gamma_e (\bu_d \kring [\sym_d]),
\]
hence corresponds to the composition factors of the functor $\bu_d \kring [\sym_d]$ that have polynomial degree exactly $e$. This reduces us to understanding the associated graded of the polynomial filtration of the functor  $\bu_d \kring [\sym_d]$.

Combining Proposition \ref{prop:bd_vs_cat_uass} with Proposition \ref{prop:structure_cat_uass} and the identifications used in the proof, under the Schur correspondence, one sees that $\bu_d \kring [\sym_d]$ corresponds to the component that is homogeneous polynomial of degree $d$ with respect to $V$ of the functor 
\[
V \mapsto \Phi T(V),
\]
where $\Phi T(V)$ is the exponential functor on $\gr\op$ associated to the tensor algebra $T(V)$ equipped with the shuffle coproduct; the latter is isomorphic to the universal enveloping algebra $U \lie (V)$ as a Hopf algebra.

Since we are only interested in the associated graded of the polynomial filtration, we replace $T(V)$ by its associated graded, i.e., the bicommutative Hopf algebra $S(\lie (V))$, where $S(-)$ denotes the free commutative algebra functor. This is primitively generated by $\lie (V)$; it is isomorphic as a Hopf algebra to the universal enveloping algebra on $\lie (V)$ considered as an abelian Lie algebra. The associated graded of the polynomial filtration of $\Phi T(V)$  is given by:
\[
V \mapsto \Phi S (\lie(V)).
\]

We relate this to $\cat \lie$ as follows. By Lemma \ref{lem:identify_free_cat_opd_module}, the Schur functor associated to $\cat \lie$ is 
\[
\cat \lie \otimes_\fb \underline{V} \cong \underline{\lie (V)},  
\]
 and $\cat \lie (d,-)$ corresponds to the degree $d$ homogeneous polynomial component with respect to $V$. Explicitly, there is a $\sym_e$-equivariant isomorphism  $\cat \lie (d, e) \otimes_{\sym_d} V^{\otimes d} \cong (\lie(V)^{\otimes e})_{[d]}$, where the subscript $_{[d]}$ indicates the degree $d$ homogeneous component. We claim that this is isomorphic to $\gamma_e (\bu_d \kring [\sym_d])$; this claim implies the result, by the Schur correspondence.

To prove the claim, using Proposition \ref{prop:right_adjoint_tr} to pass from $\sym_e$-modules to homogeneous polynomial degree $e$ functors on $\gr\op$,  it is equivalent to show that 
\begin{eqnarray}
\label{eqn:homog_cpt}
(\A^\sharp)^{\otimes e} \otimes_{\sym_e} ((\lie(V)^{\otimes e})_{[d]}) 
\cong 
\big((\A^\sharp)^{\otimes e} \otimes_{\sym_e} \lie(V)^{\otimes e}\big) _{[d]} 
\end{eqnarray}
is isomorphic to the polynomial degree $e$ (with respect to $\gr\op$) part of the associated graded of $\bu_d \kring [\sym_d]$ for the polynomial filtration.

As usual, it is useful to work `globally', treating all $d$ and $e$ at once. The functors $(\A^\sharp)^{\otimes e}$ are encoded in $\cat \ucom$, which acquires a left $\gr\op$-action that commutes with the canonical right $\fb$-action, just as for $\gropcatuass$  in Proposition \ref{prop:structure_cat_uass} (indeed, the case of $\cat \ucom$ can be deduced from that Proposition). 

Explicitly, $(\A^\sharp)^{\otimes e}$ is isomorphic to $\cat \ucom (e, -)$, considered as a functor on $\gr\op$ and equipped with the obvious $\sym_e$-action. We outline how this isomorphism is constructed, using that the $\kring$-linear category $\cat \ucom$ is equivalent to $\kring \mathbf{Fin}$ (see Example \ref{exam:com}). Evaluated on $\zed^{\star n}\in \ob \gr\op$, equipped with the associated basis $x_1, \ldots, x_n$, the isomorphism between $\cat \ucom (e, n) \cong \kring \hom_{\mathbf{Fin}} (\mathbf{e}, \mathbf{n}) $ and $\A^{\sharp} (\zed^{\star n})$ sends a generator given by a map $f : \mathbf{e} \rightarrow \mathbf{n}$ to the element $\bigotimes_{i=1}^e x^\sharp_{f(i)}$ of $(\A^\sharp)^{\otimes e}(\zed^{\star n})$, where $x^\sharp_1, \ldots, x^\sharp_n$ is the `dual basis' of $\zed^{\oplus n}\cong  \A^{\sharp} (\zed^{\star n})$. This is clearly an isomorphism of $\kring$-vector spaces and can be checked to be  compatible with the additional structures.

Thus, considering (\ref{eqn:homog_cpt}) for all $d$ and $e$ corresponds to considering 
\[
\cat \ucom \otimes_\fb \underline{\lie (V)}.
\]
The claim is therefore equivalent to the assertion that the functors $\Phi S(\lie (V))$ and $\cat \ucom \otimes_\fb \underline{\lie (V)}$ are isomorphic, naturally with respect to $V$.

Now, evaluating $\cat \ucom \otimes_\fb \underline{\lie (V)}$ on $\zed^{\star n}$ gives:
\[
S (\lie (V))^{\otimes n}, 
\]
by Lemma \ref{lem:identify_free_cat_opd_module} applied with respect to the operad $\ucom$, using that $\ucom(W)$ is the symmetric algebra $S(W)$ for a $\kring$-vector space $W$.

This is isomorphic, naturally with respect to $V$, to $\Phi S(\lie (V))$ evaluated on $\zed^{\star n}$. Moreover, the isomorphism respects the respective $\gr\op$-structures. This concludes the proof of the claim and hence of the result.
\end{proof}

This leads to the following concrete form of Theorem \ref{thm:analytic_grop_flie}:

\begin{thm}
\label{thm:morita_eq_cat_uass}
The  following are quasi-inverse equivalences of categories:
\begin{eqnarray*}
\hom_{\f_\omega (\gr\op; \kring ) } (\gropcatuass, -  ) &:& 
\f_\omega (\gr\op; \kring ) 
\rightarrow 
\f_\lie
\\
\gropcatuass  \otimes_{\cat \lie} - \ &:& \ 
\flie 
\rightarrow 
\f_\omega (\gr\op; \kring).
\end{eqnarray*}
\end{thm}

\begin{proof}
The functors are defined using the structure given by Proposition \ref{prop:fully_faithful_embedding}.
 That these induce the equivalence of Theorem \ref{thm:analytic_grop_flie} follows as in the proof of 
that Theorem. 
\end{proof}

\begin{rem}
A shorter proof of Proposition \ref{prop:fully_faithful_embedding} can be given using the fact that $\f _\omega (\gr\op; \kring)$ satisfies the Koszul property (as was done in an earlier version of this work).

The above argument has been preferred since it is independent of the Koszul property. Indeed, the Koszul property can be deduced from this: combined with Theorem \ref{thm:DPV_poly}, Theorem \ref{thm:morita_eq_cat_uass} leads to an alternative proof of Theorem \ref{thm:main_Vext} over $\kring$.

Moroever, this approach is essential when generalizing Theorem \ref{thm:morita_eq_cat_uass} to related situations in which  the Koszul property does not hold.
\end{rem}

\begin{rem}
\label{rem:PBW_induction}
The rôle of $\cat \ucom$ in the proof of Proposition \ref{prop:fully_faithful_embedding} can be explained in relation to Theorem \ref{thm:morita_eq_cat_uass} as follows. 

Corollary \ref{cor:pbw_cat} provides the isomorphism $\cat \ucom \otimes_\fb \cat \lie \cong \cat \uass$  of right $\cat \lie$-modules. Hence, if $M$ if an $\fb$-module, considered as a left $\cat \lie$-module via $\cat \lie \rightarrow \kring \fb$, there is an isomorphism of functors on $\gr\op$:
\[
\gropcatuass \otimes _{\cat \lie} M 
\cong 
\cat \ucom \otimes_\fb M, 
\]
where $\cat \ucom \otimes_\fb M$ is considered as a functor on $\gr\op$ as in the proof of  Proposition \ref{prop:fully_faithful_embedding}. Moreover, using that the $\kring$-linear category $\cat \ucom$ is equivalent to $\kring\mathbf{Fin}$, this allows $\cat \ucom \otimes _\fb M$ to be described `combinatorially' in terms of $M$.

In general, for any  left $\cat \lie$-module $M$, the underlying $\fb$-module of $\gropcatuass \otimes_{\cat \lie} M$ is isomorphic to 
$ 
\cat \ucom \otimes _\fb M.
$
 As in the proof of Proposition \ref{prop:fully_faithful_embedding}, this recovers the associated graded of the polynomial filtration of $\gropcatuass \otimes _{\cat \lie} M $, but does not in general recover  the full functoriality with respect to $\gr\op$. This is similar in nature to the Poincaré-Birkhoff-Witt isomorphism of Remark \ref{rem:classical_pbw}.
\end{rem}

\subsection{The case of Lie algebras}

Proposition \ref{prop:opd_alg} gives the faithful embedding $\alg_\lie \hookrightarrow \f_\lie$ that sends a Lie algebra $\g$ to the associated $\cat \lie$-module $\underline{\g}$. Then, by Theorem \ref{thm:morita_eq_cat_uass}, one can pass to the associated functor in $\f_\omega (\gr\op: \kring)$. Alternatively, one can consider the cocommutative Hopf algebra $U\g$ and form the exponential functor $\Phi (U\g)$. The following shows that the two constructions are equivalent:

\begin{thm}
\label{thm:lie_case}
For a Lie algebra $\g$, there is a natural isomorphism in $\f_\omega (\gr\op; \kring)$:
\[
\gropcatuass \otimes_{\cat \lie} \underline{\g}
\cong 
\Phi(U\g) 
,
\]
where $\underline{\g}$ is the $\cat \lie$-module associated to $\g$.
\end{thm} 

\begin{proof}
One reduces to the case $\g= \lie (V)$, the free Lie algebra on the vector space $V$, considered as a functor of $V$. 

On the left hand side, one has $\Phi (T(V))$. On the right hand side $\gropcatuass \otimes_{\cat \lie} \lie (V)$ is isomorphic to $\gropcatuass (V)$. These are naturally isomorphic as functors on $\gr\op$ by  construction of the action on the right hand side.
\end{proof}

\subsection{Further consequences}

Using the relationship with operadic Koszul duality, one obtains explicit
 projective resolutions of the functors $(\A^\sharp)^{\otimes d}$, as follows:

\begin{cor}
\label{cor:proj_Koszul_resolution}
For $d \in \nat$, under the equivalence of Theorem \ref{thm:morita_eq_cat_uass}, the Koszul complex of Proposition \ref{prop:koszul_complex} provides the following explicit, minimal projective resolution of $(\A^\sharp)^{\otimes d}$ in $\f_\omega (\gr\op; \kring)$:
\begin{eqnarray*}
\gropcatuass(1,-)\otimes_{\sym_1} \pbb^{\as}  (d, 1)
\rightarrow 
\gropcatuass(2,-)\otimes_{\sym_2} \pbb^{\as}  (d, 2)
\rightarrow 
\ldots 
\\
\ldots
\rightarrow 
\gropcatuass(d-1,-) \otimes_{\sym_{d-1}} \pbb^\as  (d, d-1)
\rightarrow 
\gropcatuass(d,-)
\twoheadrightarrow 
(\A^\sharp)^{\otimes d},
\end{eqnarray*}
where $\pbb^\as$  denotes the dual of $\cat (\os \otimes_H \com)$.
\end{cor}

\begin{proof}
This is an immediate consequence of the identifications used in proving Theorem \ref{thm:morita_eq_cat_uass}. In particular, the differentials in the complex are given explicitly in terms of the $\cat \lie$-module structure of $\gropcatuass$ and the structure of $\pbb^\as$, as in the construction of the Koszul complex of Proposition \ref{prop:koszul_complex}.
\end{proof}

Theorem \ref{thm:morita_eq_cat_uass} also yields the following concrete form of 
 the covariant result, Corollary \ref{cor:finite_poly_gr}:

\begin{thm}
\label{thm:covariant_concrete}
 The coinduction functor $\hom_{\modules_{\catlie}} (\gropcatuass, -)$ induces an equivalence of categories: 
\[
\hom_{\modules_{\catlie}} (\gropcatuass, -) 
:
\modules_{\catlie}\fn
\rightarrow 
\fart,
\] 
where $\modules_{\catlie}\fn$ is equivalent to the category of finite right $\lie$-modules.
 \end{thm}

\part{The convolution and tensor products}
\section{The convolution product for $\fopd$}
\label{sect:tensor}
 
This section introduces a symmetric monoidal structure on $\fopd$, where $\opd$ is taken to be a reduced operad.
There is a related structure for the category of right $\cat \opd$-modules that follows from \cite[Proposition 1.6.3]{KM}, which references \cite{MR1617616}.  Since the details of the construction are required in Section \ref{sect:compat} (for the case $\opd = \lie$), these are provided for left $\cat \opd$-modules (aka. $\fopd$).

\subsection{Constructing the convolution product}

In this section, we consider $\cat \opd$ as having objects finite sets, rather than restricting to the skeleton with objects $\mathbf{n}$, for $n \in \nat$. From this viewpoint the structure of $\cat \opd$ identifies as follows (the presentation should be compared with that of Section \ref{subsect:cat_opd}).

\begin{lem}
\label{lem:cat_opd_finset}
(Cf. \cite[Section 5.4.1]{LV}.) 
For $\opd$ a reduced operad and finite sets $X, Y$:
\[
\cat \opd (X, Y) 
= 
\bigoplus_{f \in \hom_\fs (X, Y)} 
\bigotimes _{y \in Y}
\opd (f^{-1} (y)).
\]

The symmetric monoidal structure on $\cat \opd$ is induced by the disjoint union of finite sets. Explicitly, for finite sets $X_1, X_2$ and $Y_1,Y_2$, on morphisms:
\[
\cat \opd (X_1, Y_1 ) \otimes \cat \opd (X_2, Y_2)
\rightarrow 
\cat \opd (X_1 \amalg X_2, Y_1 \amalg Y_2)
\]
is given on the factors indexed by surjections $f_1  : X_1 \rightarrow Y_1$ and $f_2 : X_2 \rightarrow Y_2$ by the isomorphism
\[
\Big(
\bigotimes _{y_1 \in Y_1}
\opd (f_1^{-1} (y_1))
\Big) 
\otimes 
\Big(
\bigotimes _{y_2 \in Y_2}
\opd (f_2^{-1} (y_2))
\Big)
\cong 
\bigotimes _{y \in Y_1 \amalg Y_2}
\opd (f^{-1} (y)),
\]
mapping to the factor indexed by  $f=f_1 \amalg f_2: X_1 \amalg X_2 \rightarrow Y_1\amalg Y_2$.
\end{lem}

The internal tensor product (or convolution product) $ \conv$  is then defined as follows: 

\begin{defn}
\label{defn:conv_fopd}
For $\opd$ a reduced operad, the convolution product  $\conv : \fopd \times \fopd \rightarrow \fopd$ is defined on objects for $F, G \in \ob \fopd$ by  
\[
F \conv G \  (Z) 
= 
\bigoplus _{X \amalg Y =Z} F(X) \otimes G(Y).
\]

For a morphism  $ \phi \in \cat \opd (Z, W)$ belonging to the factor indexed by $f \in \hom_\fs (Z,W)$, restricted to the summand indexed by the decomposition $X \amalg Y =Z$:
\begin{itemize}
\item 
if $f(X) \cap f(Y)\neq \emptyset$, then $\phi$ acts by zero; 
\item 
if  $f(X) \cap f(Y)= \emptyset$, $\phi$ is in the image of 
$ \cat \opd (X, f(X))\otimes  \cat \opd (Y, f(Y)) \rightarrow \cat \opd (Z, W) $ given by the symmetric monoidal structure,
and acts via 
 \[
F (X) \otimes G(Y) \rightarrow F (f (X)) \otimes G( f(Y)) \subset F \conv G \ (W),
\] 
where the first map is the tensor product of the maps given by the $\cat \opd$-structures of $F$ and $G$;
\item 
the general case is given by extending by $\kring$-linearity.
 \end{itemize} 
\end{defn}

In the following, $\kring$ is considered as the object of $\fopd$ given for $X \in \ob \fb$ by 
\[
X \mapsto \left\{
\begin{array}{ll}
\kring & X = \emptyset \\
0 & \mbox{otherwise.}
\end{array}
\right.
\]

\begin{prop}
\label{prop:sym_mon_fopd}
Suppose that $\opd$ is reduced.
\begin{enumerate}
\item 
 The convolution product $\conv : \fopd \times \fopd \rightarrow \fopd$ yields a symmetric monoidal structure $(\fopd, \conv, \kring)$. 
\item 
If $\opd \rightarrow \ppd$ is a morphism between reduced operads, then the restriction functor $\f_\ppd \rightarrow \fopd$ is symmetric monoidal.
\item 
The convolution product $F, G \mapsto F \conv G$ is exact with respect to both of the variables.
\end{enumerate}
\end{prop}

\begin{proof}
This is a straightforward verification from the explicit definition of the convolution product.
\end{proof}

\begin{exam}
The naturality statement with respect to the operad encodes the compatibility of the convolution product on $\fopd$ with the usual convolution product on $\smodug$ (cf. Definition \ref{defn:conv}):   the unit $I \rightarrow \opd$ induces the  forgetful functor 
$
\fopd \rightarrow \f_I \cong \smodug
$ 
and this is symmetric monoidal.
\end{exam}

\subsection{Behaviour on the standard projectives}

In this subsection, we suppose in addition that Hypothesis \ref{hyp:opd} holds for the operad $\opd$, so that, for each $m \in \nat$, the endomorphism algebra $\cat\opd (m,m)$ is isomorphic to $\kring [\sym_m]$. 

The behaviour of the convolution product is illustrated by considering the standard projectives $\popd_m$ of $\fopd$, for $m \in \nat$. By Proposition \ref{prop:proj_cat_opd-modules}, for $m \in \nat$, there is a canonical surjection 
 $
\popd_m \twoheadrightarrow \kring [\sym_m]
 $, considering $\kring [\sym_m]$ as an object of $\fopd$, and this is an isomorphism evaluated on $\mathbf{m} \in \ob \cat \opd$. Moreover, for $m, n \in \nat$, the convolution product induces a canonical isomorphism 
\[
\kring [\sym_m] \conv \kring [\sym_n] \cong \kring [\sym_{m+n}],
\]
via the inclusions $\mathbf{m} \hookrightarrow \mathbf{m+n} \cong \mathbf{m} \amalg \mathbf{n} \hookleftarrow \mathbf{n}$, treating the representation rings as objects of $\fopd$ as above. 
Using this isomorphism, one has the following:

\begin{lem}
\label{lem:surject_popd_conv}
For $m,n \in \nat$, there is a unique surjection $\popd_{m+n} \twoheadrightarrow \popd_m \conv \popd_n$ that makes the following diagram commute:
\[
\xymatrix{
\popd_{m+n}
\ar@{->>}[r]
\ar@{->>}[d]
&
\popd_m \conv \popd_n
\ar@{->>}[d]
\\
\kring [\sym_{m+n}]
\ar[r]_(.4)\cong 
&
\kring [\sym_m] \conv \kring [\sym_n],
}
\]
where the vertical morphisms are induced by the canonical surjections $\popd_t \twoheadrightarrow \kring [\sym_t]$, for $t \in \{ m+n,  m, n \}$.
\end{lem}

\begin{proof}
The surjection $\popd_m \conv \popd_n \twoheadrightarrow \kring [\sym_m] \conv \kring [\sym_n]$ is an isomorphism when evaluated on $\mathbf{m+n}$. The result follows since $\hom_{\fopd} (\popd_{m+n}, -)$ corepresents evaluation on $\mathbf{m+n}$, by Yoneda's Lemma (see Section \ref{sect:rep}).
\end{proof}

The surjection $\popd_{m+n} \twoheadrightarrow \popd_m \conv \popd_n$ has the following explicit description:

\begin{prop}
\label{prop:identify_popd_conv}
For $m, n , t \in \nat$, with respect to the identification 
\[
\popd_{m+n} (t) 
= 
\bigoplus_{f \in \hom_\fs (\mathbf{m} \amalg \mathbf{n} , \mathbf{t})} 
\bigotimes _{y \in \mathbf{t}}
\opd (f^{-1} (y))
\]
given by Lemma \ref{lem:cat_opd_finset}, $ \popd_{m} \conv \popd_n \ (t)$ identifies as the quotient
\[
 \popd_{m} \conv \popd_n \ (t)
= 
\bigoplus_{\substack{f \in \hom_\fs (\mathbf{m} \amalg \mathbf{n} , \mathbf{t})
\\
f(\mathbf{m}) \cap f(\mathbf{n}) = \emptyset}
} 
\bigotimes _{y \in \mathbf{t}}
\opd (f^{-1} (y))
\] 
obtained by sending the components indexed by $f$ such that $f(\mathbf{m}) \cap f(\mathbf{n}) \neq \emptyset$ to zero.
\end{prop}

\begin{proof}
This follows from the explicit description of $\conv$ given in Definition \ref{defn:conv_fopd}. 
\end{proof}

\section{Compatibility between convolution and the tensor product} 
\label{sect:compat}

The purpose of this section is to show that, via the equivalence given in Theorem \ref{thm:morita_eq_cat_uass}, the convolution 
product of Section \ref{sect:tensor} corresponds to the tensor product on $\f_\omega (\gr\op; \kring)$ that is induced by the tensor product on $\kmod$ by the following result.

\begin{prop}
\label{prop:tensor_analytic}
The tensor product on $\fcatk[\gr\op]$ induced by the tensor product on $\kmod$ restricts to a symmetric monoidal structure 
$(\f_\omega (\gr\op; \kring), \otimes, \kring)$, where $\kring$ denotes the constant functor.
\end{prop}

\begin{proof}
Since analytic functors are defined as colimits of polynomial functors and the tensor product commutes with colimits, the result  follows from the fact that the tensor product respects polynomiality.

More precisely, for $d, e \in \nat$,  the tensor product restricts to 
\[
\otimes : 
\f_d (\gr\op; \kring) 
\times 
\f_e (\gr\op; \kring) 
\rightarrow 
\f_{d+e} (\gr\op; \kring). 
\]
This can be seen directly by using the polynomial filtration. One reduces to the case $(\A^\sharp)^{\otimes m} \otimes (\A ^{\sharp})^{\otimes n}$, for $m\leq d$ and $n \leq e$. Since the tensor product is isomorphic to $(\A^\sharp)^{\otimes m+n}$, where $m+n \leq d +e$, the result follows in this case. 
\end{proof}

The main result of the section is the following:

\begin{thm}
\label{thm:compat_tensor_conv}
The equivalence $\gropcatuass\otimes_{\cat \lie}-\ : \ 
(\flie, \conv, \kring) 
\rightarrow 
(\f_\omega (\gr\op; \kring), \otimes, \kring) $ 
 is symmetric monoidal.
\end{thm}

\begin{exam}
\label{exam:conv_single_support}
As a first example, consider the convolution product of two objects of $\flie$ that are supported on a single object, such as $\kring [\sym_m]$, for $m \in \nat$, considered as an object of $\flie$ supported on $\mathbf{m}$.  
 
By Example \ref{exam:group_ring}, $\gropcatuass\otimes_{\cat \lie}-$ sends $\kring [\sym_m]$, to $(\A^\sharp)^{\otimes m}$. For $n \in \nat$, one has 
$ 
\kring [\sym_m] \conv \kring [\sym_n] \cong \kring [\sym_{m+n}]
$ in $\flie$. Across Theorem \ref{thm:compat_tensor_conv}, this  reflects the isomorphism $(\A^\sharp)^{\otimes m} \otimes (\A ^{\sharp})^{\otimes n} \cong (\A^\sharp)^{\otimes m+n}$.
\end{exam}

\begin{nota}
\label{nota:ao}
For totally ordered finite sets $X$ and $Y$, write  $X \ao Y$ for the disjoint union equipped with the total order extending that on $X$ and $Y$, and such that $x<y$ for all $x \in X$, $y \in  Y$. 
\end{nota}

\begin{rem}
For $m, n \in \nat$, $\mathbf{m} \ao \mathbf{n}$ is {\em canonically} isomorphic to $\mathbf{m+n}$, since there is a unique order-preserving isomorphism between the two.
\end{rem}

\begin{lem}
\label{lem:disj_cat_assu}
For $m, n, t \in \nat$, $\ao$ induces 
\[
\mathfrak{M}_t^{m,n} \ : \ 
\cat \uass (m,t) 
\otimes 
\cat \uass (n,t) 
\rightarrow 
\cat \uass (m+n,t)
\]
sending $(f, \ord (f))$ $(g, \ord (g))$ to $(f \amalg g, \ord (f \amalg g ))$ where $f\amalg g : \mathbf{m} \amalg \mathbf{n} \rightarrow \mathbf{t}$ with $(f\amalg g)^{-1} (i) := f^{-1} (i) \ao g^{-1} (i)$ for all $i \in \mathbf{t}$. 

\begin{enumerate}
\item 
The operation $\mathfrak{M}_t^{m,n}$  is $\sym_m\op \times \sym_n\op \times \sym_t$-equivariant, where $\sym_m\op \times \sym_n\op$ acts on $\cat \uass (m+n, t)$ via restriction along $\sym_m \times \sym_n \subset \sym_{m+n}$ (induced by $\mathbf{m} \ao \mathbf{n} \cong \mathbf{m+n}$), and the action of $\sym_t$ on the domain is the diagonal action.
\item 
The operation  is associative and unital: for $m,n,p \in \nat$, the two composites 
 \[
\cat \uass (m,t) 
\otimes 
\cat \uass (n,t) 
\otimes 
\cat \uass (p,t) 
\rightrightarrows 
\cat \uass (m+n+p,t)
\]
coincide; under the identification $\cat \uass (0,t)=\kring$, the operation with $m=0$ (respectively $n=0$), identifies with the identity map.
\end{enumerate}
\end{lem}

\begin{proof}
The equivariance statement is clear from the construction, as is the unital property. 
Associativity follows from the associativity of the operation $\ao$.
\end{proof}

\begin{rem}
\label{rem:warning_operation}
For fixed $m,n \in \nat$, the operation $ \mathfrak{M}_\bullet^{m,n}$ of Lemma \ref{lem:disj_cat_assu} is not (except in degenerate cases) a natural transformation of functors on $\gr\op$. The problem arises from the fact that  the action of morphisms of $\gr\op$  need not preserve the imposed order between elements of $\mathbf{m}$ and $\mathbf{n}$.
\end{rem}

\begin{proof}[Proof of Theorem \ref{thm:compat_tensor_conv}]
We first define the natural transformation
\begin{eqnarray}
\label{eqn:nat_trans_conv_tensor}
\big( \gropcatuass\otimes_{\cat \lie} F \big)
\otimes
\big( \gropcatuass\otimes_{\cat \lie}  G\big )
\rightarrow 
\gropcatuass\otimes_{\cat \lie} (F \conv G),
\end{eqnarray}
using the operations $\mathfrak{M}_\bullet^{m,n}$ of Lemma \ref{lem:disj_cat_assu}.

Fix $t \in \nat$. By construction, $\big( \gropcatuass\otimes_{\cat \lie} F \big) (t)$ is a quotient of $\bigoplus_{m \in \nat} \gropcatuass (m,t) \otimes_{\sym_m} F(m)$ and similarly for $G$ and $F \conv G$. We first construct the natural transformations:
\[
\big(
\gropcatuass (m,t) \otimes_{\sym_m} F(m)
\big) 
\otimes 
\big(
\gropcatuass (n,t) \otimes_{\sym_n} G(n)
\big) 
\rightarrow 
\big(
\gropcatuass (m+n,t) \otimes_{\sym_{m+n}} F \conv G \  (m+n)
\big)
\] 
for $m, n \in \nat$. 

The domain is naturally isomorphic to $ \big (\gropcatuass (m,t) \boxtimes \gropcatuass (n,t)\big)  \otimes_{\sym_m \times \sym_n } \big( F(m) \boxtimes G(n))$, where $\boxtimes$ is used to stress that the terms are the exterior products of the respective representations. The required natural transformation is then obtained by using the tensor product of $\mathfrak{M}^{m,n}_t$ with the $\sym_m \times \sym_n$-equivariant map $F(m) \boxtimes G(n) \hookrightarrow F \conv G \ (m+n)$ induced by $\mathbf{m} \ao \mathbf{n} \cong \mathbf{m+n}$. This is clearly natural with respect to $F$ and $G$.

The definition of the $\cat \lie$-module structure on $F\conv G$ ensures that these morphisms induce a natural transformation 
 \[
\big(
\gropcatuass (-,t) \otimes_{\cat \lie} F
\big) 
\otimes 
\big(
\gropcatuass (-,t) \otimes_{\cat \lie} G
\big) 
\rightarrow 
\big(
\gropcatuass (-,t) \otimes_{\cat \lie} F \conv G 
\big)
\] 
 that is $\sym_t$-equivariant. 
 
The next step is to show that this defines a natural transformation with respect to $\gr\op$. This relies on a  property established below of the  maps 
\begin{eqnarray}
\label{eqn:conv_cat_lie}
\gropcatuass (m+n, t) \otimes_{\sym_m\times \sym_n} F(m) \boxtimes G(n)
\rightarrow 
\gropcatuass (-,t) \otimes_{\cat \lie} F \conv G 
\end{eqnarray}
induced by the $\sym_m\times \sym_n$-equivariant inclusion  $F(m) \boxtimes G(n) \hookrightarrow F \conv G \ (m+n)$ and the passage to the quotient.  

Given a basis element  $(f : \mathbf{m \amalg n} \rightarrow \mathbf{t} , \ord(f) )$ of $\cat \uass (m+n, t)$ as in Example \ref{exam:assu}, construct the basis element  $ (f : \mathbf{m \amalg n} \rightarrow \mathbf{t} , \widetilde{\ord(f)} )$ where $\widetilde{\ord(f)}$  is the unique order on fibres obtained by the unshuffle that places the elements of $\mathbf{m}$ before those of $\mathbf{n}$. We claim that the map (\ref{eqn:conv_cat_lie}) has the following property: for any $x \in F(m) \boxtimes G(n)$ the images of the elements 
\begin{eqnarray*}
&&(f : \mathbf{m \amalg n} \rightarrow \mathbf{t} , \ord(f) ) \otimes x 
\\
&&(f : \mathbf{m \amalg n} \rightarrow \mathbf{t} , \widetilde{\ord(f)} ) \otimes x 
\end{eqnarray*}
in $\gropcatuass (-,t) \otimes_{\cat \lie} F \conv G $ are equal. 

The claim follows from the properties of the $\cat \lie$-module structure on $F \conv G$. To illustrate the argument, take $m=n=t=1$. As explained in Example \ref{exam:assu}, the vector space $ \cat \uass (2, 1)$ has basis $f_{1<2}$ and $f_{2<1}$, where the suffix indicates the order on the unique fibre and the difference $f_{1<2} - f_{2<1}$ lies in $\cat \lie (2,1) \subset \cat \uass (2,1)$. Now, $\cat \lie (2,1)$ acts trivially on $F (1) \boxtimes G(1) \subset F \conv G \ (2)$, by definition of the $\cat \lie$-module structure on $F \conv G$. From this one deduces the claim in this case. The general case is proved similarly. 

This property allows us to deal with the fact highlighted in Remark \ref{rem:warning_operation} that the operations $\mathfrak{M}^{m,n}_\bullet$ are not natural transformations with respect to $\gr\op$: namely, up to the passage to $\otimes_{\cat \lie}$, one can always {\em unshuffle} the elements in the fibres of basis elements of $\gropcatuass (m+n, t)$ so that the elements of $\mathbf{m}$ (corresponding to contributions from $F$) precede those of $\mathbf{n}$ (corresponding to contributions from $G$). It is then straightforward to show that one obtains a natural transformation (\ref{eqn:nat_trans_conv_tensor}) of functors on $\gr\op$, as required.

It remains to prove that (\ref{eqn:nat_trans_conv_tensor}) is a natural {\em isomorphism}. Both the domain and codomain are exact functors of both $F$ and $G$ and they commute with colimits. Using this, one first reduces to the case where both $F$ and $G$ are {\em finite} $\cat \lie$-modules. Then, using the canonical filtration (provided by Proposition \ref{prop:fopd_colimit}) and the five-lemma, one reduces to the case where both $F$ and $G$ are supported on a single object of $\cat \lie$ (not necessarily the same); one  reduces further to the case $F = \kring [\sym_m]$ and $G= \kring [\sym_n]$ (as in Example \ref{exam:conv_single_support}). 

In the latter case, the natural transformation (\ref{eqn:nat_trans_conv_tensor}) fits into the commutative diagram:
\[
\xymatrix{
\big( \gropcatuass\otimes_{\cat \lie} \kring [\sym_m] \big)
\otimes
\big( \gropcatuass\otimes_{\cat \lie}  \kring [\sym_n] \big )
\ar[r]
\ar[dd]_\cong^{\mathrm{PBW}}
&
\gropcatuass\otimes_{\cat \lie} (\kring [\sym_m] \conv \kring [\sym_n]),
\ar[d]^\cong
\\
&
\gropcatuass\otimes_{\cat \lie} \kring [\sym_{m+n}]
\ar[d]^\cong_{\mathrm{PBW}} 
\\
\hom _{\kring \finset} (\mathbf{m}, \mathbf{t} ) 
\otimes 
\hom _{\kring \finset} (\mathbf{n}, \mathbf{t} ) 
\ar[r]_\cong
&
\hom _{\kring \finset} (\mathbf{m+n}, \mathbf{t} ), 
}
\]
where the vertical isomorphisms labelled $\scriptstyle{\mathrm{PBW}}$ correspond to the Poincaré-Birkhoff-Witt isomorphism of Remark \ref{rem:PBW_induction}  (cf. Example \ref{exam:group_ring}), the remaining vertical isomorphism is given by the  isomorphism of $\cat \lie$-modules $\kring [\sym_m] \conv \kring [\sym_n] \cong \kring [\sym_{m+n}]$, 
 and the bottom horizontal isomorphism corresponds to the fact that $\amalg$ is the coproduct in $\finset$. This establishes the isomorphism in this case (corresponding to that of Example \ref{exam:conv_single_support}) and hence completes the proof.
 \end{proof}

\appendix
\section{The structure of $\gropcatuass$}
\label{sect:catassu}

The purpose of this Section is to make the structure of $\gropcatuass$ that is provided by Proposition \ref{prop:structure_cat_uass} more explicit. The structure of $\cat \uass$ as a $\kring$-linear category was  recalled in Example \ref{exam:assu}. 

In the following, recall that $\fg_n$ denotes the free group on the set $\mathbf{n}$, for $n \in \nat$; $x_i$ denotes the generator corresponding to $i \in \mathbf{n}$. 

\begin{lem}
\label{lem:generating_gr}
As a symmetric monoidal category, $\gr$ is generated by the following homomorphisms:

\begin{enumerate}
\item 
$m_1 :  \fg_1 \rightarrow \fg_0 = \{e \}$;
\item 
 $m_2 : \fg_1 \rightarrow \fg_2$ sending 
the generator of $\fg_1$ to $x_1 x_2$;
\item 
$m_3 : \fg_0= \{e\} \rightarrow \fg_1$;
\item 
$m_4 : \fg_1 \stackrel{(-)^{-1}}{\rightarrow } \fg_1$;
\item 
$m_5 : \fg_2 \rightarrow \fg_1$ given by the fold map given by $x_i \mapsto x_1$, for $i \in \mathbf{2}$.
\end{enumerate}
\end{lem}

\begin{proof}
This can be deduced from  Pirashvili's identification of the PROP associated to bialgebras \cite{MR1928233}. It can also be proved directly, as in \cite{PV}.
\end{proof}

For fixed $s \in \nat$, $\cat \uass (s, -)$ belongs to $\fcatk[\gr\op]$, by Proposition \ref{prop:structure_cat_uass}; equipped with this structure, it is denoted by $\gropcatuass (s,-)$. For a homomorphism $ m : \fg_{t_1} \rightarrow \fg_{t_2}$, the induced morphism $\gropcatuass (s, t_2) \rightarrow \gropcatuass (s,t_1)$ is denoted by $m^*$.

\begin{lem}
\label{lem:cat_uass_basic}
For $s \in \nat$, the homomorphisms given in Lemma \ref{lem:generating_gr} act on $\gropcatuass (s, -)$ as follows:
\begin{enumerate}
\item 
$m_1^* : \gropcatuass (s,0)\rightarrow \gropcatuass (s,1)$ is zero unless $s=0$, when it is the isomorphism induced by the unit of $\uass$;
\item 
$ m_2^* : 
\gropcatuass (s,2) 
\rightarrow 
\gropcatuass (s , 1) 
$ is the $\kring$-linear map that  sends  a generator $(f, \ord (f))$ to the unique map $g : \mathbf{s} \rightarrow \mathbf{1}$ with  $\ord(g)$ given by the ordered concatenation $f^{-1} (1) \amalg f^{-1} (2)$;
\item 
$m_3^* : \gropcatuass (s,1)\rightarrow \gropcatuass (s,0)$ is zero unless $s=0$, when it is the inverse of $m_1^*$;
\item 
$m_4^* : \gropcatuass (s,1)\rightarrow \gropcatuass (s,1)$ sends $(f, \ord(f))$ to $(-1)^s (f, \ord^- (f))$, where $\ord^- (f) $ is the reversal of $\ord (f)$;
\item 
$
m_5^* : \gropcatuass (s,1) 
\rightarrow 
\gropcatuass (s ,2)
$ is the $\kring$-linear map that sends $(f, \ord(f))$, where $\ord(f)$ is equivalent to an order  $(\mathbf{s},<)$, to 
$\sum_{\mathbf{s} = \mathbf{s}_1 \amalg \mathbf{s}_2} (f_{\mathbf{s}_1, \mathbf{s}_2}, \ord(f_{\mathbf{s}_1, \mathbf{s}_2}))$, where $f_{\mathbf{s}_1, \mathbf{s}_2}^{-1}(i) = \mathbf{s}_i$, for $i \in \{1, 2 \}$ with order inherited from $\mathbf{s}$.
\end{enumerate}
\end{lem}

\begin{proof}
This follows from the construction of the action on $\gropcatuass (s, -)$ given by Proposition \ref{prop:structure_cat_uass} (compare Proposition \ref{prop:cocomm_Hopf} and its proof).
\end{proof}

\begin{rem}
The morphisms $m_1^*$, $m_2^*$, $m_3^*$ are induced by composition with morphisms of $\cat \uass (0,1)$, $\cat \uass (2,1)$, $\cat \uass (1,0)$ respectively (more precisely, these arise from the underlying set operad of $\uass$). Hence the significant new part of the structure is the `conjugation' $m_4^*$ together with the `shuffle coproduct' $m_5^*$.
\end{rem}

The explicit identification given in Lemma \ref{lem:cat_uass_basic} extends readily  to describe the full structure of $\gropcatuass (s,-)$. 

\begin{exam}
\label{exam:p_i}
For $t \in \nat$ and $1 \leq i \leq t+1$, let 
$
p_i : \fg_{t+1} \twoheadrightarrow \fg_{t}
$ 
be the projection given by applying $m_1 : \fg_1 \rightarrow \fg_0 = \{e \}$ to the $i$th factor of $\fg_{t+1} \cong \fg_1^{\star t+1}$.

The associated $\kring$-linear map
\[
p_i^* : 
\cat \uass (s,t) 
\rightarrow 
\cat \uass (s , t+1) 
\]
is given by composition with the map $\sigma_i : \mathbf{t} \hookrightarrow \mathbf{t+1}$ defined by $\sigma_i(j) = j$ if $j<i$ and $j+1$ if $j \geq i$.
\end{exam}

This can be expressed using the  general exponential functor construction  $\Phi$ (see  Remark \ref{rem:structure_Phi}), applied working in  the category of $\fb\op$-modules, considered as a symmetric monoidal category with respect to the convolution product $\conv$ (see Definition \ref{defn:conv}). For this, the basic ingredient is:
 
 \begin{prop}
 \label{prop:cocomm_Hopf}
Considered as a $\fb\op$-module  $\mathbf{s} \mapsto \cat \uass (s,1)$, $\cat \uass (-,1)$ has the structure of a cocommutative Hopf algebra with structure morphisms given by Lemma \ref{lem:cat_uass_basic}. 
 \end{prop}

\begin{proof}
This is proved in \cite{PV} and can be checked directly. As in the proof of Proposition \ref{prop:structure_cat_uass}, a  quick argument working over a field of characteristic zero is to consider the associated Schur functor, which identifies as the tensor Hopf algebra $V \mapsto T(V)$, with concatenation product and shuffle coproduct.
\end{proof}

From the construction of the category associated to an operad (see Remark \ref{rem:cat_opd}), one has:

\begin{lem}
\label{lem:cat_uass_conv}
For $t \in \nat$, there is an isomorphism of $\fb\op$-modules:
\[
\cat \uass (-,1) ^{\conv t} 
\cong 
\cat \uass (-, t).
\]
\end{lem}

Putting these facts together,  one obtains:

\begin{prop}
\label{prop:grop_fbop_cat_uass}
The $\gr\op \times \fb\op$-module structure of $\gropcatuass(-,-)$ is given by the exponential functor $\Phi (\cat \uass (-,1))$ for the Hopf algebra structure of Proposition \ref{prop:cocomm_Hopf}.
\end{prop}

\subsection{An application to cross-effects}

The polynomiality  of functors on $\gr\op$ (see Definition \ref{defn:tre}) is defined in terms of (iterates of) the projections $p_i : \fg_{t+1} \twoheadrightarrow \fg_t$ appearing in Example \ref{exam:p_i}. This means that it is straightforward to analyse the polynomiality of the functors $\gropcatuass (d, -)$ for $d \in \nat$:

\begin{prop}
\label{prop:catuass_poly}
For $d \in \nat$, the functor $\gropcatuass (d, -) \in \ob \fcatk[\gr\op]$ has polynomial degree $d$ and
\[
\gamma_d \big(\gropcatuass (d, -)\big) \cong \kring [\sym_d]
\]
where $\gamma_d$ is the functor introduced in Corollary \ref{cor:gamma}.
\end{prop}

\begin{proof}
The first statement follows from the definition of polynomial degree  and the observation that,  if $(f, \ord (f))$ represents a generator of $\cat \uass (d, t)$ for $t > d$, then there exists $i \in \mathbf{t}$ such that $f^{-1} (i) =\emptyset$. 

The second statement is similar: for $t=d$ and $(f, \ord (f))$  a generator of $\cat \uass (d, d)$, either $f$ is not surjective or $f$ is surjective and each fibre has cardinal one, so that $f$ corresponds to an element of $\sym_d$.
\end{proof}


\providecommand{\bysame}{\leavevmode\hbox to3em{\hrulefill}\thinspace}
\providecommand{\MR}{\relax\ifhmode\unskip\space\fi MR }
\providecommand{\MRhref}[2]{%
  \href{http://www.ams.org/mathscinet-getitem?mr=#1}{#2}
}
\providecommand{\href}[2]{#2}

\end{document}